\documentclass[reqno,11pt]{amsart}
\usepackage{amsmath,amssymb}
\usepackage{color}
\usepackage{enumitem}
\usepackage{booktabs}
\usepackage[nospace]{cite}
\usepackage[pdftex]{hyperref}
\usepackage[utf8]{inputenc}
\usepackage[svgnames]{xcolor}
\usepackage[pdftex]{hyperref}
\usepackage{booktabs}
\usepackage{soul} 
\usepackage{dsfont}

\usepackage[margin={2cm,2cm}]{geometry}

\hypersetup{
  pdftitle={},
  pdfauthor={Fran\c cois Genoud, Simona Rota Nodari},
  pdfsubject={},
  pdfkeywords={}, 
  colorlinks=true,
  linkcolor=Blue  ,          
  citecolor=Purple,        
  filecolor=Magenta,      
  urlcolor=Green,           
}

\newcommand{\N}{\mathbb{N}}
\newcommand{\R}{\mathbb{R}}
\newcommand{\real}{\mathbb{R}}
\newcommand{\rn}{\mathbb{R}^N}
\newcommand{\intrn}{\int_{\mathbb{R}^N}}
\newcommand{\C}{\mathbb{C}}
\newcommand{\eps}{\varepsilon}
\newcommand{\ffi}{\varphi}

\DeclareMathOperator \im{Im}
\DeclareMathOperator \re{Re}

\newcommand{\wto}{\rightharpoonup}

\newcommand{\wt}{\widetilde}
\newcommand{\diff}{\,\mathrm{d}}

\renewcommand\emptyset{\mbox{\Large \o}}
\newcommand{\sm}{\setminus}

\let\le=\leqslant
\let\ge= \geqslant

\newcommand\txt{\textstyle}

\theoremstyle{plain}
\newtheorem{theorem}{Theorem}
\newtheorem{lemma}{Lemma}
\newtheorem{proposition}{Proposition}
\newtheorem{corollary}{Corollary}
\theoremstyle{definition}

\newtheorem{remark}{Remark}

\numberwithin{equation}{section}
\numberwithin{theorem}{section}
\numberwithin{lemma}{section}
\numberwithin{proposition}{section}
\numberwithin{corollary}{section}

\title[Standing wave solutions of a quasilinear Schrödinger equation]
{Standing wave solutions of a quasilinear Schrödinger equation in the small frequency limit}
\author{Fran\c cois Genoud}
\address{Cours de Mathématiques Spéciales, Ecole Polytechnique Fédérale de Lausanne,
\newline\indent
EPFL Station 4, 1015 Lausanne, Switzerland}
\email{francois.genoud@epfl.ch}
\author{Simona Rota Nodari}
\address{Laboratoire J.A. Dieudonné, Université Côte d'Azur, CNRS UMR 7351, 
\newline\indent
Parc Valrose,
28, avenue Valrose, 06108 Nice Cedex 2, France}
\email{simona.rotanodari@univ-cotedazur.fr}

\begin{document}
\begin{abstract}
This article is concerned with the quasilinear Schr\"odinger equation
\[
\Delta u-\omega u+|u|^{p-1}u+\delta\Delta(|u|^2)u=0,
\]
where $\delta>0$, $N=2$ and $p>1$ or $N\ge3$ and $1<p<\frac{3N+2}{N-2}$.
After proving uniqueness and non-degeneracy of the positive solution $u_\omega$
for all $\omega>0$, our main results establish the asymptotic behavior of $u_\omega$ 
in the limit $\omega\to 0^+$. Three different regimes arise, termed `subcritical',
`critical' and `supercritical', corresponding respectively (when $N\ge3$) to
$1<p<\frac{N+2}{N-2}$, $p=\frac{N+2}{N-2}$ and $\frac{N+2}{N-2}<p<\frac{3N+2}{N-2}$.
In each case a limit equation is exhibited which governs, in a suitable scaling, 
the behavior of $u_\omega$ in the limit $\omega\to 0^+$. 
The critical case is the most challenging, technically speaking. In this case, the limit equation is
the famous Lane-Emden-Fowler equation. A substantial part of our efforts is
dedicated to the study of the function $\omega\mapsto M(\omega)=\intrn u_\omega^2$. 
We find that, for small $\omega>0$, $M(\omega)$ is increasing if $1<p\le 1+\frac4N$ and
decreasing if $1+\frac4N< p\le\frac{N+2}{N-2}$. In the supercritical case, the monotonicity of
$M(\omega)$ depends on the dimension, except in the regime $p\ge 3+\frac4N$, where $M(\omega)$
is always decreasing close to $\omega=0$. The crucial role played by $M(\omega)$ for the
orbital stability of the standing wave $e^{i\omega t}u_\omega$, 
and for the uniqueness of normalized ground states, is discussed in the 
introduction.
\end{abstract}

\maketitle

\section{Introduction}
This work is concerned with the quasilinear Schr\"odinger equation
\begin{equation}
\label{eq:time-dep}
i\partial_t \varphi
+\Delta \varphi+|\varphi|^{p-1}\varphi+\delta\Delta(|\varphi|^2)\varphi=0,
\end{equation}
and the associated stationary equation satisfied by standing waves 
$\varphi(t,x)=e^{i\omega t}u(x)$,
\begin{equation}\label{eq:E_omega}
\Delta u-\omega u+|u|^{p-1}u+\delta\Delta(|u|^2)u=0 \tag{$\mathrm{E}_\omega$}.
\end{equation}
We consider the problem in $\rn$ with $N\ge3$, $\omega>0$ and $\delta>0$.

This equation belongs to a family of quasilinear Schr\"odinger equations of the form
\begin{equation}
    \label{eq:quasilineargen}
    i\partial_t \varphi
+\Delta \varphi+f(|\varphi|^{2})\varphi+\varphi\ell'(|\varphi|^2)\Delta\ell(|\varphi|^2)=0,
\end{equation}
which appear in several physical situations (see \cite{BruLan-86} and references therein). Here $f$ and $\ell$ are given smooth functions. When $f(s)=s^{(p-1)/2}$ and $\ell(s)=\sqrt{\delta}s$, we obtain~\eqref{eq:E_omega} which is relevant in various problems in plasma physics and nonlinear optics (see \cite{PorGol-76, LaeSpaSte-83, Goldman-84}). The local and global well-posedness of the Cauchy problem associated to~\eqref{eq:quasilineargen} have been studied by Poppenberg in~\cite{Poppenberg-01}
for smooth initial data (belonging to the space $H^\infty$). In~\cite{Colin-02, CJS_2010}, the authors improved the local well-posedness result for initial data in $H^{s+2}(\rn)$ for $s=2E(N/2)+2$ with $E(N/2)$ the integer part of $N/2$. More precisely, in \cite{Colin-02}, equation~\eqref{eq:quasilineargen} is solved locally for smooth nonlinearities $\ell$ and $f$ such that there exists a positive constant $C_\ell$ with $1-4\sigma \ell'^2(\sigma)>C_\ell \ell'^2(\sigma)$ for any $\sigma\in \R_+$, while~\cite{CJS_2010} deals with the case $\ell(\sigma)=\sigma$ and $f\in C^{s+2}(\R_+)$.

The parameter $\delta>0$ is a coupling constant relevant to describe the strength of the 
quasilinear interaction in the associated physical models. In several works about 
\eqref{eq:E_omega}, it is simply set equal to one. In the series of papers
\cite{adachi-watanabe_2014,AdachiWatanabe_subcrit,AdachiWatanabe_crit,AdachiWatanabe_supercrit}, the asymptotic behavior of solutions of \eqref{eq:E_omega} 
is studied in the limit $\delta\to0^+$. It is worth noting here that a 
simple scaling transforms \eqref{eq:E_omega} with $\omega=1$ into
\[
\Delta u-\delta^\frac{p-1}{2} u+|u|^{p-1}u+\Delta(|u|^2)u=0.
\]
Some of the results obtained in \cite{adachi-watanabe_2014,AdachiWatanabe_subcrit,AdachiWatanabe_crit,AdachiWatanabe_supercrit} 
can thus be recovered from the present approach, which we believe to be much more straightforward.
Furthermore, as explained below, our main focus in this work is the asymptotic behavior
of the function $M(\omega)$ defined in \eqref{eq:def_of_mass}, which is not addressed
in detail in \cite{adachi-watanabe_2014,AdachiWatanabe_subcrit,AdachiWatanabe_crit,AdachiWatanabe_supercrit}.

Concerning the existence of solutions to \eqref{eq:E_omega}, the following result has been proved by Colin, Jeanjean and Squassina \cite{CJS_2010}. 

\begin{theorem}\label{thm:existence}
Let $N=2$ and $p>1$, or $N\ge 3$ and $1<p<\frac{3N+2}{N-2}$.
For any $\omega>0$, there exists a solution 
$u_\omega\in H^1(\rn)\cap C^2(\rn)$ of \eqref{eq:E_omega}. Furthermore, 
$u_\omega$ is positive, spherically symmetric, radially decreasing and 
decays exponentially at infinity, together with
its derivatives up to second order.
\end{theorem}

\begin{remark}
    As pointed out in \cite{CJS_2010} (see also \cite{Sel-10}), a bootstrap argument allows one 
    to show that $u_\omega$ belongs to $\cap_{s>0} H^s(\R^N)$ and, in particular, is of class $C^\infty$.
\end{remark}



The main idea used to prove Theorem~\ref{thm:existence} is to remark that finding a solution $u$ of the equation~\eqref{eq:E_omega} is equivalent to finding a solution $v$ of the semilinear equation 
\begin{equation}\label{eq:semi_intro}
    -\Delta v=\frac{1}{\sqrt{1+2\delta r(v)^2}}\left(|r(v)|^{p-1}r(v)-\omega r(v)\right),
\end{equation}
with $r$ a suitably chosen function (see Section~\ref{sec:quasitosemi} for more details). Variational arguments are then used to deduce an existence result for~\eqref{eq:semi_intro}. Finally, a solution to~\eqref{eq:E_omega} is obtained by setting $u=r(v)$.

Note that that for $N\ge 3$, $p<\frac{3N+2}{N-2}$ is a necessary condition for the existence of nontrivial solutions $u\in H^{1}(\rn)\cap L^{\infty}(\rn)$. This is a consequence of the 
following integral identities, which will play an important role in our analysis.
They can be proved as sketched in \cite{CJS_2010}. Let
\begin{equation}
X:=\Big\{u\in H^1(\rn,\C): \intrn |u|^2|\nabla|u||^2\diff x<\infty\Big\}
\end{equation}
and
\begin{equation}
\wt X:=\Big\{u\in \dot H^1(\rn,\C): \intrn |u|^2|\nabla|u||^2\diff x<\infty\Big\}.
\end{equation}

\begin{proposition}\label{prop:integral_identities}
Let $N\ge3$. Any solution $u \in X$ of \eqref{eq:E_omega} satisfies
\begin{equation}\label{eq:pohozaev}
\frac{1}{2^*}\intrn|\nabla u|^2 \diff x 
+ \frac{2}{2^*} \delta  \intrn |u|^2|\nabla |u||^2\diff x
=\frac{1}{p+1}\intrn|u|^{p+1}\diff x-\frac{\omega}{2}\intrn|u|^{2}\diff x
\end{equation}
and
\begin{equation}\label{eq:nehari}
\frac12\intrn|\nabla u|^2 \diff x + 2\delta \intrn |u|^2|\nabla |u||^2\diff x
=\frac{1}{2}\intrn|u|^{p+1}\diff x-\frac{\omega}{2}\intrn|u|^{2}\diff x.
\end{equation}
The above identities still hold true for a solution $u\in\wt X$ of \eqref{eq:E_omega}, 
in case $\omega=0$.
\end{proposition}

As already noted above, it follows from Proposition~\ref{prop:integral_identities} that, in dimensions $N\ge 3$,
$p<\frac{3N+2}{N-2}$ is a necessary condition for the existence of nontrivial solutions $u\in X$.

The main goal of this paper is to study the qualitative properties of a branch of solutions of \eqref{eq:E_omega}, parametrized by $\omega$. In particular, we want to investigate the \emph{uniqueness} of positive solutions and their \emph{non-degeneracy}, \emph{i.e.}~the fact that the kernel of the linearized operators is trivial, modulo phase and space translations.

The study of uniqueness of positive solutions to semilinear Schr\"odinger equations of the form
\begin{equation*}
    \Delta u +g(u)=0,
\end{equation*}
with $g$ a well-behaved nonlinearity, has a very long history, see \emph{e.g.}~\cite{Coffman-72, Kwong-89,  CheLin-91, Jang-10, McLeod-93, LeoSer-87, PelSer-83, PucSer-98, SerTan-00, LewRot-20}. Quasilinear equations have attracted less attention, and fewer results on the uniqueness or non-degeneracy of positive solutions of equations of this type are available, see 
\emph{e.g.}~\cite{Sel-10, GlaSqu-12, AdaShiWat-18, LewRot-15}.

A first important result that will be proved here is the uniqueness and non-degeneracy of the positive solution $u_\omega$ to~\eqref{eq:E_omega}, for any $\omega>0$.

\begin{theorem}[Uniqueness and non-degeneracy]\label{thm:uniqueness} Let $N=2$ and $p>1$, or $N\ge 3$ and $1<p<\frac{3N+2}{N-2}$.
For any $\omega>0$, the positive solution $u_\omega$ to the nonlinear equation \eqref{eq:E_omega} is unique, modulo space translation. Moreover, it is non-degenerate:
\begin{equation}
    \label{eq:nondegeneracythm}
    \left\{
    \begin{aligned}
        &\ker(L_+)=\mathrm{span}\{\partial_{x_1} u_{\omega},\ldots,\partial_{x_N}u_\omega\},\\
        &\ker(L_-)=\mathrm{span}\{u_\omega\},
    \end{aligned}
    \right.
\end{equation}
where the linear operators $L_+$ and $L_-$ are defined by
\begin{align}
    \label{eq:defLplus}
    &L_+=-(1+2\delta u_\omega^2)\Delta-4\delta u_\omega\nabla u_\omega\cdot\nabla -\delta(4u_\omega \Delta u_\omega +2 |\nabla u_\omega|^2) -pu_\omega^{p-1}+\omega,\\
    \label{eq:defLminus}
    &L_-=-\Delta -\delta(2u_\omega \Delta u_\omega +2 |\nabla u_\omega|^2)-u_\omega^{p-1}+\omega.
\end{align}
\end{theorem}

As a consequence of the non-degeneracy of $u_\omega$, 
the following proposition will also be established.

\begin{proposition}
    \label{prop:regularityfamsol} Let $\omega>0$ and $u_\omega$ be the unique positive solution of \eqref{eq:E_omega}. Then the map $\omega\mapsto u_\omega$ 
    is of class $C^1((0,\infty),H^1(\R^N))$. 
\end{proposition}

The non-degeneracy property is also crucial for the study of {\em the mass}
\begin{equation}\label{eq:def_of_mass}
    M(\omega)=\intrn u_\omega(x)^2\diff x
\end{equation}
of the unique positive solution $u_\omega$. An important motivation for studying this quantity is the central role played by the function $\omega \mapsto M(\omega)$ in the Grillakis-Shatah-Strauss theory of stability \cite{Weinstein-85, GriShaStr-87, GriShaStr-90, BieGenRot-15, BieRot-19}, which can be applied to the standing wave solutions $e^{i\omega t}u_\omega(x)$ of the time-dependent equation \eqref{eq:time-dep}. In particular, the standing wave is expected to be \emph{orbitally stable} when $M'(\omega)>0$ and \emph{unstable} when $M'(\omega)<0$. Therefore, it is important to be able to identify the regions where $M$ is increasing, which correspond to stable solutions, and those where $M$ is decreasing, corresponding to unstable solutions. For the classical nonlinear Schr\"odinger equation (NLS) with a single-power nonlinearity, which corresponds to~\eqref{eq:E_omega} with $\delta=0$, $N=2$ and $p>1$, or $N\ge 3$ and $p<\frac{N+2}{N-2}$, the mass is an explicit function of $\omega$. Indeed, for $\delta=0$, the solution $u_\omega$ for $\omega>0$ can be obtained from the unique positive solution to $\Delta Q-Q+Q^p=0$ by the simple scaling $u_\omega(x)= \omega^{\frac{1}{p-1}}Q(\sqrt{\omega}x)$. This leads to 
\begin{equation*}
    M_{\mathrm{NLS}}(\omega)=\omega^{\frac{4+N-Np}{2(p-1)}}\intrn Q(x)^2\,\diff x.
\end{equation*}
However, when $\delta>0$, the presence of the quasilinear term prevents one from using a scaling argument and any hope of 
obtaining a simple expression for $M(\omega)$ vanishes. A similar situation occurs in the case of the double-power nonlinearity considered in \cite{LewRot-20}.
Those two problems seem rather different at first sight, but they share a common feature, in the form of an extra term living at another spatial scale.

Hence, in the same spirit as the works of Lewin, Rota Nodari \cite{LewRot-20} and Moroz, Muratov \cite{MorMur-14}, our main theorem
regarding \eqref{eq:E_omega} establishes the behavior of $u_\omega$ and $M(\omega)$ as $\omega\to 0^+$.

\begin{theorem}\label{thm:asympt_omega_to_0}
(i) Suppose $N=2$ and $p>1$, or $N\ge 3$ and $1<p<\frac{N+2}{N-2}$. 
Then, as $\omega\to 0^+$, the rescaled function
\begin{equation}\label{eq:rescaledsub}
\frac{1}{\omega^{\frac{1}{p-1}}}u_\omega\left(\frac{x}{\sqrt{\omega}}\right)
\end{equation}
converges in $H^1(\rn)\cap C^2(\rn)$ to the ground state 
$Q$ of the stationary nonlinear Schr\"odinger equation
\begin{equation}\label{eq:NLSpowersub}
\Delta Q - Q + |Q|^{p-1}Q = 0.
\end{equation}
Furthermore, as $\omega\to 0^+$,
\begin{align}\label{eq:subcrit_formula_M_omega}
M(\omega)
&=\omega^\frac{4-N(p-1)}{2(p-1)}\intrn Q^2\diff x \notag \\
&+\omega^\frac{8-N(p-1)}{2(p-1)}\frac{2(p-1)+8-N(p-1)}{4(p-1)}
\delta\intrn |\nabla Q^2|^2\diff x 
+o\Big(\omega^\frac{8-N(p-1)}{2(p-1)}\Big).
\end{align}
In a neighborhood of $\omega=0$,
$\omega \mapsto M(\omega)$ is increasing if $p\le 1+\frac{4}{N}$
and decreasing if $p> 1+\frac{4}{N}$.

\medskip
(ii) Suppose $N\ge 3$ and $p=\frac{N+2}{N-2}$. Then
there exists a function $\omega\mapsto\lambda_\omega:(0,\infty)\to(0,\infty)$ such that, as $\omega\to 0^+$, 
$\lambda_\omega\to\infty$ and the rescaled function
\begin{equation}\label{eq:rescaledcrit}
\lambda_\omega^{\frac{N-2}{2}}u_\omega(\lambda_\omega x)
\end{equation}
converges in $\dot H^1(\rn)\cap C^2(\rn)$ to the function
\begin{equation}\label{eq:aubin_talenti_fct}
U(x)=\left(1+\frac{|x|^2}{N(N-2)}\right)^{-\frac{N-2}{2}},
\end{equation}
which is the (up to dilations) unique positive radial-decreasing solution
of the Lane-Emden-Fowler equation 
\begin{equation}\label{eq:LEM}
\Delta U + |U|^{\frac{4}{N-2}}U=0. 
\end{equation}
The scaling function $\omega\mapsto\lambda_\omega$ can be chosen so that, as $\omega\to 0^+$, 
\begin{equation}\label{eq:lower-upper_bounds}
\begin{cases}
\omega^{-\frac14}\lesssim \lambda_\omega \lesssim \omega^{-\frac14} & \text{if } N=3, \\
\left(\omega\log\frac{1}{\omega}\right)^{-1/4} \lesssim \lambda_\omega \lesssim \left(\omega\log\frac{1}{\omega}\right)^{-1/4}  & \text{if } N=4, \\
\omega^{-\frac{1}{N}}\lesssim \lambda_\omega \lesssim \omega^{-\frac 1N} & \text{if }  N\ge 5.
\end{cases}
\end{equation}
Furthermore,
\begin{equation}
\lim_{\omega\to0^+}M(\omega)=-\lim_{\omega\to0^+}M'(\omega)=+\infty.
\end{equation}
In a neighborhood of $\omega=0$,
$\omega \mapsto M(\omega)$ is decreasing.

\medskip
(iii) Suppose $N\ge 3$ and $\frac{N+2}{N-2}<p<\frac{3N+2}{N-2}$. 
Then, as $\omega\to 0^+$, $u_\omega$ converges in $\dot H^1(\rn)\cap C^2(\rn)$
to the unique positive radial-decreasing solution 
$u_0\in \dot H^1(\rn)\cap L^{p+1}(\rn)$ of the equation
\begin{equation}\label{eq:E_0}
\Delta u+|u|^{p-1}u+\delta\Delta(|u|^2)u=0.
\end{equation}
Furthermore, $u_0(x)=O(|x|^{-(N-2)})$ as $|x|\to\infty$.
If $N\ge 5$, we have $u_0\in L^2(\rn)$ and $u_\omega\to u_0$ in $L^2(\rn)$
as $\omega\to 0^+$. \\
Finally,
\begin{equation}\label{eq:supercrit_M_omega}
\lim_{\omega\to0^+}M(\omega)=
\begin{cases}
+\infty & \text{if} \ N\in\{3,4\}, \\
\|u_0\|_{L^2}^2 & \text{if} \ N\ge 5,
\end{cases}
\end{equation}
and
\begin{equation}\label{eq:supercrit_M'_omega}
\lim_{\omega\to0^+}M'(\omega)=
\begin{cases}
-\infty & \text{if} \ N\in\{3,4,5,6\}, \\
M'(0)\in\R & \text{if} \ N\ge 7.
\end{cases}
\end{equation}
In a neighborhood of $\omega=0$,
$\omega \mapsto M(\omega)$ is decreasing in dimension $N\in\{3,4,5,6\}$. In dimensions $N\ge 7$, $\omega \mapsto M(\omega)$ is decreasing for $p\ge 3+\frac{4}{N}$.

\end{theorem}

\begin{remark}
In the subcritical case (i), the classical NLS scaling \eqref{eq:rescaledsub} 
kills the quasilinear term in the limit $\omega\to0$, 
yielding the single-power NLS \eqref{eq:NLSpowersub} as limit equation. As in the case of the NLS with a double-power nonlinearity \cite{LewRot-20}, it seems natural to use an implicit function argument to recover the branch of solutions $u_\omega$ starting from the ground state $Q$ of~\eqref{eq:NLSpowersub}. In the case $\delta>0$, the presence of the quasilinear term 
in \eqref{eq:E_omega} requires a clever choice of functional framework.

In the supercritical case (iii), the limit equation is formally obtained 
by letting $\omega=0$ in \eqref{eq:E_omega}, which yields \eqref{eq:E_0}.
However, \eqref{eq:E_0} has no nontrivial $\dot H^1$ solutions if $p\le \frac{N+2}{N-2}$,
as can be seen by combining \eqref{eq:nehari} and \eqref{eq:pohozaev} (see \eqref{eq:nehari_pohozaev}).
In the critical case (ii),
revealing the asymptotic behavior of $u_\omega$ as $\omega\to0$ is more delicate.
Our approach follows that laid down in \cite{MorMur-14}, based on the concentration-compactness
principle. It shows that, in the rescaled variable \eqref{eq:rescaledcrit}, only two terms 
survive as $\omega\to0$, thus yielding the limit equation \eqref{eq:LEM}. 
As in \cite{MorMur-14}, both in the critical and the supercritical cases, we will take advantage of the variational characterization of solutions of the auxiliary semilinear problem \eqref{eq:semi_intro} 
to deduce the desired convergence of $u_\omega$. 
The main difficulty is thus to deal with the function $r$ whose explicit expression is not known.
\end{remark}

\begin{remark}
    In the supercritical case, as for the double-power nonlinearity \cite{LewRot-20}, although we are able to prove that $M'$ admits a finite limit when $\omega\to 0$ in dimension $N\ge 7$, we cannot determine its sign in the full range of parameters. However, if 
    \begin{equation*}
            \frac{N+2}{N-2}<
            p<2+\frac{4}{N}-\sqrt{1-16\frac{N+2}{N^2(N-2)}}\ \text{ or }\ 
            2+\frac{4}{N}+\sqrt{1-16\frac{N+2}{N^2(N-2)}}<p
            <\frac{3N+2}{N-2},
    \end{equation*}
    then $M'(\omega)<0$ for $\omega$ small enough, see Proposition~\ref{prop:upper_bound_M'} below. This condition is probably not optimal but it allows us to conclude that $\omega \mapsto M(\omega)$ is a decreasing function in a neighborhood of $\omega=0$ whenever $p\ge 3+\frac{4}{N}$, for any $N\ge 7$. 
\end{remark}

As a consequence of Theorem~\ref{thm:asympt_omega_to_0}, for any $N\ge 2$ and $3+\frac{4}{N}<p<\frac{3N+2}{N-2}$ ($ 3+\frac{4}{N}<p<\infty$ if $N=2$), the positive solution $u_\omega$, for $\omega$ small enough, lies on an unstable branch of solutions, which is in agreement with the result of instability by blow-up obtained by Colin, Jeanjean, Squassina~\cite[Theorem 1.5]{CJS_2010}. 
\footnote{For simplicity, we speak of the stability of $u_\omega$ when actually referring to the
orbital stability of the associated standing wave $e^{i\omega t}u_\omega$ of \eqref{eq:time-dep}.}

In~\cite{CJS_2010}, it has been conjectured that whenever $1<p<3+\frac{4}{N}$, the positive solution to~\eqref{eq:E_omega} is orbitally stable. On the one hand, our analysis for $\omega$ close to $0$ confirms this conjecture for $1<p\le 1+\frac{4}{N}$ since $\omega \mapsto M(\omega)$ is an increasing function in a neighborhood of $0$ in this case. On the other hand, when $ 1+\frac{4}{N}<p< 3+\frac{4}{N}$, the function $\omega \mapsto M(\omega)$ is decreasing in a neighborhood of $0$, at least in dimension $N\in \{2,3,4,5,6\}$. This implies that $u_\omega$, the positive solution of~\eqref{eq:E_omega}, should be unstable at least for $\omega$ small enough in the latter case.

The original conjecture from Colin, Jeanjean and Squassina was supported by the results obtained for the so-called \emph{normalized solutions}, \emph{i.e.}~solutions of~\eqref{eq:E_omega} with a prescribed $L^2$-norm. In particular, for a given $\lambda>0$, one can look for solutions of
\begin{equation}
    \label{eq:min_intro}
    I(\lambda)=\inf\left\{\mathcal E(u) : u\in X, \intrn |u|^2\,\diff x=\lambda\right\},
\end{equation}
where $\mathcal E$ is the energy functional defined by
\begin{equation}
    \label{eq:energy_intro}
    \mathcal E(u)=\frac{1}{2}\intrn |\nabla u|^2\diff x + \delta\intrn |u|^2|\nabla u|^2\, \diff x-\frac{1}{p+1}\intrn |u|^{p+1}\,\diff x,
\end{equation}
with $\delta>0$.
To each minimizer $u$ of~\eqref{eq:min_intro} corresponds a Lagrange multiplier $\omega>0$ such that $u=u_\omega$ (after an appropriate space translation), where $u_\omega$ is the unique positive solution to~\eqref{eq:E_omega}. 

The minimization problem~\eqref{eq:min_intro} has been studied extensively over the last decade (see \cite{CJS_2010,JeaLuo-13, JeaTinWan-15, YeYu-21}). The known results can be summarized as follows:
\begin{enumerate}
    \item For all $\lambda>0$, $I(\lambda)\in (-\infty,0]$ if $1<p<3+\frac{4}{N}$ and $I(\lambda)=-\infty$ if $p>3+\frac{4}{N}$.
    \item If $1<p<1+\frac{4}{N}$, then for all $\lambda>0$, $I(\lambda)<0$ and $I(\lambda)$ admits a minimizer.
    \item If $1+\frac{4}{N}\le p <3+\frac{4}{N}$, there exists $\lambda_c\in (0,\infty)$ such that
    $I(\lambda)=0$ for all $0<\lambda\le \lambda_c$ and $\lambda\mapsto I(\lambda)$ is negative and strictly decreasing on $(\lambda_c,\infty)$. Moreover, if $1+\frac{4}{N}< p <3+\frac{4}{N}$
    then $I(\lambda)$ admits a minimizer if and only if $\lambda \in [\lambda_c,\infty)$. If $p=1+\frac{4}{N}$, then $I(\lambda)$ admits a minimizer if and only if $\lambda \in (\lambda_c,\infty)$.
    \item If $p=3+\frac{4}{N}$, then either $I(\lambda)=0$ or $I(\lambda)=-\infty$. As a consequence, $I(\lambda)$ has no minimizers for all $\lambda\in (0,\infty)$.
    \item The set of minimizers of~\eqref{eq:min_intro}, when it is not empty, is orbitally stable.
\end{enumerate}

Unfortunately, the orbital stability of the set of minimizers does not allow one to deduce the orbital stability of a single solution $u_\omega$ of~\eqref{eq:E_omega} for a fixed $\omega>0$. On the one hand, for fixed $\omega>0$, one should first determine whether a solution $u_\omega$ of \eqref{eq:E_omega} is also a solution of the minimization problem \eqref{eq:min_intro} with $\lambda=M(\omega)$. Thanks to the non-degeneracy of $u_\omega$ and the spectral properties of $L_+$, we know from \cite[App.~E]{Weinstein-85} and \cite[Theorem 5.3.2]{KapPro-13} (see also \cite{Maddocks-85,Maddocks-88}) that, $u_\omega$ is a strict local minimum of $\mathcal E$ at fixed mass $\lambda=M(\omega)$, if $M'(\omega)>0$, whereas the solution $u_\omega$ is not a local minimum when $M'(\omega)<0$. 
As a consequence, a solution $u_\omega$ lying on a decreasing branch of the function $M$ cannot be a
solution of~\eqref{eq:min_intro}, so its stability cannot be deduced from the orbital stability of the set of minimizers.
On the other hand, the set of minimizers may contain more than one solution. 
More precisely, the uniqueness of positive solutions to~\eqref{eq:E_omega} at fixed $\omega$ does not imply the uniqueness of energy minimizers. As already mentioned, any minimizer, when it exists, is positive and solves~\eqref{eq:E_omega} for some Lagrange multiplier $\omega>0$. The difficulty here is that the Lagrange multiplier is \emph{a priori} not uniquely determined : for a given mass $\lambda>0$, there can be several minimizers that share the same energy $I(\lambda)$ but give rise to different Lagrange multipliers $\omega$. In other words, there may not be a one-to-one mapping $\lambda\mapsto\omega(\lambda)$.
Nevertheless, any candidate to be a Lagrange multiplier in this situation must be a solution to the equation $M(\omega)=\lambda$, hence the importance of studying the behavior of the function $M$ and its variations. In particular, if one is able to find a region of $\lambda$'s where the function $M$ is one-to-one, then the uniqueness of energy minimizers follows for such $\lambda$'s.

To conclude, we mention the following result, that is a corollary of Theorem~\ref{thm:asympt_omega_to_0}.

\begin{corollary}\label{cor:energy_omega}
    Let $N=2$ and $p>1$, or $N\ge 3$ and $1<p<\frac{3N+2}{N-2}$. For any $\omega>0$, let $u_\omega$ be the unique positive solution to~\eqref{eq:E_omega} and define $E(\omega):=\mathcal E(u_\omega)$.  
    \begin{enumerate}
        \item $E'(\omega)=-\frac{\omega}{2}M'(\omega)$, for all $\omega>0$.
        \item 
        \begin{equation*}
            \lim_{\omega\to 0^+} E(\omega)=\begin{cases}
                0 & \text{ if } N=2 \text{ or } N\ge 3 \text{ and } p<\frac{N+2}{N-2},\\
                \frac{1}{N}\intrn |\nabla U|^2\, \diff x & \text{ if } N\ge 3 \text{ and } p=\frac{N+2}{N-2},\\
                \frac{2}{(3N+2)-p(N-2)}\intrn |\nabla u_0|^2\, \diff x & \text{ if } N\ge 3 \text{ and } p>\frac{N+2}{N-2},
            \end{cases}
        \end{equation*}
        where $U$ is the unique positive solution to~\eqref{eq:LEM} and $u_0$ is the unique positive solution to~\eqref{eq:E_0}.
        \item In a neighborhood of $\omega=0$, $E(\omega)<0$ for $p\ge 1+\frac{4}{N}$ and $E(\omega)>0$ for $p>1+\frac{4}{N}$.
    \end{enumerate}
\end{corollary}

Thanks to this result, we can conclude that if $1+\frac{4}{N}<p<3+\frac{4}{N}$, then the solution $u_\omega$, for $\omega$ close to $0$, cannot be a global minimizer of $I(\lambda)$ since its energy is strictly positive. Thus, we cannot deduce its stability from the orbital stability of the set of minimizers. In fact, our conjecture is that $u_\omega$, for $\omega$ small, is unstable for any $p>1+\frac{4}{N}$ (at least in dimensions $N\le 6$).

All the previous remarks emphasize the importance of studying the function $M$ and, in particular, the number of sign changes of $M'$. Nevertheless, getting a complete picture on the behavior of $M$ seems out of reach for the moment. As in~\cite{LewRot-20}, a first step will be to understand the exact behavior of $M(\omega)$ at the two endpoints of its interval of definition : the current paper was devoted to the limit $\omega\to 0^+$, while the limit $\omega\to \infty$ will be discussed in a follow-up paper currently in preparation.

\subsection*{Organization of the paper}
Section~\ref{sec:quasitosemi} is devoted to basic properties of the change of 
variables $r$ and of the auxiliary semilinear equation \eqref{eq:semi_intro}.
In Section~\ref{sec:uniqueness}, the uniqueness and non-degeneracy stated in Theorem~\ref{thm:uniqueness} are proved, as well as
Proposition~\ref{prop:regularityfamsol}. Finally,
the proof of Theorem~\ref{thm:asympt_omega_to_0} is given 
in Section~\ref{sec:limitomegazero}, where it is split into three subsections,
dealing with cases (i), (ii) and (iii), respectively.


\section{Reformulation of the problem: from a quasilinear to a semilinear equation}\label{sec:quasitosemi}
To prove Theorem~\ref{thm:uniqueness}, and Theorem~\ref{thm:asympt_omega_to_0}~(ii) and (iii),
we will use a change of variables $v=h(u)$ borrowed from \cite{liu-wang-wang},
given by
\begin{equation*}
h(t):=\frac{1}{2\sqrt{2}\delta}\left(\ln\left(\sqrt{2}\delta t
+\sqrt{1+2\delta t^2}\right)
+\sqrt{2}\delta t\sqrt{1+2\delta t^2}\right), \quad t\ge0.
\end{equation*}
Since $h'(t)=\sqrt{1+2\delta t^2}\ge0$ for all $t\ge0$, 
it follows that $h$ has an inverse function
$r:[0,\infty)\to[0,\infty)$, which is $C^1$ on $(0,\infty)$ 
and satisfies the first order Cauchy problem
\begin{equation}\label{eq:ODE_r}
r'(s)=\frac{1}{\sqrt{1+2\delta r(s)^2}}, \quad s\in(0,\infty), \qquad r(0)=0.
\end{equation}
We now extend $r$ to the whole real line by letting $r(s)=-r(-s)$ when $s<0$.

\begin{lemma}\label{propofr.lem}
The odd function $r:\R\to\R$ has the following properties:
\item[(i)] $r\in C^1(\real)$ and $r$ is strictly increasing on $\R$, with $r'(s)>0$ for all $s\in\R$;
\item[(ii)] $|r'(s)|\leq 1$ for all $s\in\R$ and $r$ is a Lipschitz function on $\R$;

\item[(iii)]
$$ 
\lim_{s\to0^+}\frac{r(s)}{s}=1 
\quad\text{and}\quad \lim_{s\to+\infty}\frac{r(s)}{\sqrt{s}}
=\Big(\frac{2}{\delta}\Big)^{1/4}.
$$

\item[(iv)] for all $s\in \R_+$, 
$$
\frac{1}{2}r(s)\le \frac{s}{\sqrt{1+2\delta r^2(s)}}\le r(s).
$$
\end{lemma}

\begin{proof}
(i) and (ii) follow immediately from \eqref{eq:ODE_r}. To prove (iii), 
integrating \eqref{eq:ODE_r} yields
$$
\int_0^s r'(\sigma)\sqrt{1+2\delta r(\sigma)^2}\diff\sigma = s \iff
\sqrt{2\delta}r(s)\sqrt{1+2\delta r(s)^2}
+\mathrm{arcsinh}\big(\sqrt{2\delta}r(s)\big)=2\sqrt{2\delta}s
$$
(whence the formula for $h$ if you rather decided to define $r$ by \eqref{eq:ODE_r}). 
Since, by construction, $r(s)\to0$ as $s\to0^+$ and $r(s)\to\infty$ as $s\to+\infty$,
the limits in (iii) then easily follow from the above relation.
To prove (iv), we consider the function $g_\lambda : \R_+\to \R$ defined by $g_\lambda(s)=\lambda r(s)\sqrt{1+2\delta r^2(s)}-s$ for $\lambda\in\{\tfrac{1}{2},1\}$. Clearly, $g_\lambda(0)=0$ and 
\begin{equation*}
    g'_\lambda(s)=\lambda r'(s)\sqrt{1+2\delta r^2(s)}+\lambda 2\delta r^2(s)(r'(s))^2-1= (\lambda-1)+\lambda \frac{2\delta r^2(s)}{1+2\delta r^2(s)}.
\end{equation*}
As a consequence, we get 
\begin{equation*}
    \left\{\begin{aligned}
        &g_{1/2}'(s)=-\frac{1}{2}+\frac{1}{2}\frac{2\delta r^2(s)}{1+2\delta r^2(s)}\le 0,\\
        &g_{1}'(s)=\frac{2\delta r^2(s)}{1+2\delta r^2(s)}\ge 0,
    \end{aligned}
    \right.
\end{equation*}
which implies  (iv).
\end{proof}

\begin{remark}\label{preservint.rem}
Since $r:\R\to \R$ is a smooth function such that $r(0)=0$ we have (see \cite[Theorem 2.87]{BahCheDan-11}):
$$
\forall s>0, \ v\in H^{s}(\rn)\cap L^{\infty}(\rn) \implies
r(v)\in H^{s}(\rn)\cap L^{\infty}(\rn),
$$
$$
\forall q\ge1, \ v\in L^q(\rn) \implies r(v)\in L^q(\rn).
$$
\end{remark}

For $\omega\ge0$, we now define $f_\omega:\R\to\R$,
\begin{equation}\label{f_omega}
f_\omega(s):=\frac{1}{\sqrt{1+2\delta r(s)^2}}\left(|r(s)|^{p-1}r(s)-\omega r(s)\right),
\end{equation}
and we consider the nonlinear elliptic equation on $\rn$
\begin{equation}\label{S_omega}
-\Delta v=f_\omega(v) \tag{$\mathrm{S}_\omega$}. 
\end{equation}

Note that the nonlinearity $f_\omega$ is of the form $f_\omega(s)=r'(s)P_\omega (r(s))$ with
\begin{equation}
    \label{p_omega}
    P_\omega (\tau)=|\tau|^{p-1}\tau-\omega \tau.
\end{equation}
As a consequence, for any $s\in \R_+$, 
\begin{equation}
    \label{f_omegaderivative}
    f'_\omega (s)=r''(s)P_\omega (r(s))+(r'(s))^2P'_\omega (r(s)).
\end{equation}
Furthermore, it follows from Lemma~\ref{propofr.lem} that there exists a constant $C_0>0$ such that
\begin{equation}\label{eq:upper-bound_on_f}
\forall s\ge0, \quad f(s)\le \frac{r(s)^p}{\sqrt{1+2\delta r(s)^2}}\le C_0 s^\frac{p-1}{2}.    
\end{equation}

The two following lemmas are well-known, see \cite{liu-wang-wang} 
and \cite{BerLio-83}, respectively.

\begin{lemma}\label{lem:equivsols}
$u=r(v)$ is a classical solution of \eqref{eq:E_omega} if and only if $v$ is 
a classical solution of \eqref{S_omega}.
\end{lemma}

\begin{lemma}\label{regularity.lem}
Any weak solution of \eqref{eq:E_omega}
is a classical $C^2$ solution of \eqref{eq:E_omega}. 
The same holds for \eqref{S_omega}.
\end{lemma}

For all $\omega>0$, we attack problem \eqref{S_omega} via a 
standard variational approach.
Let\footnote{The normalization factor $2^*$ in \eqref{m_omega} is introduced for 
coherence with the classical minimization problem for the critical 
Sobolev inequality which will be used in Section~\ref{sec:critical}.}
\begin{equation}\label{m_omega}
m_\omega:=\inf\Big\{\intrn |\nabla z|^2 \diff x : \, z\in H^1(\rn), 
\ 2^*\intrn F_\omega(z)\diff x=1\Big\},
\end{equation}
where, using \eqref{eq:ODE_r},
\begin{equation}\label{F_omega}
F_\omega(s):=\int_0^s f_\omega(\sigma)\diff\sigma=
\frac{1}{p+1}|r(s)|^{p+1}-\frac{\omega}{2}r(s)^2.
\end{equation}

It follows by classical results of Berestycki-Lions \cite[Theorem~2]{BerLio-83}
and Berestycki-Gallou\"et-Kavian \cite{BerGalKav-83}  that,
for all $\omega>0$, the minimization problem \eqref{m_omega} 
has a solution $z_\omega\in H^1(\rn)$ which is spherically symmetric and 
radially nonincreasing. Furthermore, there is a corresponding Lagrange multiplier 
$\theta_\omega>0$ such that
\begin{equation}\label{Ptilde_omega}
-\Delta z_\omega=\theta_\omega f_\omega(z_\omega).
\end{equation}
Then, by elliptic regularity, 
$z_\omega\in C^2\cap L^\infty$, with $z_\omega(x)\to0$ exponentially as 
$|x|\to\infty$. Hence, $z_\omega\in L^q$ for all $q\ge1$. 
Furthermore, the following classical integral identities (Nehari and Pohozaev) 
provide a relation between $m_\omega$ and $\theta_\omega$:
\begin{equation}\label{int_id}
\intrn|\nabla z|^2\diff x=\theta_\omega\intrn f_\omega(z)z\diff x, \quad 
\intrn|\nabla z|^2\diff x=2^*\theta_\omega\intrn F_\omega(z)\diff x.
\end{equation}
Indeed, if $z_\omega\in H^1\cap L^{p+1}$ satisfies $\intrn|\nabla z|^2\diff x=m_\omega$
and $2^*\intrn F_\omega(z_\omega) \diff x=1$, it follows that
\begin{equation*}
m_\omega=\theta_\omega>0.
\end{equation*}

Once a solution $z_\omega$ of \eqref{Ptilde_omega} is obtained, one gets 
a solution $v_\omega$ of \eqref{S_omega} by a simple dilation,
\begin{equation*}
v_\omega(x):=z_\omega({\theta_\omega}^{-1/2}x), \quad x\in\rn.
\end{equation*}

Finally, $u_\omega=r(v_\omega)$ is a solution of \eqref{eq:E_omega} for all $\omega>0$.


\section{Uniqueness and non-degeneracy}\label{sec:uniqueness}

Equipped with the change of variables in Lemma~\ref{lem:equivsols}, we can now prove uniqueness and non-degeneracy of
$u_\omega$. Our proof is a consequence of results from McLeod~\cite{McLeod-93} and Lewin, Rota Nodari~\cite{LewRot-20} 
(see also Adachi et al.~\cite{AdaShiWat-18}).

\begin{proof}[Proof of Theorem~\ref{thm:uniqueness}]
It is straightforward to check that $f_\omega$ satisfies the hypotheses of \cite[Theorem 4.1]{Frank-13}. As a consequence, if $u_\omega$ is a positive solution to~\eqref{eq:E_omega}, then it is radial decreasing with respect to some point in $\rn$. Next a direct computation shows that
$f_\omega$ satisfies the hypotheses of \cite[Theorem 2]{McLeod-93} when $3<p<\tfrac{3N+2}{N-2}$ for $N\ge 3$ or $p>3$ for $N=2$, and of \cite[Theorem 1]{LewRot-20} when $1<p\le 3$ (see also \cite{AdaShiWat-18}). Thus, we deduce that the solution $v_\omega$ is the unique positive radial solution of \eqref{S_omega}, modulo translations. Furthermore, it is non-degenerate:
\begin{equation}
    \label{eq:nondegeneracyPomega}
    \left\{
    \begin{aligned}
        &\ker(-\Delta -f_\omega '(v_\omega))=\mathrm{span}\{\partial_{x_1} v_{\omega},\ldots,\partial_{x_N}v_\omega\},\\
        &\ker\left(-\Delta -\frac{f_\omega(v_\omega)}{v_\omega}\right)=\mathrm{span}\{v_\omega\}.
    \end{aligned}
    \right.
\end{equation}

Since $f_\omega$ satisfies the hypotheses of \cite[Theorem 2]{GidNiNir-81}, all positive classical solutions of \eqref{S_omega} that tend to zero at infinity are radial decreasing about some point in $\R^N$. As a consequence, the solution $v_\omega$ is the unique positive solution of \eqref{S_omega}, modulo translations. 

Thanks to the monotonicity of $r(s)$, we can conclude that $u_\omega=r(v_\omega)$ is the unique positive solution of \eqref{eq:E_omega}, modulo translations.

To prove the non-degeneracy of $u_\omega$, we have to compute the linearized operator of \eqref{eq:E_omega} around $u_\omega$. A straightforward computation gives $\mathcal L u= L_+ \re(u)+i L_- \im{u}$ with $L_+$ and $L_-$ defined by \eqref{eq:defLplus} and \eqref{eq:defLminus} respectively. Note the first eigenvalue of the operators $L_+$ and $L_-$, when it exists, is necessarily simple with a positive eigenvalue. This can be proved for instance by using that $\langle w,L_{+/-} w\rangle\ge \langle |w|,L_{+/-} |w|\rangle$ and Harnack's inequality.

Since $u_\omega$ is a solution of \eqref{eq:E_omega}, it is clear that $L_-u_\omega=0$. Moreover, $u_\omega>0$. As a consequence, $u_\omega$ is the first eigenfunction of $L_-$ and $\ker(L_-)=\mathrm{span}\{u_\omega\}$.

Next, recall that $u_\omega=r(v_\omega)$ and let $\eta=\frac{w}{r'(v_\omega)}$. As a consequence, 
\begin{align*}
    L_+w=&-\nabla \cdot \left((1+2\delta u_\omega^2)\nabla w\right) -\delta(4u_\omega \Delta u_\omega +2 |\nabla u_\omega|^2)w -pu_\omega^{p-1}w+\omega w\\
    =&-\nabla \cdot \left(\frac{1}{(r'(v_\omega))^2}\nabla w\right)  -\delta(4r(v_\omega)r'(v_\omega)  \Delta v_\omega + 4r(v_\omega)r''(v_\omega)|\nabla v_\omega|^2  +2 (r'(v_\omega))^2 |\nabla v_\omega|^2)w \\
    &-P'_\omega(r(v_\omega)) w\\
    =&-\nabla \cdot \left(\frac{1}{r'(v_\omega)}\nabla \eta+\frac{r''(v_\omega)}{(r'(v_\omega))^2}\eta \nabla v_\omega\right)   -\delta( 4r(v_\omega)r''(v_\omega) +2 (r'(v_\omega))^2) |\nabla v_\omega|^2 r'(v_\omega)\eta\\
    &+ 2\frac{r''(s)}{(r'(s))^2}(\Delta v_\omega)\eta-P'_\omega(r(v_\omega)) r'(v_\omega)\eta\\
    =&-\frac{\Delta \eta}{r'(v_\omega)}
    -\left(\frac{r''(v_\omega)}{(r'(v_\omega))^2}\right)'\eta |\nabla v_\omega|^2   -\delta( 4r(v_\omega)r''(v_\omega) +2 (r'(v_\omega))^2) |\nabla v_\omega|^2 r'(v_\omega)\eta\\
    &+ \frac{r''(s)}{(r'(s))^2}(\Delta v_\omega)\eta-P'_\omega(r(v_\omega)) r'(v_\omega)\eta\\
    =&-\frac{\Delta \eta}{r'(v_\omega)}
    - \frac{r''(s)}{(r'(s))}P_\omega(r(v_\omega))\eta-P'_\omega(r(v_\omega)) r'(v_\omega)\eta=\frac{-\Delta \eta -f_\omega'(v_\omega)\eta}{r'(v_\omega)}\\
\end{align*}
where we use the fact that $r''(s)=-2\delta r(s) (r'(s))^4$ and $v_\omega$ is a solution of \eqref{S_omega}.

Since $0<v_\omega\le v_\omega(0)$, the multiplier $r'(v_\omega)$ is bounded away from $0$ and we deduce that $\eta\in L^2(\R^N)$ if and only if $w\in L^2(\R^N)$. Hence, $w\in \ker(L_+)$ if and only if $\eta\in \ker(-\Delta-f'_\omega)$. As a consequence, the non-degeneracy of $v_\omega$ implies that $w\in \ker(L_+)$ if and only if $\eta=\frac{w}{r'(v_\omega)}\in \mathrm{span}\{\partial_{x_1}v_\omega,\ldots,\partial_{x_N}v_\omega\}$. As a conclusion, using that $\partial_{x_i}u_\omega=r'(v_\omega)\partial_{x_i}v_\omega$ for all $i=1,\ldots,N$, we deduce that
\begin{equation*}
    \ker(L_+)=\mathrm{span}\{\partial_{x_1}u_\omega,\ldots,\partial_{x_N}u_\omega\}.
\end{equation*}
This concludes the proof of Theorem \ref{thm:uniqueness}.    
\end{proof}

The next proposition will be useful to prove parts (ii) and (iii) of Theorem~\ref{thm:asympt_omega_to_0}.

\begin{proposition}\label{prop:spectrum_L+}
The linearized operator $L_+$ has exactly one negative eigenvalue. 
\end{proposition}

\begin{proof}
We shall use again the relation
\begin{equation*}
    L_+w=\frac{-\Delta \eta -f_\omega'(v_\omega)\eta}{r'(v_\omega)},
\end{equation*}
with $\eta=\frac{w}{r'(v_\omega)}$.
In this proof we will also denote by $'$ differentiation with respect to $r\in(0,\infty)$. We shall use the same notation
for $v_{\omega}$ and its radial counterpart, \emph{i.e.}~$v_{\omega}(r)$, with $r=|x|$, and similarly for other spherically symmetric functions.
On the one hand, by taking $w=v'_\omega r'(v_\omega)$, we see that
\begin{equation*}
    \langle L_+ w, w\rangle =\langle (-\Delta  -f_\omega'(v_\omega))v'_\omega, v'_\omega\rangle =-(N-1)\intrn \frac{(v'_\omega(x))^2}{|x|^2}\diff x<0.
\end{equation*}
As a consequence, $L_+$ has at least one negative eigenvalue. Let $\lambda_1<0$ be the first eigenvalue of $L_+$.

Note that, for any eigenvalue $\lambda$ of $L_+$, there exists $w\in L^2(\R^N)$ such that $L_+w=\lambda w$ while the corresponding $\eta$ solves
\begin{equation*}
    (-\Delta  -f_\omega'(v_\omega) -\lambda r'(v_\omega)^2)\eta=0,
\end{equation*}
that is $\eta \in \ker(-\Delta  -f_\omega'(v_\omega) -\lambda r'(v_\omega)^2)$. Since the operator $-\Delta  -f_\omega'(v_\omega) -\lambda r'(v_\omega)^2$ commutes with space rotations, it may be written as a direct sum $\bigoplus_{\ell\ge 0} A^{(\ell)}_\lambda\otimes \mathds{1}$ with
\begin{equation*}
    A^{(\ell)}_\lambda\eta=-\eta''-\frac{N-1}{s}\eta'+\frac{\ell(\ell+N-2)}{s^2}\eta -f_\omega'(v_\omega)\eta -\lambda r'(v_\omega)^2\eta. 
\end{equation*}
For any $\lambda<0$ and $\ell\ge 1$, $A^{(\ell)}_\lambda>A^{(\ell)}_0\ge 0$ in the sense of quadratic forms. As a consequence, $\ker (A^{(\ell)}_\lambda)=\{0\}$ for all $\lambda<0$ and $\ell\ge 1$. 
Hence, to the first eigenvalue $\lambda_1<0$ of $L_+$, there corresponds $\eta_1\in L^2(\rn)$, $\eta_1>0$, such that
\begin{equation*}
    \eta_1''+\frac{N-1}{s}\eta_1'+f_\omega'(v_\omega)\eta_1+\lambda_1 r'(v_\omega)^2\eta_1=0.
\end{equation*}
Moreover, using that $f_\omega'(v_\omega)+\lambda_1 r'(v_\omega)^2\in L^{\infty}(\R^N)$, we can easily show that $\eta_1\in H^{1}(\rn)$. A classical bootstrap argument shows that $\eta_1\in W^{2,q}(\rn)$ for any $q\in [2,+\infty)$. This, thanks to Sobolev inequalities, implies $\eta_1\in C^{1}(\rn)$. Finally, using the ODE satisfied by $\eta_1$ and the fact that $f_\omega'(v_\omega)+\lambda_1 r'(v_\omega)^2\in L^{\infty}(\R^N)$, we can then show that $\eta_1'(0)=0$.

Next, suppose, by contradiction, that $L_+$ has a second eigenvalue $\lambda_2\in (\lambda_1,0)$. As above, there exists $\eta_2\in L^2(\rn)$ such that 
\begin{equation*}
    \begin{cases}
    \eta_2''+\frac{N-1}{s}\eta_2'+f_\omega'(v_\omega)\eta_2+\lambda_2 r'(v_\omega)^2\eta_2=0,\\
    \eta_2'(0)=0.
    \end{cases}
\end{equation*}
Furthermore, from the proof of \cite[Lemma 7.3]{LewRot-20}, we know that the unique solution to 
\begin{equation*}
    \begin{cases}
    \eta_0''+\frac{N-1}{s}\eta_0'+f_\omega'(v_\omega)\eta_0=0,\\
    \eta_0(0)=1,\ \eta_0'(0)=0,
    \end{cases}
\end{equation*}
vanishes exactly once and is such that $\lim_{s\to +\infty}\eta_0(s)=-\infty=\lim_{s\to +\infty}\eta'_0(s)$.

Since $f_\omega'(v_\omega)+\lambda_2 r'(v_\omega)^2>f_\omega'(v_\omega)+\lambda_1 r'(v_\omega)^2$, using Sturm's comparison theorem, we deduce that $\eta_2$ vanishes at least once in $(0,+\infty)$. Let $\mu_2>0$ such that $\eta_2(\mu_2)=0$. Since $f_\omega'(v_\omega)>f_\omega'(v_\omega)+\lambda_2 r'(v_\omega)^2$, we can apply twice Sturm's comparison theorem and deduce that $\eta_0$ has at least one zero in $(0,\mu_2)$ and at least one zero in $(\mu_2,+\infty)$. This contradicts the fact that $\eta_0$ vanishes exactly once and concludes the proof. 
\end{proof}

We conclude this section with the proof of Proposition \ref{prop:regularityfamsol},
which is a consequence of the implicit function theorem.

\begin{proof}[Proof of Proposition \ref{prop:regularityfamsol}]
Let $s\in \N$ such that $s>N$ and define $G:\R\times W^{2,s}\cap H^2_{\mathrm{rad}}\to L^{s}\cap L^{2}_{\mathrm{rad}}$ as
\begin{equation*}
    G(\omega,u)=-\Delta u +\omega u -\delta \Delta(u^2)u -u^p.
\end{equation*}
Using the fact that $W^{1,s}(\R^N)\subset L^{\infty}(\R^N)$, we deduce that $u,\partial_{x_j}u\in L^{\infty}(\R^N)$ for any $u\in W^{2,s}\cap H^2_{\mathrm{rad}}$ and any $j\in \{1,\ldots,N\}$. Then we can easily show that $G$ is well-defined and continuously Fr\'echet differentiable. 

For any $\omega_0>0$, let $u_{\omega_0}$ be the unique positive solution to \eqref{eq:E_omega}. Then $G(\omega_0,u_{\omega_0})=0$ and, since $u_{\omega_0}$ is non-degenerate, $D_{u}G(\omega_0, u_{\omega_0})=L_+$ is one-to-one. Hence, it remains to prove that $L_+$ is an isomorphism from $W^{2,s}\cap H^2_{\mathrm{rad}}$ onto $L^{s}\cap L^{2}_{\mathrm{rad}}$. 
Since $u_{\omega_0}$ is in the Schwartz space $\mathcal S(\R^N)$, the operator $L_+$ can be seen as a compact perturbation of $-(1+2\delta u_{\omega_0}^2)\Delta+\omega_0$ which is itself an isomorphism from $W^{2,s}\cap H^2_{\mathrm{rad}}$ onto $L^{s}\cap L^{2}_{\mathrm{rad}}$. 

Applying the implicit function theorem, we conclude that there exists a $C^1$ map $\omega \mapsto u(\omega)\in W^{2,s}\cap H^2_{\mathrm{rad}}$ defined in a neighborhood of $\omega_0$ of solutions to \eqref{eq:E_omega}.

To conclude that $u(\omega)=u_{\omega}$, it suffices to prove that $u(\omega)$ is positive, which can be done via a spectral theory argument. For each $\omega>0$, consider the operator 
\begin{equation*}
    H_\omega =-\Delta + V_\omega,
\end{equation*}
where
\begin{equation*}
    V_\omega(x):=\omega -\delta \Delta(u(\omega)(x)^2)-u(\omega)(x)^{p-1}, \quad x\in\rn.
\end{equation*}
Note that $V_\omega\in L^\infty$ and $H_\omega$ has a zero eigenvalue with eigenfunction $u(\omega)$. Moreover, when $\omega=\omega_0$, $u(\omega_0)=u_{\omega_0}>0$ and $0$ is an isolated simple eigenvalue at the bottom of the spectrum. Hence, for $\omega$ close to $\omega_0$, the zero eigenvalue of $H_\omega$ must also be at the bottom of the spectrum. Hence, by applying \cite[Theorem XIII.46]{ReeSim4}, we deduce that $u(\omega)$ must be positive for any $\omega$ close to $\omega_0$.
\end{proof}


\section{Proof of Theorem~\ref{thm:asympt_omega_to_0}}\label{sec:limitomegazero}

We shall decompose the proof of Theorem~\ref{thm:asympt_omega_to_0} into three parts:
first (i), then (iii), and finally (ii), which is more involved.

\subsection{Subcritical case}\label{sec:subcritical}

In the subcritical case, the behavior of $u_\omega$ as $\omega\to0^+$ can be determined
by a straightforward application of the implicit function theorem.

\begin{proof}[Proof of Theorem~\ref{thm:asympt_omega_to_0}~(i)]
When $N=2$ or $N\ge 3$ and $1<p<\frac{N+2}{N-2}$, the rescaled function $\tilde u_\omega$ defined by \eqref{eq:rescaledsub} solves
\begin{equation}
    \label{subrescaledeq}
    \Delta \tilde u_\omega-\tilde u_\omega+\tilde u_\omega^p+\omega^{\frac{2}{p-1}}\delta\Delta (\tilde u_\omega^2)\tilde u_\omega=0
\end{equation}
and it converges to the unique positive solution $Q$ to the NLS equation \eqref{eq:NLSpowersub}. More precisely, the implicit function theorem gives 
\begin{equation}\label{expansiontildeu}
    \|\tilde u_\omega-Q+\omega^{\frac{2}{p-1}}\delta(\mathcal L_Q)_{\mathrm{rad}}^{-1}(-\Delta(Q^2)Q)\|_{H^2(\R^N)\cap L^\infty(\rn)}=o\left(\omega^{\frac{2}{p-1}}\right),
\end{equation}
where $\mathcal L_Q:=-\Delta -pQ^{p-1}+1$ and $\|\cdot\|_{H^2(\R^N)\cap L^\infty(\rn)}=\max\{\|\cdot\|_{H^2(\R^N)}, \|\cdot\|_{L^\infty(\rn)}\}$. 
Using \eqref{expansiontildeu}, we obtain
\begin{align}\label{eq:subcrit_formula_M_omega_1}
    M(\omega)=&\,\int_{\R^N}u_\omega^2(x)\diff x=\omega^{\frac{2}{p-1}-\frac{N}{2}} \int_{\R^N}\tilde u_\omega^2(x)\diff x \notag \\
    =&\, \omega^{\frac{4-N(p-1)}{2(p-1)}}\intrn Q^2\diff x -2\omega^{\frac{8-N(p-1)}{2(p-1)}}\langle Q, \delta(\mathcal L_Q)_{\mathrm{rad}}^{-1}(-\Delta(Q^2)Q)\rangle+o\left(\omega^{\frac{8-N(p-1)}{2(p-1)}}\right).
\end{align}
To compute the derivative of $M$, we note that $L_+ \partial_\omega u_\omega=-u_\omega$ and, thanks to the non-degeneracy of $u_\omega$, we can write
\begin{equation*}
    M'(\omega)=2\intrn u_\omega(x)\partial_\omega u_\omega(x)\diff x=-2\intrn u_\omega(x) (L_+)^{-1} u_\omega(x)\diff x.
\end{equation*}
Using the change of variables defined in \eqref{eq:rescaledsub}, we obtain 
\begin{equation*}
    M'(\omega)=-2\omega^{\frac{4-N(p-1)}{2(p-1)}-1}\intrn \tilde u_\omega(x) (\tilde{\mathcal L}(\omega))^{-1} \tilde u_\omega(x)\diff x.
\end{equation*}
with 
\begin{equation*}
    \label{eq:defLtildeomega}
    \tilde{\mathcal L}(\omega)= -(1+2\delta \omega^{\frac{2}{p-1}} \tilde u_\omega^2)\Delta-4\delta \omega^{\frac{2}{p-1}} \tilde u_\omega\nabla \tilde u_\omega\cdot\nabla -\omega^{\frac{2}{p-1}}\delta(4\tilde u_\omega \Delta \tilde u_\omega +2 |\nabla \tilde u_\omega|^2) -p\tilde u_\omega^{p-1}+1.
\end{equation*}
Since $\tilde u_\omega$ converges to $Q$ in $L^\infty$ as $\omega\to0$, we deduce that $\tilde{\mathcal L}(\omega)$ converges to $\mathcal L_Q$ in the norm resolvent sense (see \cite[Theorem VIII.25]{ReeSim1}). Since $0\in \rho((\mathcal L_Q)_{\mathrm{rad}})\cap\rho(\tilde{\mathcal L}(\omega)_{\mathrm{rad}})$, we obtain the convergence
\begin{equation}
    (\tilde{\mathcal L}(\omega)_{\mathrm{rad}})^{-1}\to ((\mathcal L_Q)_{\mathrm{rad}})^{-1}
\end{equation}
in norm. More precisely, using \eqref{expansiontildeu}, we have
\begin{align*}
    \mathcal L_Q &\,- \tilde{\mathcal L}(\omega)= \omega^{\frac{2}{p-1}} \left(2\delta \tilde u_\omega^2\Delta+4\delta  \tilde u_\omega\nabla \tilde u_\omega\cdot\nabla +\delta(4\tilde u_\omega \Delta \tilde u_\omega +2 |\nabla \tilde u_\omega|^2)\right) +p(u_\omega^{p-1}-Q^{p-1})\\
    &\,=\omega^{\frac{2}{p-1}} \delta\left(2 Q^2\Delta+4Q\nabla Q\cdot\nabla +4Q \Delta Q +2 |\nabla Q|^2-p(p-1)Q^{p-2}(\mathcal L_Q)_{\mathrm{rad}}^{-1}(-\Delta(Q^2)Q)\right) +o(\omega^{\frac{2}{p-1}})\\
    &\,=-\omega^{\frac{2}{p-1}}\delta(-2\nabla \cdot( Q^2\nabla)+\Theta)+o(\omega^{\frac{2}{p-1}}),
\end{align*}
where $\Theta:=-(4Q \Delta Q +2 |\nabla Q|^2-p(p-1)Q^{p-2}(\mathcal L_Q)_{\mathrm{rad}}^{-1}(-\Delta(Q^2)Q))$ is understood as a multiplication operator. By iterating the resolvent identity, we obtain 
\begin{equation*}
    \|(\tilde{\mathcal L}(\omega)_{\mathrm{rad}})^{-1}- ((\mathcal L_Q)_{\mathrm{rad}})^{-1}+\omega^{\frac{2}{p-1}}\delta(\mathcal L_Q)_{\mathrm{rad}}^{-1}(-2\nabla \cdot(Q^2\nabla)+\Theta)(\mathcal L_Q)_{\mathrm{rad}}^{-1}\|=o(\omega^{\frac{2}{p-1}}).
\end{equation*}
Note that $(\mathcal L_Q)_{\mathrm{rad}}^{-1}(-\Delta(Q^2)Q)$ decays at the same rate as $Q$ so that $Q^{p-2}(\mathcal L_Q)_{\mathrm{rad}}^{-1}(-\Delta(Q^2)Q)$ tends to $0$ at infinity even if $p<2$. As a consequence, 
\begin{align}
    \omega^{1-\frac{4-N(p-1)}{2(p-1)}}M'(\omega)=&\,-2\left\langle Q-\omega^{\frac{2}{p-1}}(\mathcal L_Q)_{\mathrm{rad}}^{-1}(-\Delta(Q^2)Q), (\tilde{\mathcal L}(\omega)_{\mathrm{rad}})^{-1}(Q-\omega^{\frac{2}{p-1}}(\mathcal L_Q)_{\mathrm{rad}}^{-1}(-\Delta(Q^2)Q))\right\rangle\nonumber\\
    &\,+o(\omega^{\frac{2}{p-1}})\nonumber\\
    =&\,-2\left\langle Q, (\mathcal L_Q)_{\mathrm{rad}}^{-1}Q\right\rangle+2\omega^{\frac{2}{p-1}}\delta\left\langle Q, (\mathcal L_Q)_{\mathrm{rad}}^{-1}(-2\nabla \cdot(Q^2\nabla)+\Theta)(\mathcal L_Q)_{\mathrm{rad}}^{-1}Q\right\rangle\nonumber\\ \label{eq:derivativemasssub}
    &\,+4\omega^{\frac{2}{p-1}}\delta\left\langle (\mathcal L_Q)_{\mathrm{rad}}^{-1}Q,(\mathcal L_Q)_{\mathrm{rad}}^{-1}(-\Delta(Q^2)Q)\right\rangle +o(\omega^{\frac{2}{p-1}}).
\end{align}
With a scaling argument, we can easily compute
\begin{equation*}
    (-\Delta +p Q^{p-1}+1)\left(\frac{rQ'}{2}+\frac{Q}{p-1}\right)=-Q
\end{equation*}
for the NLS case, so that
\begin{equation*}
    (\mathcal L_Q)_{\mathrm{rad}}^{-1}Q=-\left(\frac{rQ'}{2}+\frac{Q}{p-1}\right).
\end{equation*}
This leads to 
\begin{equation*}
    -2\left\langle Q, (\mathcal L_Q)_{\mathrm{rad}}^{-1}Q\right\rangle=2\left\langle Q, \frac{rQ'}{2}+\frac{Q}{p-1}\right\rangle=\frac{4-N(p-1)}{2(p-1)}\|Q\|^2_{L^2}
\end{equation*}
and gives the first-order term in $M'$ whenever $p\neq 1+\frac{4}{N}$.

To compute the next-order term in $M$ and $M'$ we proceed as follows. Let $Q_{\omega}(x)=\omega^{\frac{1}{p-1}}Q(\omega^{1/2}x)$. We know that 
\begin{equation*}
    -\Delta Q_{\omega}-Q^p_\omega+\omega Q_\omega=0 \text{ and } \partial_\omega Q_\omega=-(\mathcal{L}_{Q_\omega})_{\mathrm{rad}}^{-1}Q_\omega
\end{equation*}
with 
\begin{equation*}
    \mathcal{L}_{Q_\omega}=-\Delta-pQ^{p-1}_\omega+\omega.
\end{equation*}
On the one hand, 
\begin{align*}
    -2\left\langle {Q_\omega}, (\mathcal L_{Q_\omega})_{\mathrm{rad}}^{-1}(-\Delta (Q^2_\omega)Q_\omega)\right\rangle&\,=2\left\langle \partial_\omega{Q_\omega}, -\Delta (Q^2_\omega)Q_\omega\right\rangle=\frac{1}{2}\partial_\omega\intrn Q_\omega^2(-\Delta (Q^2_\omega))\diff x\\
    &\,= \frac{1}{2}\partial_\omega\intrn \left|\nabla Q_\omega^2\right|^2\diff x.
\end{align*}
On the other hand, 
\begin{align*}
    \intrn \left|\nabla Q_\omega^2\right|^2\diff x=\omega^{\frac{4}{p-1}+1-\frac{N}{2}}\intrn \left|\nabla Q^2\right|^2\diff x
\end{align*}
so that 
\begin{align}\label{eq:subcrit_formula_M_omega_2}
    -2\left\langle {Q}, (\mathcal L_{Q})_{\mathrm{rad}}^{-1}(-\Delta (Q^2)Q)\right\rangle=\frac{1}{2}\left(\frac{4}{p-1}+1-\frac{N}{2}\right)\intrn \left|\nabla Q^2\right|^2\diff x.
\end{align}
Formulas \eqref{eq:subcrit_formula_M_omega_1} and \eqref{eq:subcrit_formula_M_omega_2} prove \eqref{eq:subcrit_formula_M_omega}.
(Note that, alternatively, \eqref{eq:subcrit_formula_M_omega_2} can be obtained by using directly the explicit expression of $(\mathcal L_{Q})_{\mathrm{rad}}^{-1}Q$.)

Similarly, we consider $\partial^2_\omega Q_\omega$ and we compute 
\begin{align*}
    \mathcal{L}_{Q_\omega}\partial^2_\omega Q_\omega=p(p-1)Q_{\omega}^{p-2}(\partial_\omega Q_\omega)^2-2\partial_\omega Q_\omega=p(p-1)Q_{\omega}^{p-2}((\mathcal{L}_{Q_\omega})_{\mathrm{rad}}^{-1}Q_\omega)^2+2(\mathcal{L}_{Q_\omega})_{\mathrm{rad}}^{-1}Q_\omega.
\end{align*}
Hence, 
\begin{align*}
    \frac{1}{4}\partial^2_\omega \intrn \left|\nabla Q_\omega^2\right|^2\diff x=&\,\left\langle \partial^2_\omega{Q_\omega}, -\Delta (Q^2_\omega)Q_\omega\right\rangle+\left\langle (\partial_\omega{Q_\omega})^2, -\Delta (Q^2_\omega)\right\rangle+2 \left\langle Q_\omega\partial_\omega{Q_\omega}, -\Delta (Q_\omega\partial_\omega Q_\omega)\right\rangle\\
    =&\,\left\langle p(p-1)Q_{\omega}^{p-2}((\mathcal{L}_{Q_\omega})_{\mathrm{rad}}^{-1}Q_\omega)^2+2(\mathcal{L}_{Q_\omega})_{\mathrm{rad}}^{-1}Q_\omega,\right(\mathcal{L}_{Q_\omega})_{\mathrm{rad}}^{-1}(-\Delta (Q^2_\omega)Q_\omega)\rangle\\
    &\,+\left\langle ((\mathcal{L}_{Q_\omega})_{\mathrm{rad}}^{-1}Q_\omega)^2, -2|\nabla Q_\omega|^2-4 Q_\omega\Delta Q_\omega\right\rangle
    +2\left\langle\partial_\omega Q_\omega,-\nabla\cdot(Q_\omega^2\nabla \partial_{\omega}Q_\omega)\right\rangle\\
    =&\,\left\langle(\mathcal{L}_{Q_\omega})_{\mathrm{rad}}^{-1}Q_\omega,-2\nabla\cdot(Q_\omega^2\nabla (\mathcal{L}_{Q_\omega})_{\mathrm{rad}}^{-1}Q_\omega)\right\rangle
    +\left\langle (\mathcal{L}_{Q_\omega})_{\mathrm{rad}}^{-1}Q_\omega, \Theta(\omega) (\mathcal{L}_{Q_\omega})_{\mathrm{rad}}^{-1}Q_\omega\right\rangle\\
    &\,+2\left\langle (\mathcal{L}_{Q_\omega})_{\mathrm{rad}}^{-1}Q_\omega,\right(\mathcal{L}_{Q_\omega})_{\mathrm{rad}}^{-1}(-\Delta (Q^2_\omega)Q_\omega)\rangle
\end{align*}
with $\Theta(\omega):=-(4Q \Delta Q_\omega +2 |\nabla Q_\omega|^2-p(p-1)Q_\omega^{p-2}(\mathcal L_{Q_\omega})_{\mathrm{rad}}^{-1}(-\Delta(Q_\omega^2)Q_\omega))$. As a consequence, it is enough to evaluate $\frac{1}{2}\partial^2_\omega \intrn \left|\nabla Q_\omega^2\right|^2\diff x$ at $\omega=1$ to conclude. In particular, this gives
\begin{align*}
    &\omega^{1-\frac{4-N(p-1)}{2(p-1)}}M'(\omega) \\
    &=\frac{4-N(p-1)}{2(p-1)}\|Q\|^2_{L^2}+\omega^{\frac{2}{p-1}}\delta\frac{(8-(N-2)(p-1))(8-N(p-1))}{8(p-1)^2}\|\nabla(Q^2)\|^2_{L^2}
     +o(\omega^{\frac{2}{p-1}}).
\end{align*}
For $p=1+\frac{4}{N}$, we obtain 
\begin{align*}
    M'(\omega)
    =\omega^{\frac{N}{2}-1}\delta\frac{N(N+2)}{4}\|\nabla(Q^2)\|^2_{L^2}
     +o(\omega^{\frac{N}{2}-1}).
\end{align*}
The monotonicity properties of the map $\omega\mapsto M(\omega)$ follow from these formulas. This concludes the proof of part (i) of Theorem~\ref{thm:asympt_omega_to_0}.
\end{proof}


\subsection{Supercritical case}\label{sec:supercritical}

To prove part (iii) of Theorem~\ref{thm:asympt_omega_to_0},
we shall take advantage of another variational characterization of solutions 
of \eqref{S_omega}. We define the functional $J_\omega: K_\omega\to\R$ by
\begin{equation*}
J_\omega(z):=\frac{\intrn |\nabla z|^2 \diff x}{\left(2^*\intrn F_\omega(z)\diff x\right)^{1-2/N}},
\end{equation*}
where
\begin{equation*}
K_\omega:=\{z\in H^1(\rn) : \, \txt\intrn F_\omega(z) \diff x>0\}.
\end{equation*}
It follows from \cite{BerLio-83} that $K_\omega\neq\emptyset$.
We next observe that $J_\omega$ is invariant under dilations: 
$J_\omega(z_\lambda)=J_\omega(z)$ if $z\in K_\omega$ and
\begin{equation*}
z_\lambda(x):=z(\lambda^{-1/2}x), \quad \lambda>0, \ x\in\rn.
\end{equation*}

\begin{proposition}\label{prop:var_char_omega}
Let $m_\omega$ be defined by \eqref{m_omega}.
There holds
\begin{equation*}
\inf_{z\in K_\omega} J_\omega(z)=m_\omega.
\end{equation*}
Furthermore, suppose $\ffi\in K_\omega$ satisfies $J_\omega(\ffi)=m_\omega$.
Then, letting $\lambda=(2^*\intrn F_\omega(\ffi)\diff x)^{-2/N}$, 
the dilation $\ffi_\lambda$ is a minimizer for
\eqref{m_omega}.
\end{proposition}

\begin{proof}
Let $z\in H^1$ be a minimizer for \eqref{m_omega}:
\begin{equation*}
\int|\nabla z|^2=m_\omega, \quad 2^*\int F_\omega(z)=1.
\end{equation*}
It follows that $z\in K_\omega$ and $J_\omega(z)=m_\omega$. Hence,
$\inf_{z\in K_\omega} J_\omega(z)\le m_\omega$. Suppose by contradiction that
$\inf_{z\in K_\omega} J_\omega(z)< m_\omega$. One can then find $z\in K_\omega$
such that $J_\omega(z)< m_\omega$. 
Letting $\lambda=(2^*\int F_\omega(\ffi))^{-2/N}$, the dilation $z_\lambda$ satisfies
$2^*\int F_\omega(z_\lambda)=1$, so that
\begin{equation*}
m_\omega\le \int|\nabla z_\lambda|^2=J_\omega(z_\lambda)=J_\omega(z)<m_\omega.
\end{equation*}
This contradiction shows that, indeed, $\inf_{z\in K_\omega} J_\omega(z)=m_\omega$. 

Finally, if $\ffi\in K_\omega$ is such that $J_\omega(\ffi)=m_\omega$, 
again the dilation factor $\lambda=(2^*\int F_\omega(\ffi))^{-2/N}$ yields
$2^*\int F_\omega(\ffi_\lambda)=1$, and so
$\ffi_\lambda$ is a minimizer for \eqref{m_omega}.
\end{proof}

We shall now consider the so-called `zero-mass case', $\omega=0$. 
We let
\begin{equation}\label{eq:m_0}
m_0:=\inf\Big\{\intrn |\nabla z|^2 \diff x : \, z\in \dot{H}^1(\rn)\cap L^{p+1}(\rn), 
\ 2^*\intrn F_0(z)\diff x=1\Big\},
\end{equation}
where $f_0$ and $F_0$ are still defined by \eqref{f_omega} 
and \eqref{F_omega}, respectively. 
The existence of a spherically symmetric and radially nonincreasing
minimizer $z_0\in \dot{H}^1\cap L^{p+1}\cap C^2$ of \eqref{eq:m_0} follows 
from \cite{BerLio-83}. It is also known (see \cite{adachi-watanabe,flucher-muller}) 
that $z_0(x)$ decays like $|x|^{-(N-2)}$ as $|x|\to\infty$.
Furthermore, as in the case $\omega>0$, 
there exists a Lagrange multiplier $\theta_0>0$ such that
\begin{equation}\label{eq:Ptilde_0}
-\Delta z_0=\theta_0 f_0(z_0).
\end{equation}
The integral identities \eqref{int_id} still hold for $\omega=0$, and it follows that
\begin{equation*}
m_0=\theta_0>0.
\end{equation*}
Again, the dilation
\begin{equation*}
v_0(x):=z_0({\theta_0}^{-1/2}x), \quad x\in\rn,
\end{equation*}
produces a solution of 
\begin{equation}\label{eq:P_0}
-\Delta v=f_0(v) \tag{$\mathrm{P}_0$},
\end{equation}
and the change of variables $u_0=r(v_0)$ yields a solution of \eqref{eq:E_0}. 
Note that, as in the case $\omega>0$, the above changes of variables
give a one-to-one correspondence between $u_0$ and $z_0$. 
Furthermore, $u_0$ and $z_0$ have the same regularity and decay at infinity.
The main difference here compared to the case $\omega>0$ is that, in dimensions $N=3,4$,
$z_0\not\in L^2$, due to its slow decay as $|x|\to\infty$. 
Only in dimensions $N\ge5$ does $z_0\in L^2$.

We now define $J_0: K_0\to\R$ by
\begin{equation*}
J_0(z):=\frac{\intrn |\nabla z|^2\diff x}{\left(2^*\intrn F_0(z)\diff x\right)^{1-2/N}},
\end{equation*}
where
\begin{equation*}
\dot{K}_0:=\{z\in \dot{H}^1(\rn)\cap L^{p+1}(\rn) : \, \txt\intrn F_0(z)\diff x>0\}.
\end{equation*}
The following result is proved in the exact same way as 
Proposition~\ref{prop:var_char_omega}.

\begin{proposition}\label{eq:var_char_0}
There holds
\begin{equation*}
\inf_{z\in \dot{K}_0} J_0(z)=m_0.
\end{equation*}
Furthermore, suppose $\ffi\in \dot{K}_0$ satisfies $J_0(\ffi)=m_0$.
Then, letting $\lambda=(2^*\intrn F_0(\ffi)\diff x)^{-2/N}$, the dilation $\ffi_\lambda$ 
is a minimizer for \eqref{eq:m_0}.
\end{proposition}

We also have uniqueness, up to translations, and non-degeneracy of the minimizer $z_0$, 
as a consequence of the following theorem.

\begin{theorem}\label{thm:uniqueness_z_0}
Suppose $\frac{N+2}{N-2}<p<\frac{3N+2}{N-2}$.
Then the following properties hold true.
\item[(i)] \eqref{eq:Ptilde_0} has a unique positive radial solution $z_0$ 
in $\dot{H}^1$.
\item[(ii)] $z_0$ non-degenerate:
\begin{equation*}
\ker(L_0)=\mathrm{span}\{\partial_i z_0 : \, i=1,\dots,N\},
\end{equation*}
where 
\begin{equation}
L_0=-\Delta -\theta_0f_0'(z_0).
\end{equation}
\end{theorem}

The proof of the uniqueness is a direct application of \cite[Theorem 2]{Tang-2001} while the non-degeneracy in $L^2$ can be proved as in \cite[Lemma A.1]{LewRot-20}.

We can now state and prove the first main result of this section.

\begin{proposition}\label{prop:conv_of_vn}
Suppose $\frac{N+2}{N-2}<p<\frac{3N+2}{N-2}$.
Consider a sequence $(\omega_n)\subset (0,\infty)$, such that $\omega_n\to0$. 
Then
\begin{equation}\label{eq:asymptotics_of_vn}
v_{\omega_n}\to v_0 \quad\text{in}\quad\dot{H}^1(\rn)\cap L^q(\rn)\cap C^2(\rn), \quad \forall q\ge 2^*,
\end{equation}
and
\begin{equation}\label{eq:L2-asymptotics_of_vn}
\omega_n\|r(v_{\omega_n})\|_{L^2}^2\to 0.
\end{equation}
\end{proposition}

We will prove Proposition~\ref{prop:conv_of_vn} using several lemmas. 
The first one is technical.

\begin{lemma}\label{lem:lowdimensions_estimates}
Let $N\in\{3,4\}$.
Consider a function $z\in C(\rn)$, a number $\rho>0$ 
and $\eta_\rho\in C_0^\infty(\rn)$ a cut-off function
such that $\eta_\rho\equiv1$ on $B_\rho(0)$ and $\eta_\rho\equiv 0$ on 
$\R^N\setminus B_{2\rho}(0)$.
Suppose there exists a constant $C>0$ such that
\begin{equation*}
\lim_{|x|\to\infty}\frac{|z(x)|}{|x|^{N-2}}=C.
\end{equation*}
Then we have the following asymptotics as $\rho\to\infty$:
\begin{equation}\label{eq:lowdimensions_estimates}
\intrn |\eta_\rho z|^2\diff x=
\begin{cases}
O(\rho) & \text{if} \ N=3, \\
O(\log(\rho))  & \text{if} \ N=4.
\end{cases}
\end{equation}
\end{lemma}

\begin{proof}
Let $f:(0,\infty)\to \real$ be a continuous function such that
$0<f(\rho)<\rho$ for all $\rho>0$. For $\rho \gg 1$, we have
$$
\intrn |\eta_\rho z|^2\diff x 
\ge \int_{f(\rho)\le|x|<\rho} |\eta_\rho z|^2\diff x
= \int_{f(\rho)\le|x|<\rho} |z(x)|^2\diff x
\gtrsim \int_{f(\rho)}^\rho r^{-2(N-2)}r^{N-1}\diff r.
$$
If $N=3$, choosing $f(\rho)=\rho/2$ yields
$$
\intrn |\eta_\rho z|^2\diff x
\gtrsim \int_{\rho/2}^\rho\diff r=\frac{\rho}{2}.
$$
If $N=4$, we choose $f(\rho)=\sqrt{\rho}$ and we get
$$
\intrn |\eta_\rho z|^2\diff x
\gtrsim \int_{\sqrt{\rho}}^\rho r^{-1}\diff r=\frac{\log\rho}{2}.
$$
This gives the desired lower bounds. The upper bounds are straightforward.
\end{proof}

\begin{lemma}\label{lem:omega-asymptotics}
As $\omega\to0$, there holds
$$
1\le\frac{m_\omega}{m_0}\le 1+o(1).
$$
\end{lemma}

\begin{proof}
Firstly, for $\omega>0$ and $z>0$,
\begin{equation*}
F_0(z)=\frac{1}{p+1}r(z)^{p+1}
>\frac{1}{p+1}r(z)^{p+1}-\frac{\omega}{2}r(z)^2=F_\omega(z).
\end{equation*}
Hence, using the minimizer $z_\omega$ as a test function for $J_0$, we have
\begin{equation*}
m_0\le J_0(z_\omega)
=\frac{\int |\nabla z_\omega|^2}{\left(2^*\int F_0(z_\omega)\right)^{1-2/N}}
<J_\omega(z_\omega)=m_\omega.
\end{equation*}

Now, if $N\ge5$, we have that $z_0\in L^2$. Furthermore, since $z_0\in \dot{K}_0$,
it follows by continuity that, for $\omega>0$ small enough, $z_0\in K_\omega$.
Therefore, we can use $z_0$ as a test function for $J_\omega$, which yields
\begin{equation*}
m_\omega\le J_\omega(z_0)
=\frac{\int |\nabla z_0|^2}{\left(2^*\int F_\omega(z_0)\right)^{1-2/N}}
=\frac{\int |\nabla z_0|^2}{\left(2^*\int F_0(z_0)\right)^{1-2/N}}
\cdot \frac{\left(2^*\int F_0(z_0)\right)^{1-2/N}}{\left(2^*\int F_\omega(z_0)\right)^{1-2/N}}
=m_0\cdot \left(\frac{\int F_0(z_0)}{\int F_\omega(z_0)}\right)^{1-2/N}.
\end{equation*}
By Remark~\ref{preservint.rem}, $r(z_0)\in L^{p+1}$ and we have
\begin{align*}
\frac{\int F_0(z_0)}{\int F_\omega(z_0)}
&=\frac{\frac{1}{p+1}\int r(z_0)^{p+1}}
{\frac{1}{p+1}\int r(z_0)^{p+1}
-\frac{\omega}{2}\int r(z_0)^{2}}\\
&=\frac{1}{1-\left(\frac{p+1}{2}\right)
\left(\frac{\int r(z_0)^{2}}{\int r(z_0)^{p+1}}\right)\omega}
=1+O(\omega), \quad\text{as} \ \omega\to0,
\end{align*}
which concludes the proof in case $N\ge5$.

For $N\in\{3,4\}$, we let $R>0$ and we introduce a cut-off 
function $\eta_R\in C_0^\infty(\R)$ such that $\eta_R(s)=1$ for $|s|<R$, 
$0<\eta_R<1$ for $R<|s|<2R$, $\eta_R(s)=0$ for $|s|>2R$,
and $|\eta_R'(s)|\le 2/R$ for all $s\in\R$. We shall simply write $\eta_Rz_0$
for the function $x\mapsto\eta_R(|x|) z_0(x)$. We have
$\eta_R z_0 \in L^2(\rn)$ and we now use it as a test function for $J_\omega$. 
As above, we have
\begin{equation*}
m_\omega\le J_\omega(\eta_R z_0)
=\frac{\int |\nabla\eta_R z_0|^2}{\left(2^*\int F_0(\eta_R z_0)\right)^{1-2/N}}
\cdot \left(\frac{\int F_0(\eta_Rz_0)}{\int F_\omega(\eta_Rz_0)}\right)^{1-2/N}.
\end{equation*}
First note that $\intrn |\nabla\eta_R z_0|^2\diff x \to 
\int |\nabla z_0|^2\diff x=m_0$ as $R\to\infty$ by dominated convergence. Next,
\begin{equation*}
    \intrn r(\eta_Rz_0(x))^{p+1}\diff x= 
    \intrn r(z_0(x))^{p+1}\diff x+ \int_{\R^N\setminus B_R(0)} (r(\eta_Rz_0(x))^{p+1}-r(z_0(x))^{p+1})\diff x.
\end{equation*}
Now, for any $x\in \R^N\setminus B_R(0)$, there exists 
$\tau(x)\in (z_0(x)-(1-\eta_R(|x|))z_0(x),z_0(x))$ such that
\begin{equation*}
    r(z_0(x))^{p+1}-r(\eta_Rz_0(x))^{p+1}=(p+1)r'(\tau_x)r(\tau_x)^{p}(1-\eta_R(|x|))z_0(x).
\end{equation*}
Since $z_0(x)$ decays like $|x|^{-(N-2)}$ as $|x|\to\infty$, so does $\tau(x)$. Hence, 
\begin{equation*}\label{eq:cutoff-estimate0}
     r(\eta_Rz_0(x))^{p+1}-r(z_0(x))^{p+1}=O(|x|^{-(p+1)(N-2)}).
\end{equation*}
This leads to 
\begin{equation}\label{eq:cutoff-estimate1}
    \intrn r(\eta_Rz_0(x))^{p+1}\diff x= \intrn r(z_0(x))^{p+1}\diff x+O(R^{N-(p+1)(N-2)}),
\end{equation}
where we observe that $N-(p+1)(N-2)<0$ since $p>\frac{N+2}{N-2}$.
As consequence, we have
$$
\frac{\int |\nabla\eta_R z_0|^2}{\left(2^*\int F_0(\eta_R z_0)\right)^{1-2/N}} \to m_0,
\quad \text{as} \ R\to\infty.
$$
Furthermore, by Lemma~\ref{lem:lowdimensions_estimates},
\begin{equation}\label{eq:cutoff-estimate2}
f_N(R):=\intrn r(\eta_R z_0)^2 \diff x =
\begin{cases}
O(R) & \text{if} \ N=3, \\
O(\log(R)) & \text{if} \ N=4,
\end{cases}
\quad \text{as} \ R\to\infty.
\end{equation}
For $\omega,R>0$, by \eqref{eq:cutoff-estimate1} and \eqref{eq:cutoff-estimate2} 
we have
$$
\frac{\int F_0(\eta_R z_0)}{\int F_\omega(\eta_R z_0)}
=\left(1-\left(\frac{p+1}{2}\right)\left(\frac{f_N(R)}
{\int r(z_0)^{p+1}+O(R^{N-(p+1)(N-2)})}\right)\omega\right)^{-1}.
$$
We now conclude the proof in the following way.

If $N=3$, we let $R=\omega^{-1/2}$ and we have
$$
\frac{\int F_0(\eta_R z_0)}{\int F_\omega(\eta_R z_0)}
=1+O(\omega f_3(\omega^{-1/2}))=1+O(\omega^{1/2}), \quad\text{as} \ \omega\to0.
$$

If $N=4$, we let $R=\omega^{-1}$ and we have
$$
\frac{\int F_0(\eta_R z_0)}{\int F_\omega(\eta_R z_0)}
=1+O(\omega f_4(\omega^{-1}))=1+O(\omega \log(\omega^{-1})), 
\quad\text{as} \ \omega\to0,
$$
which completes the proof.
\end{proof}

\begin{lemma}\label{lem:omega_z_to_0}
Consider a sequence $(\omega_n)\subset (0,\infty)$, such that $\omega_n\to0$. 
For all $n\in\N$, let $z_n:=z_{\omega_n}$ be a minimizer for \eqref{m_omega} with $\omega=\omega_n$. 
Then
\begin{equation}\label{eq:L2-asymptotics_z}
\|\nabla z_n\|_{L^2}^2 \to m_0=\|\nabla z_0\|_{L^2}^2, \quad
\omega_n\|r(z_n)\|_{L^2}^2\to 0, \quad
\|r(z_n)\|_{L^{p+1}}^{p+1}\to \frac{p+1}{2^*}.
\end{equation}
\end{lemma}

\begin{proof}
Since $\|\nabla z_n\|_{L^2}^2=m_{\omega_n}$,
the first limit follows directly from Lemma~\ref{lem:omega-asymptotics}.
Furthermore, 
\begin{equation}\label{eq:L2_control}
\frac{1}{2^*}
=\intrn F_{\omega_n}(z_n)\diff x=\frac{1}{p+1}\intrn r(z_n)^{p+1}\diff x
-\frac{\omega_n}{2}\intrn r(z_n)^2 \diff x
=\intrn F_{0}(z_n)\diff x-\frac{\omega_n}{2}\intrn r(z_n)^2\diff x.
\end{equation}
Hence,
\begin{equation*}
J_0(z_n)=\frac{\int |\nabla z_n|^2}{\left(2^*\int F_0(z_n)\right)^{1-2/N}}
=\frac{\int |\nabla z_n|^2}{\left(1+ \frac{2^*}{2}\omega_n\int r(z_n)^2\right)^{1-2/N}}
=\frac{m_{\omega_n}}{\left(1+\frac{2^*}{2}\omega_n\int r(z_n)^2\right)^{1-2/N}}.
\end{equation*}
Now, suppose by contradiction that $\limsup \omega_n\|r(z_n)\|_{L^2}^2>0$.
Then, using Lemma~\ref{lem:omega-asymptotics}, there is a subsequence along which, 
for $n$ large enough,
$$
m_0\le J_0(z_n)=\frac{m_{\omega_n}}{\left(1+\frac{2^*}{2}\omega_n\int r(z_n)^2\right)^{1-2/N}}
\le  \frac{m_{0}(1+o(1))}{\left(1+\frac{2^*}{2}\omega_n\int r(z_n)^2\right)^{1-2/N}}<m_0.
$$
This contradiction proves that, indeed, $\omega_n\|r(z_n)\|_{L^2}^2\to 0$ as
$n\to\infty$.
Finally, the third limit follows directly from \eqref{eq:L2_control}.
\end{proof}

The next two lemmas provide classical results that are crucial in our analysis, and which will be proved in the Appendix
for completeness.

\begin{lemma}[Radial Lemma]\label{lem:radial_lemma}
\item[(i)] Let $s\ge 1$ and $u\in L^s(\rn)$ be a radial nonincreasing function. 
Then,
\begin{equation}\label{eq:radial_estimate}
\forall x\neq0, \quad  |u(x)|\le C_{N,s}\|u\|_{L^s}|x|^{-N/s},
\end{equation}
where $C_{N,s}=(N/|\mathbb{S}^{N-1}|)^{1/s}$.
\item[(ii)] Let $\{u_n\}\subset \dot H^1(\rn)$ be a sequence of radial
nonincreasing functions, such that 
$$
\sup_{n\in\N}\|u_n\|_{\dot H^1}<\infty.
$$
Then there exists $u\in \dot H^1(\rn)$ such that, up to a subsequence:
\begin{equation}
\forall R>0, \ \forall q>2^*, \quad u_n \to u \ \text{in} \ 
L^\infty(\rn\sm B_R(0))\cap L^q(\rn\sm B_R(0)).
\end{equation}
\end{lemma}


\begin{lemma}\label{lem:uniform_bound_z}
Let $\frac{N+2}{N-2}\le p <\frac{3N+2}{N-2}$, $N\ge3$.
Let $z_\omega\in \dot H^1(\rn)$ be the minimizer obtained above, which solves
the Euler-Lagrange equation
\begin{equation}
-\Delta z_\omega=m_\omega f_\omega (z_\omega).    
\end{equation}
\item[(i)] There exists $\eta_0>0$ and $M_0>0$ such that 
\begin{equation}\label{eq:uniform_bound_z}
\sup_{0<\omega<\eta_0}\|z_\omega\|_{L^\infty}\le M_0.
\end{equation}
\item[(ii)] 
There exists $\eta>0$ such that
for all $s\ge 2^*$ there exists a constant 
$K_{N,s}>0$ such that
\begin{equation}\label{eq:uniform_Ls_bounds}
\forall \omega\in(0,\eta), \quad \|z_\omega\|_{L^s(\rn)}\le K_{N,s}.
\end{equation}
\end{lemma}

Note that the result in (ii) is trivial for $s=2^*$ since $\dot H^1(\rn) \hookrightarrow L^{2^*}(\rn)$.
For $s>2^*$, the proof follows a Moser-type iteration argument, which will be given in the
Appendix.

We are now in a position to prove Proposition~\ref{prop:conv_of_vn}.

\begin{proof}[Proof of Proposition~\ref{prop:conv_of_vn}]
Suppose $\omega_n\to0$ and let $z_n$ be defined as in Lemma~\ref{lem:omega_z_to_0}. 
By \eqref{eq:L2-asymptotics_z}, $\{z_n\}$ is bounded in $\dot H^1(\rn)$. Hence,
there exists $z_*\in \dot H^1(\rn)$ and a subsequence of $\{z_n\}$ (still denoted by $\{z_n\}$) 
such that
\begin{equation*}
z_n \wto z_* \quad\text{weakly in} \ \dot{H^1} \quad\text{and}\quad 
z_n\to z_* \quad \text{a.e.~in} \ \rn.
\end{equation*}
Furthermore, by Lemma~\ref{lem:omega-asymptotics} and Lemma~\ref{lem:uniform_bound_z}~(ii),
for all $s\ge 2^*$ there exists a constant $C_s>0$ such that 
\begin{equation}\label{eq:moser}
\|z_n\|_{L^s(\rn)}\le C_s, \quad \forall n\in\N.
\end{equation}
Let $q>2^*$. We will first show that, up to a further subsequence,
\begin{equation}\label{eq:Lq_conv}
z_n \to z_* \quad\text{in} \ L^q(\rn).
\end{equation}
Let $w_n:=z_n-z_*$.
By the Radial Lemma, we already know that, up to a subsequence,
\begin{equation}\label{eq:Lq_outside}
w_n \to 0 \quad\text{in} \ L^q(\rn\sm B_1(0)).
\end{equation}
We now fix $s>q$. Since $\{z_n\}$ is bounded in $L^s(\rn)$, there
exists $z_{**}\in L^s(\rn)$ such that, up to a subsequence, 
$z_n\wto z_{**}$ weakly in $L^s$ and $z_n\to z_{**}$
a.e.~in $\rn$. But then $z_{**}=z_*$ a.e.~and we conclude that 
$z_*\in L^s(\rn)$. As a consequence, by \eqref{eq:radial_estimate} and \eqref{eq:moser},
there is a constant $C>0$ such that
\begin{equation*}
\forall n\in\N, \ \forall x\neq0, \quad
|w_n(x)|\le C_{s,N}\|z_n-z_*\|_{L^s}|x|^{-N/s}\le C |x|^{-N/s}.
\end{equation*}
Hence, $|w_n(x)|^q\le C^q|x|^{-N\frac{q}{s}}$ 
for $x\in B_1(0)\sm\{0\}$. Since
$|x|^{-N\frac{q}{s}}\in L^1(B_1(0))$ and $|w_n|^q\to 0$ a.e. in $B_1(0)$,
it follows by dominated convergence that
\begin{equation}\label{eq:Lq_inside}
w_n \to 0 \quad\text{in} \ L^q(B_1(0)).
\end{equation}
Thus, \eqref{eq:Lq_conv} follows from \eqref{eq:Lq_outside} 
and \eqref{eq:Lq_inside}.

Next, using \eqref{eq:Lq_conv} with $q=p+1$, we have
\begin{equation*}
\|r(z_*)-r(z_n)\|_{L^{p+1}}\le \|z_*-z_n\|_{L^{p+1}}\to 0,
\end{equation*}
so that
\begin{equation*}
\|r(z_*)\|_{L^{p+1}}=\lim_{n\to\infty}\|r(z_n)\|_{L^{p+1}}.
\end{equation*}
Hence, using the constraint $\int F_{\omega_n}(z_n)=\frac{1}{2^*}$ for all $n\in\N$,
it follows by Lemma~\ref{lem:omega_z_to_0} that
\begin{equation}\label{eq:zstar_constraint}
\intrn F_0(z_*)\diff x=\frac{1}{p+1}\|r(z_*)\|_{L^{p+1}}^{p+1}
=\frac{1}{p+1}\lim_{n\to\infty}\|r(z_n)\|_{L^{p+1}}^{p+1}
=\lim_{n\to\infty}\intrn F_{\omega_n}(z_n)\diff x=\frac{1}{2^*}.
\end{equation}
We deduce that $z_*\neq0$ and $z_*$ satisfies the constraint in \eqref{eq:m_0}.

Furthermore, by weak lower semi-continuity of $z\mapsto \|\nabla z\|_{L^2}$
on $\dot H^1(\rn)$,
\begin{equation*}
m_0\le \|\nabla z_*\|_{L^2}^2\le \liminf_{n\to\infty}\|\nabla z_n\|_{L^2}^2
=\lim_{n\to\infty}m_{\omega_n}=m_0.
\end{equation*}
Hence, $z_*$ is a minimizer for \eqref{eq:m_0}. It follows that $z_*$ is a positive
solution of \eqref{eq:Ptilde_0} with $\theta_0=m_0$, and so $z_*=z_0$ by 
Theorem~\ref{thm:uniqueness_z_0}.

We now prove that $z_n\to z_0$ in $\dot H^1(\rn)$. We have
\begin{equation*}
\|\nabla z_n - \nabla z_0\|_{L^2}^2=
\|\nabla z_n\|_{L^2}^2+\|\nabla z_0\|_{L^2}^2
-2\intrn \nabla z_n\cdot \nabla z_0\diff x.
\end{equation*}
On the one hand,
\begin{equation*}
\|\nabla z_n\|_{L^2}^2=m_{\omega_n}\to m_0 \quad\text{and}\quad \|\nabla z_0\|_{L^2}^2=m_0.
\end{equation*}
On the other, using \eqref{eq:Ptilde_0} we find 
\begin{align*}
\intrn \nabla z_n\cdot \nabla z_0\diff x
&=-\intrn z_n\Delta z_0\diff x \\
&=-\intrn z_0\Delta z_0\diff x-\intrn \Delta z_0(z_n-z_0)\diff x \\
&=\intrn |\nabla z_0|^2\diff x + m_0 \intrn f_0(z_0)(z_n-z_0)\diff x \\
&=m_0 + m_0 \intrn f_0(z_0)(z_n-z_0)\diff x.
\end{align*}
By Hölder's inequality, we have
\begin{align*}
\left|\intrn f_0(z_0)(z_n-z_0)\diff x\right|
&\le \intrn r(z_0)^p|z_n-z_0|\diff x \\
&\le \|z_0\|_{L^{p+1}}^p\|z_n-z_0\|_{L^{p+1}}\to 0.
\end{align*}
Therefore,
\begin{equation*}
\|\nabla z_n - \nabla z_0\|_{L^2}^2=
\|\nabla z_n\|_{L^2}^2+\|\nabla z_0\|_{L^2}^2
-2\intrn \nabla z_n\cdot \nabla z_0\diff x
\to 0.
\end{equation*}
We conclude that
\begin{equation*}
z_n\to z_0 \quad\text{in}\quad\dot{H}^1(\rn)\cap L^q(\rn), \quad \forall q\ge 2^*.
\end{equation*}

Next, by Lemma~\ref{lem:uniform_bound_z}~(i), the sequence $\{z_n\}$ is bounded in $L^\infty(\rn)$.
It then follows by standard elliptic theory arguments (see \emph{e.g.}~the proof of
Proposition~\ref{prop:C^2-convergence} below) that $z_n\to z_0$ in $C^2(\rn)$.

To complete the proof, we now turn to $\{v_{\omega_n}\}$. Observing that
\begin{equation*}
v_{\omega_n}(x)=z_n(m_{\omega_n}^{-1/2}x) \quad\text{and}\quad
v_0(x)=z_0(m_{0}^{-1/2}x),
\end{equation*}
we deduce from the conclusions obtained for the sequence $\{z_n\}$ that, up to a subsequence,
\begin{equation*}
v_{\omega_n}\to v_0 \quad\text{in}\quad\dot{H}^1(\rn)\cap L^q(\rn)\cap C^2(\rn), \quad \forall q\ge 2^*.
\end{equation*}
Finally, since the only possible limit point is $v_0$, a proof by contradiction shows that, 
in fact, the whole sequence $\{v_{\omega_n}\}$ converges to $v_0$. 
\end{proof}

Going back to the variable $u_\omega=r(v_\omega)$, we obtain the following.

\begin{proposition}\label{prop:conv_of_un}
Suppose $\frac{N+2}{N-2}<p<\frac{3N+2}{N-2}$.
Consider a sequence $(\omega_n)\subset (0,\infty)$, such that $\omega_n\to0$. 
Then
\begin{equation}\label{eq:asymptotics_of_un}
u_{\omega_n}\to u_0=r(v_0) \quad\text{in}\quad\dot{H}^1(\rn)\cap L^q(\rn)\cap C^2(\rn), \quad \forall q\ge 2^*,
\end{equation}
and
\begin{equation}\label{eq:L2-asymptotics_of_un}
\omega_n\|u_{\omega_n}\|_{L^2}^2\to 0.
\end{equation}
\end{proposition}

\begin{proof}
First, \eqref{eq:L2-asymptotics_of_un} is a direct reformulation of \eqref{eq:L2-asymptotics_of_vn}. Next, in 
\eqref{eq:asymptotics_of_un}, the $L^q$-convergence of $u_{\omega_n}$ follows directly from the $L^q$-convergence 
of $v_{\omega_n}$ in \eqref{eq:asymptotics_of_vn} and the fact that $r$ is Lipschitz. The $\dot H^1$-convergence follows
from the $\dot H^1$-convergence of $v_{\omega_n}$ by dominated convergence. Finally, $C^2$-convergence
follows from the $C^2$-convergence of $v_{\omega_n}$, using the $L^\infty$-bound on $v_{\omega_n}$ 
(Lemma~\ref{lem:uniform_bound_z}~(i)) and the fact that $r$, $r'$ and $r''$ are Lipschitz on compact subsets
of $\R$.
\end{proof}


\subsubsection{Asymptotic behavior of $M(\omega)=\|u_\omega\|_{L^2}^2$ as $\omega\to0^+$}

\begin{proposition}
    \label{prop:asymptotic_mass_supercritical} Let $N\ge 3$ and $\frac{N+2}{N-2}<p<\frac{3N+2}{N-2}$. As $\omega \to 0^+$, we have 
    \begin{equation*}
    \lim_{\omega\to 0}\|u_\omega\|_{L^2}=
        \begin{cases}
            +\infty & \text{if } N\in \{3,4\},\\
            \|u_0\|_{L^2} & \text{if } N\ge 5.
        \end{cases}
    \end{equation*}
    More precisely,  
    \begin{equation}
        v_\omega\to v_0 \text{ and } u_\omega\to u_0 \text{ in } L^q(\rn)
    \end{equation}
    for all $q>\frac{N}{N-2}$.\\
    Moreover, letting  $\eta_\omega=r(v_\omega)r'(v_\omega)$ for any $\omega>0$, there holds
    \begin{equation}
         \eta_\omega\to \eta_0 
    \end{equation}
    in $L^q(\rn)$ for all $q>\frac{N}{N-2}$.
\end{proposition}

\begin{remark}
    The strong convergence of $v_\omega$, $u_\omega$ and $\eta_\omega$ holds in $L^q(\rn)$ for all $q>\frac{N}{N-2}$. This includes $L^2(\rn)$ only in dimensions $N\ge 5$.
\end{remark}

\begin{proof}
    Let us start with the case $N\in \{3,4\}$. First, we show that $\|v_\omega\|_{L^2}\to\infty$. Indeed, suppose by contradiction there is a sequence $\omega_n\to 0$ such that 
$\|v_n\|_{L^2}$ is bounded, where $v_n:=v_{\omega_n}$. Then using
Proposition~\ref{prop:conv_of_vn} $\{v_n\}$ is bounded in $H^1$ and (using Rellich-Kondrachov and Radial Lemma)
up to a subsequence, $v_n\to v_0$ in $L^2$, hence $v_0\in L^2$, a contradiction.
As a consequence, $\|v_\omega\|_{L^2}\to\infty$ as $\omega\to 0$.

Next, by Lemma~\ref{lem:uniform_bound_z}~(i), there exist $\eta_0,M_0>0$
such that $\|v_\omega\|_{L^\infty}\le M_0$ for all $0<\omega<\eta_0$.
Then, by Lemma~\ref{propofr.lem}, for all $0<\omega<\eta_0$,
$$
r(v_\omega)^2\le v_\omega^2\le M_0^2
$$
and
$$
\|u_\omega\|_{L^2}^2=\intrn|r(v_\omega)|^2\diff x 
\ge \intrn\frac{v_\omega^2}{1+2\delta r(v_\omega)^2}\diff x 
\ge \frac{1}{1+2\delta M_0^2}\intrn v_\omega^2\diff x.
$$
This implies, $\|u_\omega\|_{L^2}\to\infty$ as $\omega\to 0$.

Since $v_\omega\to v_0$ a.e.~and $\{v_\omega\}$ is bounded in $L^\infty$, dominated convergence yields $\|v_\omega-v_0\|_{L^q(B_R(0))}\to 0$
as $\omega\to0$, for any $R>0$ and $q>\frac{N}{N-2}$. Furthermore, a maximum principle argument as in \cite{LewRot-20}
allows one to improve the classical bound
$$
v_\omega(x)\le \frac{C_0}{|x|^\frac{N-2}{2}} \quad (|x|\ge 1)
$$
to
$$
v_\omega(x)\le \frac{C_\eps}{|x|^{N-2-\eps}} \quad (|x|\ge R_\eps)
$$
with $R_\eps,C_\eps$ independent of $\omega$. This provides an upper bound in $L^q(|x|\ge R_\eps)$, for $q(N-2)>N$,
if $\eps>0$ is chosen small enough, and it follows by dominated convergence that
$\|v_\omega-v_0\|_{L^q(\rn\sm B_{R_\eps}(0))}\to 0$ as $\omega\to0$.

Next, by Lemma~\ref{propofr.lem},
$$
\|u_\omega-u_0\|_{L^q} = \|r(v_\omega)-r(v_0)\|_{L^q} 
\le \|v_\omega-v_0\|_{L^q} \to 0, \quad \omega\to0.
$$
Finally, let $q>\frac{N}{N-2}$ and $\eta_\omega=r(v_\omega)r'(v_\omega)$,
\begin{align*}
\|\eta_\omega-\eta_0\|_{L^q}^q
&=\intrn |r(v_\omega)r'(v_\omega)-r(v_0)r'(v_0)|^q\diff x \\
&\le \intrn |r'(v_\omega)|^q|r(v_\omega)-r(v_0)|^q\diff x
    +\intrn |r(v_0)|^q|r'(v_\omega)-r'(v_0)|^q\diff x \\
&\le \intrn |r(v_\omega)-r(v_0)|^q\diff x
    +\intrn |r(v_0)|^q|r'(v_\omega)-r'(v_0)|^q\diff x .
\end{align*}
The first term of the right-hand side of this inequality is simply 
$\|u_\omega-u_0\|_{L^q}^q$
and goes to zero as $\omega\to0$. 
Since $v_\omega\to v_0$ a.e.~and $r\in C^1(\real)$,
the second term also goes to zero by dominated convergence.
We conclude that $\|\eta_\omega-\eta_0\|_{L^q}\to 0$ as $\omega\to0$.
\end{proof}


 \subsubsection{Asymptotic behavior of $M'(\omega)$ as $\omega\to 0$.}\label{subsec:upper_bound_M'_supercritical}

Following ideas from \cite{LewRot-20} (see also \cite{KilOhPoc-17}), we derive an upper bound on $M'(\omega)$. 

\begin{proposition}
    \label{prop:upper_bound_M'} Let $N\ge 3$ and $\frac{N+2}{N-2}<p<\frac{3N+2}{N-2}$. Then,  for $\omega>0$ small enough, we have
    \begin{align}
        \label{eq:upper_bound_M'}
        \frac{M'(\omega)}{2}&\,\frac{N-2}{N+2}\left[-\left(\frac{p-1}{p+1}\right)^2\left(p-\frac{3N+2}{N-2}\right)\beta^2(\omega)-\frac{p-1}{p+1}(p-3)\beta(\omega)-\frac{4}{N}\right]\nonumber\\
        &< \frac{M^2(\omega)}{2 T(\omega)} \left[\frac{p-1}{p+1}\left(N+2-\frac{N}{2}(p-1)\right)\beta(\omega)-2\right]
    \end{align}
   where 
    \begin{equation*}
        T(\omega)= \int_{\R^N} |\nabla u_\omega(x)|^2\diff x,\quad \beta(\omega)=T(\omega)^{-1}\int_{R^N}u_\omega(x)^{p+1}\diff x.
    \end{equation*}
    Moreover, 
    \begin{enumerate}
        \item if $N\in \{3,4\}$, then
            \begin{equation*}
                \lim_{\omega\to 0^+} M'(\omega)=-\infty;
            \end{equation*}
        \item if $N=5$, then
            \begin{equation*}
                M'(\omega)<0
            \end{equation*}
            for all $\omega$ small enough;
        \item if $N\ge 6$ and 
        \begin{equation}
            \label{eq:cond_p_supercritical_1}
            p<\frac{2(N+2)}{N}-\sqrt{1-16\frac{N+2}{N^2(N-2)}}\ \text{ or }\ 
            \frac{2(N+2)}{N}+\sqrt{1-16\frac{N+2}{N^2(N-2)}}<p
        \end{equation}
         then
            \begin{equation*}
                M'(\omega)<0
            \end{equation*}
            for all $\omega$ small enough.
    \end{enumerate}
\end{proposition}

\begin{remark} For $N\ge 6$, the condition~\eqref{eq:cond_p_supercritical_1} is probably not optimal. 
As an illustration, for $N=7$, \eqref{eq:cond_p_supercritical_1} is satisfied whenever $1.8< p\le 1.92$ or $3.22 \le p <4.6$. For $N=8$, we need $\tfrac{5}{3}<p\le 1.73$ or $3.27 \le p <\tfrac{13}{3}$.
\end{remark}

\begin{proof}
    For any $\omega>0$, let $\mathcal L(\omega)=L_+$ be defined by~\eqref{eq:defLplus}. We know from Proposition~\ref{prop:spectrum_L+} that, for any $\omega>0$, $\mathcal L(\omega)$ has exactly one negative eigenvalue. Then we define the symmetric matrix $L=(L_{ij})$ given by the restriction of $\mathcal L(\omega)$ to the finite dimensional space spanned by $\left\{\partial_\omega u_\omega, u_\omega, x\cdot\nabla u_\omega+\frac{N}{2}u_\omega\right\}$. 

    Since
    \begin{align*}
        &\mathcal L(\omega)\partial_\omega u_\omega=-u_\omega, \\
        &\mathcal L(\omega) u_\omega=-2\delta u_\omega \Delta u_\omega^2+(1-p)u_\omega^p,\\
        &\mathcal L(\omega)\left(x\cdot\nabla u_\omega+\frac{N}{2}u_\omega\right)=-N\delta u_\omega\Delta u_\omega^2+\left(2+\frac{N}{2}(1-p)\right)u_\omega^p-2\omega u_\omega,
    \end{align*}
    straightforward computations give
    \begin{align*}
        L_{11}:=&\,\langle \partial_\omega u_\omega, \mathcal L(\omega)\partial_\omega u_\omega\rangle=-\frac{M'(\omega)}{2},\quad L_{12}:=\langle \partial_\omega u_\omega, \mathcal L(\omega) u_\omega\rangle=-M(\omega),\\
        L_{13}:=&\,\langle \partial_\omega u_\omega ,\mathcal L(\omega)\left(x\cdot\nabla u_\omega+\frac{N}{2}u_\omega\right)\rangle=-\langle u_\omega ,x\cdot\nabla u_\omega+\frac{N}{2}u_\omega\rangle=0,\\ 
        L_{22}:=&\,\langle  u_\omega, \mathcal L(\omega) u_\omega\rangle= 8\delta \intrn u^2_\omega|\nabla u_\omega|^2-(p-1)\intrn u_\omega^{p+1},\\
        L_{23}:=&\,\langle  u_\omega,\mathcal L (\omega)\left(x\cdot\nabla u_\omega+\frac{N}{2}u_\omega\right)\rangle=4\delta N \intrn u_\omega^2|\nabla u_\omega|^2+\left(2+\frac{N}{2}(1-p)\right)\intrn u_\omega^{p+1}-2\omega M(\omega),\\
        L_{33}:=&\,\langle  x\cdot\nabla u_\omega+\frac{N}{2}u_\omega,\mathcal L (\omega)\left(x\cdot\nabla u_\omega+\frac{N}{2}u_\omega\right)\rangle=2\delta N\left(1+\frac{N}{2}\right) \intrn u_\omega^2|\nabla u_\omega|^2\\
        &+\frac{N}{2}\frac{p-1}{p+1}\left(2+\frac{N}{2}(1-p)\right)\intrn u_\omega^{p+1}.
    \end{align*}

    Next, for any $\omega>0$, let $Q(\omega)=\int_{\R^n}u_\omega(x)^2 |\nabla u_\omega(x)|^2\diff x$. As a consequence of Proposition~\ref{prop:integral_identities}, we can write
    \begin{align*}
        &T(\omega)+2\delta Q(\omega)=2^*\left(\frac{1}{p+1}T(\omega)\beta(\omega)-\frac{\omega}{2}M(\omega)\right),\\
        &T(\omega)+4\delta Q(\omega)=T(\omega)\beta(\omega)-\omega M(\omega),
    \end{align*}
    with $2^*=\frac{2N}{N-2}$, and deduce that 
    \begin{equation}
        \label{eq:QomegaMomega}
        \begin{cases}
            &2\delta Q(\omega)=\left(1-\frac{2^*}{p+1}\right)T(\omega)\beta(\omega)+\frac{2}{N-2}\omega M(\omega),\\
            &\omega M(\omega)=\frac{N-2}{N+2}\left(2\frac{2^*}{p+1}-1\right)T(\omega)\beta(\omega)-\frac{N-2}{N+2}T(\omega).
        \end{cases}
    \end{equation}
    As a consequence, 
    \begin{equation*}
        2\delta Q(\omega)=\frac{N}{N+2}\frac{p-1}{p+1}T(\omega)\beta(\omega)-\frac{2}{N+2}T(\omega)
    \end{equation*}
    and
    \begin{align*}
        L_{22}=&\,T(\omega)\left[\frac{p-1}{p+1}\left(2\frac{N-2}{N+2}-(p-1)\right)\beta(\omega)-\frac{8}{N+2}\right], \\
        L_{23}=&\, T(\omega)\left[\frac{N}{2}\frac{p-1}{p+1}(3-p)\beta(\omega)-2\right],\\
        L_{33}=&\, T(\omega)\left[\frac{N}{2}\frac{p-1}{p+1}\left(N+2-\frac{N}{2}(p-1)\right)\beta(\omega)-N\right].
    \end{align*}
    Moreover, 
    \begin{equation}
        \label{eq:betaomega_relation}
        \left(\frac{3N+2}{N-2}-p\right)\frac{\beta(\omega)}{p+1}=1+\frac{N+2}{N-2}\frac{\omega M(\omega)}{T(\omega)}.
    \end{equation}
    In particular, using $\frac{N+2}{N-2}<p<\frac{3N+2}{N+2}$, this leads to
    \begin{equation}
        \label{eq:betaomega_limit}
        1=\lim_{\omega\to 0^+}\left(\frac{3N+2}{N-2}-p\right)\frac{\beta(\omega)}{p+1}
    \end{equation}
    since $\lim_{\omega\to 0^+}\omega M(\omega)=0$. 
    
    Now, a tedious but straightforward computation gives
    \begin{align*}
        L_{22}L_{33}&-L^2_{23}=\frac{N(N-2)}{N+2}T^2(\omega)\left[\left(\frac{p-1}{p+1}\right)^2\left(p-\frac{3N+2}{N-2}\right)\beta^2(\omega)+\frac{p-1}{p+1}(p-3)\beta(\omega)+\frac{4}{N}\right]\\
    \end{align*}
    and, using again $\frac{N+2}{N-2}<p<\frac{3N+2}{N+2}$, we obtain from \eqref{eq:betaomega_limit} that 
    \begin{align*}
        L_{22}L_{33}-&\,L^2_{23}\\
        \underset{\omega \to 0}{\sim}& \frac{N(N-2)}{N+2}T^2(0)\left(\frac{3N+2}{N-2}-p\right)^{-1}\left[-(p-1)^2+(p-1)(p-3)+\frac{4}{N}\left(\frac{3N+2}{N-2}-p\right)\right]\\
        &= \frac{2N(N-2)}{N+2}T^2(0)\left(\frac{3N+2}{N-2}-p\right)^{-1}\left[-(p-1)+\frac{2}{N}\left(\frac{3N+2}{N-2}-p\right)\right]\\
        &= 2(N-2)T^2(0)
       \left(\frac{3N+2}{N-2}-p\right)^{-1}\left(\frac{N+2}{N-2}-p\right)<0.
    \end{align*}
    Since $\mathcal L(\omega)$ has a unique negative eigenvalue, we deduce that the determinant of $L$ is negative:
    \begin{align*}
        0>\det(L)=\frac{M'(\omega)}{2}(L^2_{23}-L_{22}L_{33})-M(\omega)^2T(\omega)\left[\frac{N}{2}\frac{p-1}{p+1}\left(N+2-\frac{N}{2}(p-1)\right)\beta(\omega)-N\right].
    \end{align*}
    This gives the estimate~\eqref{eq:upper_bound_M'}.

    Using again \eqref{eq:betaomega_limit}, we have 
    \begin{align*}
        \left[\frac{p-1}{p+1}\left(N+2-\frac{N}{2}(p-1)\right)\beta(\omega)-2\right] \underset{\omega \to 0}{\sim} \left(\frac{3N+2}{N-2}-p\right)^{-1} C(p)
    \end{align*}
    with 
    \begin{align*}
        C(p)= (p-1)\left(N+2-\frac{N}{2}(p-1)\right)-2\left(\frac{3N+2}{N-2}-p\right)=-\frac{N}{2}p^2+2(N+2)p-\frac{3N^2+10N}{2(N-2)}.
    \end{align*}
    The function $C(p)$ is a second order polynomial and its maximum on $\left(1, +\infty\right)$ is reached at $p_*(N)=\frac{2(N+2)}{N}$. In particular, 
    \begin{align*}
        C(p_*(N))=\frac{N^3-2N^2-16N-32}{2N(N-2)}<0
    \end{align*}
    for $N\in \{3,4,5\}$. This, together with Proposition~\ref{prop:asymptotic_mass_supercritical}, concludes the proof for $N\in \{3,4,5\}$.
    
    For $N\ge 6$, $C(p_*(N))>0$ and $C(p)$ vanishes twice in  $\left(1, +\infty\right)$ at
    \begin{equation*}
        p_-(N)=p_*(N)-\sqrt{\frac{2}{N}C(p_*(N))} \text{ and } p_+(N)=p_*(N)+\sqrt{\frac{2}{N}C(p_*(N))}.
    \end{equation*}
    Note that $p_-(N)> \frac{N+2}{N-2}$ and $p_+(N)<\frac{3N+2}{N-2}$, so that $C(p)<0$ for $p\in \left(\frac{N+2}{N-2}, p_-(N)\right)$ or $p\in \left(p_+(N),\frac{3N+2}{N-2}\right)$. As a consequence, for such $p$'s, $M'(\omega)<0$ whenever $\omega$ is small enough.
\end{proof}

Proposition~\ref{prop:upper_bound_M'} gives a complete description of the asymptotic behavior of $M'(\omega)$ for $N\in\{3,4\}$. We conclude this section with the study of $M'(\omega)$ for $N\ge 5$. We know from above that $u_\omega\to u_0$ in $\dot{H}^1(\rn)\cap L^{\infty}(\rn)$, $u_0$ being the unique positive radial-decreasing solution to~\eqref{eq:E_0}. More precisely, for any $\omega>0$, $u_\omega=r(v_\omega)$ and $v_\omega\to v_0$ in $\dot{H}^1(\rn)\cap L^{\infty}(\rn)$ with $u_0=r(v_0)$ and $v_0$ the unique positive radial-decreasing solution to 
\begin{equation*}
    -\Delta v=f_0(v).
\end{equation*}

To discuss the asymptotic behavior of $ M'(\omega)$ as $\omega \to 0$, we proceed as follows. On the one hand,
\begin{align*}
    M'(\omega)=\partial_\omega \intrn u_\omega^2\diff x
    =2 \intrn u_\omega \partial_\omega u_\omega \diff x
    =2 \intrn r(v_\omega) r'(v_\omega) \partial_\omega v_\omega \diff x= 2 \intrn \eta_\omega \partial_\omega v_\omega\diff x,
\end{align*}
with $\eta_\omega=r(v_\omega) r'(v_\omega)$.
Now, since $v_\omega$ is a solution to~\eqref{S_omega}, we have
\begin{equation*}
    (-\Delta-f'_\omega(v_\omega))\partial_\omega v_\omega=-r(v_\omega) r'(v_\omega).
\end{equation*}
This implies
\begin{align*}
    M'(\omega)=- 2 \intrn \eta_\omega (L_\omega)^{-1}_{\mathrm{rad}} \eta_\omega
\end{align*}
with $L_\omega:=-\Delta -f'_\omega(v_\omega)$.

On the other hand, from Theorem~\ref{thm:uniqueness_z_0}, we know that the limiting linearized operator
\begin{equation*}
    L_0:=-\Delta -f'_0(v_0)
\end{equation*}
has a trivial kernel $\ker(L_0)=\mathrm{span}\{\partial_i v_0 : \, i=1,\dots,N\}$ so that $(L_0)^{-1}_{\mathrm{rad}}$ can be defined using functional calculus. Moreover, since $v_\omega$ converges to $v_0$ in $L^\infty$ as $\omega\to0$, we deduce that $L_\omega$ converges to $L_0$ in the norm resolvent sense (see \cite[Theorem VIII.25]{ReeSim1}). More precisely, we have the following lemma.

\begin{lemma}
    \label{lem:conv_linearized_op_supercritical}
    Let $N\ge 5$ and $\frac{N+2}{N-2}<p<\frac{3N+2}{N-2}$. For any $\omega\ge 0$, define $V_\omega:=-f'_\omega(v_\omega)$. Then, as $\omega\to 0$,
    \begin{equation}
        \label{eq:conv_linearized_op_supercritical}
        V_\omega\to V_0 \text{ in } L^{\infty}(\rn), \text{ and }\  (V_\omega-\omega) \to V_0 \text{ in } L^{q}(\rn) 
    \end{equation}
    for all $q>\frac{N}{N-2}$ if $p\ge 2$, for all $q\ge \frac{N}{4}$ if $p<2$. 
\end{lemma}

\begin{proof}
    First of all, using $f_\omega(s)=f_0(s)-\omega r'(s)r(s)$ and the properties of $r(s)$, we have
    \begin{equation*}
        f'_\omega(v_\omega)=f'_0(v_\omega)-\omega\frac{1}{(1+2\delta r^2(v_\omega))^2}
    \end{equation*}
    so that
    \begin{equation*}
        V_0-V_\omega=f'_\omega(v_\omega)- f'_0(v_0)= f'_0(v_\omega)-f'_0(v_0)-\omega\frac{1}{(1+2\delta r^2(v_\omega))^2}
    \end{equation*}
    and
    \begin{equation*}
        V_0-(V_\omega -\omega)=f'_\omega(v_\omega)+\omega- f'_0(v_0)= f'_0(v_\omega)-f'_0(v_0)+\omega\left(1-\frac{1}{(1+2\delta r^2(v_\omega))^2}\right).
    \end{equation*}
    Next, for any $s\in(0,\infty)$, $f'_0(s)$ can be written as
    \begin{equation*}
        f'_0(s)=r'(s)^2 h_p(s) r(s)^{p-1}\ \text{with}\ h_p(s)=p-\frac{2\delta r^2(s)}{1+2\delta r^2(s)}.
    \end{equation*}
    A direct computation gives, for any $t,s \in(0,\infty)$,
    \begin{align*}
        |r'(s)^2-r'(t)^2|\le 2\delta (|r(s)|+|r(t)|)\,|s-t|,\\
        |h_p(s)-h_p(t)|\le 2\delta (|r(s)|+|r(t)|)\,|s-t|,
    \end{align*}
    while
    \begin{equation*}
        |r(s)^{p-1}-r(t)^{p-1}|\le 
        \begin{cases}
            |s-t|^{p-1} & \text{ if } p<2,\\
            (p-1)(|r(s)|^{p-2}+|r(t)|^{p-2})\,|s-t| & \text{ if } p\ge 2.
        \end{cases}
    \end{equation*}
    As a consequence, for any $x\in \R^N$,
    \begin{equation}\label{eq:bound_f'0}
        |f'_0(v_\omega(x))-f'_0(v_0(x))|\lesssim 
        \begin{cases}
            |v_\omega(x)-v_0(x)|^{p-1} & \text{ if } p<2,\\
            |v_\omega(x)-v_0(x)| & \text{ if } p\ge 2.
        \end{cases}
    \end{equation}
    where we used that $\|v_\omega\|_{L^\infty}$ is uniformly bounded for $\omega$ small enough. 
    This, together with the $L^\infty$-convergence of $v_\omega$, implies
    \begin{equation*}
        \|f'_0(v_\omega)-f'_0(v_0)\|_{L^\infty} \to 0 \text{ as } \omega\to 0.
    \end{equation*}
    Finally, 
    \begin{equation*}
        \|V_0-V_\omega\|_{L^\infty} \le \|f'_0(v_\omega)-f'_0(v_0)\|_{L^\infty}+\omega \to 0 \text{ as } \omega\to 0.
    \end{equation*}

    Next, using Proposition~\ref{prop:asymptotic_mass_supercritical} and~\eqref{eq:bound_f'0}, we deduce that 
    \begin{equation*}
        \|f'_0(v_\omega)-f'_0(v_0)\|_{L^q} \to 0 \text{ as } \omega\to 0
    \end{equation*}
    for any $q>\frac{N}{N-2}$ if $p\ge 2$ or $q\ge \frac{N}{4}$ if $p<2$. Hence, to prove that $(V_\omega-\omega)\to V_0$ in $L^q(\rn)$, it is enough to show that 
    \begin{equation*}
        \intrn \left(1-\frac{1}{(1+2\delta r^2(v_\omega(x)))^2}\right)^q\,\diff x
    \end{equation*}
    is uniformly bounded. This is the case for any $q>\frac{N}{2(N-2)}$. Indeed, using again that $\|v_\omega\|_{L^\infty}$ is uniformly bounded for $\omega$ small enough and that there exist $C,R>0$ independent of $\omega$ such that 
    \begin{equation*}
        v_\omega(x)\le \frac{C}{|x|^{N-2}}
    \end{equation*}
    for all $x\in \rn\setminus B_R(0)$ (see proof of Proposition~\ref{prop:asymptotic_mass_supercritical}), we have 
    \begin{equation*}
        \intrn \left(1-\frac{1}{(1+2\delta r^2(v_\omega(x)))^2}\right)^q\,\diff x\lesssim \intrn r^{2q}(v_\omega(x))\,\diff x \lesssim 1 + \int_{\rn\setminus B_R(0)} \frac{1}{|x|^{2q(N-2)}}\,\diff x \lesssim 1
    \end{equation*}
    for any $q>\frac{N}{2(N-2)}$. Since, in any case, $\frac{N}{2(N-2)}<\min\{\frac{N}{N-2},\frac{N}{4}\}$, this completes the proof.
\end{proof}

Since $L_0$ is the limit, in the norm resolvent sense, of $L_\omega$ and $L_\omega$ has exactly one negative eigenvalue (see Proposition \ref{prop:spectrum_L+}), we deduce that $(L_0)^{-1}_{\mathrm{rad}}$ has one negative eigenvalue and is otherwise positive and unbounded from above.
As a consequence, we expect that, in the limit $\omega \to 0$, $M'(\omega)$ behaves like $-2\langle \eta_0, (L_0)^{-1}_{\mathrm{rad}} \eta_0 \rangle$.

As in \cite[Lemma 4.2]{LewRot-20}, the first step to give a proper interpretation of the quadratic form $$\langle\eta_0, (L_0)^{-1}_{\mathrm{rad}} \eta_0 \rangle$$ is to prove that the quadratic form domain of $(L_0)^{-1}_{\mathrm{rad}}$ is the same as for the free Laplacian if $N\ge 5$. More precisely, we have the following lemma.

\begin{lemma}
    \label{lem:form_domain_L0}
    Let $N\ge 5$ and $\frac{N+2}{N-2}<p<\frac{3N+2}{N-2}$. There exists a constant $C>0$ such that
    \begin{equation*}
        \frac{1}{C}(-\Delta)_{\mathrm{rad}}^{-1}-C\le (L_0)^{-1}_{\mathrm{rad}} \le C (-\Delta)_{\mathrm{rad}}^{-1}.
    \end{equation*}
\end{lemma}

The proof of this lemma is exactly the same as in \cite[Lemma 4.2]{LewRot-20} and we do not reproduce it here. The key ingredients are the regularity of the potential $V_0=-f'_0(v_0)$ and the fact $L_0$ has exactly one negative eigenvalue. In particular, since $r(s)\sim_0 s$ and $v_0(|x|)\sim_{+\infty} |x|^{2-N}$, we have that $V_0$ behaves like $p |x|^{(2-N)(p-1)}$ at infinity. As a consequence, $V_0\in L^{q}(\rn)$ for any $q\ge \frac{N}{4}$ which is the regularity used in the proof of the lemma. Using that $L_0$ has exactly one negative eigenvalue, we then deduce that the operator $1+K_0=1+(-\Delta)^{-\frac12}V_0(-\Delta)^{-\frac12}$ is invertible (within the sector of radial functions) and that $(L_0)^{-1}_{\mathrm{rad}}$ can be written as
\begin{equation*}
    (-\Delta)^{-\frac12}(1+K_0)^{-1} (-\Delta)^{-\frac12}.
\end{equation*}

Next, we give the upper bound on $M'(\omega)$.

\begin{proposition}\label{prop:upper_bound_limit_supercritical}
    Let $N\ge 5$ and $\frac{N+2}{N-2}<p<\frac{3N+2}{N-2}$. Then
    \begin{equation}
        \label{eq:upper_bound_limit_supercritical}
        \limsup_{\omega \to 0^+} M'(\omega)\le -2\langle\eta_0, (L_0)^{-1}_{\mathrm{rad}} \eta_0 \rangle\in [-\infty,+\infty)
    \end{equation}
    in the sense of quadratic forms and where $\eta_0=r(v_0)r'(v_0)$. 
    In dimensions $N\in\{5,6\}$ the right side equals $-\infty$ whereas it is finite for $N\ge 7$.
\end{proposition}

As above, the proof of this lemma is exactly the same as in \cite[Lemma 4.3]{LewRot-20}. Once again we use the convergence of $L_\omega$ in the norm resolvent sense and the strong convergence of $\eta_\omega$ in $L^2(\rn)$ for $N\ge 5$. The fact that the right side of~\eqref{eq:upper_bound_limit_supercritical} is infinite in dimension $N\in \{5,6\}$ and finite for $N\ge 7$ depends on the behavior of $\|(-\Delta)^{-\frac12}\eta_0\|_{L^2}$. In particular, it depends on the behavior at infinity of $\eta_0=r(v_0)r'(v_0)$. On the one hand, we have
\begin{equation*}
    \eta_0(x)=r(v_0(x))r'(v_0(x))\le r(v_0(x))\le v_0(x).
\end{equation*}
On the other hand, using Lemma~\ref{propofr.lem}, we have
\begin{equation*}
    \eta_0(x)=r'(v_0(x))\frac{r(v_0(x))}{v_0(x)}v_0(x)\ge r'(v_0(x))^2 v_0(x)\ge \frac{1}{1+2\delta \|v_0\|^2_{L^\infty}}v_0(x).
\end{equation*}
As a consequence, using the decay of $v_0$ at infinity, we deduce that, for any $|x|\ge1$,
\begin{equation*}
    c\frac{1}{|x|^{N-2}}\le \eta_0(x)\le C\frac{1}{|x|^{N-2}}
\end{equation*}
as in \cite[Lemma 4.3]{LewRot-20}.



Proposition~\ref{prop:upper_bound_limit_supercritical} gives \eqref{eq:supercrit_M'_omega} for $N\le 6$. It remains to prove it in dimensions $N\ge 7$.

\begin{proposition}\label{prop:limit_supercritical_M'}
    Let $N\ge 7$ and $\frac{N+2}{N-2}<p<\frac{3N+2}{N-2}$. Then
    \begin{equation}
        \label{eq:limit_supercritical_M'}
        \lim_{\omega \to 0^+} M'(\omega)= -2\langle\eta_0, (L_0)^{-1}_{\mathrm{rad}} \eta_0 \rangle.
    \end{equation}
\end{proposition}

The proof of this lemma is similar to that of \cite[Lemma 4.4]{LewRot-20}. The main difference is linked to the particular form of $f_\omega(v_\omega)$ which contains a term of the form $\omega r'(v_\omega)r(v_\omega)$. 

Here the quadratic form $\langle\eta_0, (L_0)^{-1}_{\mathrm{rad}} \eta_0 \rangle$ should be interpreted as $\langle (-\Delta)^{-\frac12}\eta_0, (1+K_0)^{-1}(-\Delta)^{-\frac12}\eta_0\rangle$ with $K_0=(-\Delta)^{-\frac12}V_0 (-\Delta)^{-\frac12}$.

\begin{proof}
    We start by noticing that $L_\omega$ can be written as
    \begin{equation}
        \label{eq:Lomega_exp1}
        L_\omega=-\Delta-f'_\omega(v_\omega)=-\Delta+\omega-(f'_\omega(v_\omega)+\omega)=-\Delta+\omega+V_\omega=(-\Delta+\omega)^{\frac12}(1+K_\omega)(-\Delta+\omega)^{\frac12}
    \end{equation}
    where
    \begin{equation*}
        K_\omega:=(-\Delta+\omega)^{-\frac12}V_\omega(-\Delta+\omega)^{-\frac12}=\left(\frac{-\Delta}{-\Delta+\omega}\right)^{\frac12}(-\Delta)^{-\frac12}V_\omega(-\Delta)^{-\frac12}\left(\frac{-\Delta}{-\Delta+\omega}\right)^{\frac12}.
    \end{equation*}
    Thanks to Lemma~\ref{lem:conv_linearized_op_supercritical}, $f'_\omega(v_\omega)+\omega\to f'_0(v_0)=-V_0$ in $L^{N/2}(\rn)$. This, together with the Hardy-Littlewood-Sobolev inequality, implies that 
    \begin{equation*}
        (-\Delta)^{-\frac12}V_\omega(-\Delta)^{-\frac12}\to (-\Delta)^{-\frac12}V_0(-\Delta)^{-\frac12}=K_0
    \end{equation*}
    in operator norm. Since the operator $K_0$ is compact and $\left(\frac{-\Delta}{-\Delta+\omega}\right)^{\frac12}$ converges strongly to the identity, we deduce that $K_\omega\to K_0$ in operator norm. This implies that the spectrum of $K_{\omega}$ converges to the spectrum of $K_0$ and, since $(1+K_0)$ is invertible, we deduce that $(1+K_\omega)^{-1}$ is bounded and converges to $(1+K_0)^{-1}$ in operator norm. As a consequence, we can invert~\eqref{eq:Lomega_exp1} and write
    \begin{equation*}
        (L_\omega)^{-1}_{\mathrm{rad}}=(-\Delta+\omega)^{-\frac12}(1+K_\omega)^{-1}(-\Delta+\omega)^{-\frac12}
    \end{equation*}
    and 
    \begin{equation*}
        M'(\omega)=-2\langle (-\Delta+\omega)^{-\frac12}\eta_\omega, (1+K_\omega)^{-1}(-\Delta+\omega)^{-\frac12}\eta_{\omega}\rangle.
    \end{equation*}
    To conclude, it remains to prove that $(-\Delta+\omega)^{-\frac12}\eta_\omega$ converges to $(-\Delta)^{-\frac12}\eta_0$ in $L^2(\rn)$. From the HLS inequality, we have 
    \begin{equation*}
        \|(-\Delta)^{-\frac12}\eta_0\|_{L^2}\lesssim \|\eta_0\|_{L^{\frac{2N}{N+2}}}
    \end{equation*}
    which is finite for $N\ge 7$. More precisely, if $N\ge 7$, Proposition~\ref{prop:asymptotic_mass_supercritical} gives the strong convergence of $\eta_\omega$ towards $\eta_0$ in ${L^{\frac{2N}{N+2}}}(\rn)$ as $\omega$ goes to $0$. As a consequence,
    \begin{align*}
        \| (-\Delta+\omega)^{-\frac12}\eta_\omega-(-\Delta)^{-\frac12}\eta_0\|_{L^2}&\le\| (-\Delta+\omega)^{-\frac12}(\eta_\omega-\eta_0)\|_{L^2}+\| ((-\Delta+\omega)^{-\frac12}-(-\Delta)^{-\frac12})\eta_0\|_{L^2}\\
        &\lesssim \|\eta_\omega -\eta_0\|_{L^{\frac{2N}{N+2}}}+\left\|\left(\left(\frac{-\Delta}{-\Delta+\omega}\right)^{\frac12}-1\right)(-\Delta)^{-\frac12}\eta_0\right\|_{L^2}
    \end{align*}
    which tends to zeros. Hence, we can pass to the limit and obtain
    \begin{equation*}
        \lim_{\omega\to 0} M'(\omega)=-2\langle (-\Delta)^{-\frac12}\eta_0, (1+K_0)^{-1}(-\Delta)^{-\frac12}\eta_0\rangle
    \end{equation*}
    as expected. 
\end{proof}

To sum up, we finally gather the main results for the supercritical case.

\begin{proof}[Proof of Theorem~\ref{thm:asympt_omega_to_0}~(iii)]
The convergence of $u_\omega$ to $u_0$ is a consequence of Proposition~\ref{prop:conv_of_un}. 
The limits in \eqref{eq:supercrit_M_omega} follow from Proposition~\ref{prop:asymptotic_mass_supercritical}, 
while those in \eqref{eq:supercrit_M'_omega} follow from Proposition~\ref{prop:upper_bound_M'},
Proposition~\ref{prop:upper_bound_limit_supercritical} and Proposition~\ref{prop:limit_supercritical_M'}.
The proof of the supercritical case is complete.
\end{proof}


\subsection{Critical case}\label{sec:critical} 
We now consider the case $p=\frac{N+2}{N-2}$,
\emph{i.e.}~$p+1=2^*$. Combining the two identities from Proposition~\ref{prop:integral_identities} 
with $\omega=0$, we find that
\begin{equation}\label{eq:nehari_pohozaev}
2\intrn |u|^2|\nabla |u||^2\diff x=
\Big(1-\frac{2^*}{p+1}\Big)\intrn|u|^{p+1}\diff x
\end{equation}
for any solution $u\in \wt X$.
Hence, for $p+1=2^*$, the limit equation \eqref{eq:E_0} has no nontrivial
solution in $\dot H^1(\rn)$. In fact, we will see that, in the critical
case, a rescaled version of the solution $v_\omega$ converges,
as $\omega\to0$, to the unique (up to scaling) radial positive solution
of the critical Lane-Emden-Fowler equation
\begin{equation*}
\Delta w + w^\frac{N+2}{N-2} = 0.
\end{equation*}
On the formal level, considering a scaling parameter $\lambda_\omega>0$,
if $u>0$ is a solution of \eqref{eq:E_omega}, it follows that
$$
w(x):=\lambda_\omega^{2/(p-1)}u(\lambda_\omega x)
$$
satisfies
$$
\Delta w - \omega\lambda_\omega^2 w + w^p 
+ \lambda_\omega^{-4/(p-1)}\Delta(w^2)w = 0.
$$
Thus, choosing $\lambda_\omega\to\infty$ such that $\omega\lambda_\omega^2\to0$
yields the limit equation \eqref{eq:LEM}.

In order to establish rigorously the convergence of $u_\omega$ 
in the limit $\omega\to0$ we shall rather work with $z_\omega=h(u_\omega)(\sqrt{m_\omega}\cdot)$, 
as in Subsection~\ref{sec:supercritical}. We will compare, in the limit $\omega\to0$, the minimization problem 
$\inf_{z\in K_\omega} J_\omega(z)=m_\omega$ and its minimizers $z_\omega$
with the following one:
\begin{equation*}
m_*:=\inf_{\dot H^1\sm\{0\}}J_*,
\end{equation*}
where
\begin{equation}
J_*(w):=
\frac{\int |\nabla w|^2}{\left(\int|w|^{p+1}\right)^{\frac{2}{p+1}}}
=\frac{\int |\nabla w|^2}{\left(\int|w|^{2^*}\right)^{\frac{2}{2^*}}}
=\frac{\int |\nabla w|^2}{\left(\int|w|^{\frac{2N}{N-2}}\right)
^{1-\frac{2}{N}}}.
\end{equation}
$J_*$ is invariant under the scaling 
\begin{equation}\label{eq:H1_scaling_crit}
w(x)\mapsto w_\lambda(x):=\lambda^{-\frac{N-2}{2}}w(\lambda^{-1} x) \quad (\lambda>0).
\end{equation}
It is well-known \cite{talenti} that $m_*$ is attained on the family 
of positive radial functions
$$
\{W_\lambda\}_{\lambda>0}\subset \dot H^1(\rn), \quad 
W_1(x)\equiv W(x):=U(\sqrt{m_*} x),
$$
where $U$ is the Aubin-Talenti function defined in \eqref{eq:aubin_talenti_fct}.
The following properties of the family of minimizers $\{W_\lambda\}_{\lambda>0}$
are noteworthy:
\begin{equation}\label{eq:lane_emden_fowler}
-\Delta W_\lambda=m_* W_\lambda^\frac{N+2}{N-2} \quad \text{in} \ \rn,
\end{equation}
\begin{equation}\label{eq:Wscaling}
\intrn |\nabla W_\lambda|^2 \diff x = \intrn |\nabla W|^2\diff x = m_*
\quad\text{and}\quad
\intrn W_\lambda^\frac{2N}{N-2}\diff x = \intrn W^\frac{2N}{N-2} \diff x = 1,
\quad \forall\lambda>0.
\end{equation}

\begin{proposition}\label{prop:conv_of_vn'}
Suppose $p=\frac{N+2}{N-2}$.
There exists a scaling function $\omega\mapsto\lambda_\omega:(0,\infty)\to(0,\infty)$ with the following properties.
For any sequence $(\omega_n)\subset (0,\infty)$ such that $\omega_n\to0$, letting
\begin{equation*}
\lambda_n=\lambda_{\omega_n}, \quad m_n=m_{\omega_n}, \quad z_n=z_{\omega_n}, 
\end{equation*}
and 
\begin{equation}\label{eq:def_rescaled_min}
w_n:=\lambda_n^\frac{N-2}{2}z_n(\lambda_n \cdot)
\end{equation}
we have
\begin{equation}
w_n\to W \quad\text{in} \ 
\dot{H}^1(\rn)\cap L^{2^*}(\rn), \quad\text{as} \ n\to\infty.
\end{equation}
\end{proposition}

The proof of Proposition~\ref{prop:conv_of_vn'} will 
use several lemmas. The following 
bounds are direct consequences of Lemma \ref{propofr.lem} and \eqref{eq:Wscaling}.

\begin{lemma}\label{lem:lambda-apriori}
    For any $\lambda >0$ and $W_\lambda$ defined as in \eqref{eq:lane_emden_fowler}, we have
    \begin{equation}\label{eq:as_int_Wp}
                \intrn r(W_\lambda(x))^{2^*}\diff x \ge \left(\frac{1}{1+2\delta \lambda^{-(N-2)}}\right)^{\frac{2^*}{2}}
    \end{equation}
    and
    \begin{equation}\label{eq:as_int_W2}
        \intrn r(W_\lambda(x))^2\diff x \le 4 \lambda^{2}\intrn W(x)^2\diff x. 
    \end{equation}
\end{lemma}

We next establish the lemma corresponding to Lemma~\ref{lem:omega-asymptotics}
in the supercritical case.

\begin{lemma}\label{lem:omega-asymptotics_critical}
As $\omega\to0$, there holds
\begin{equation}\label{eq:m_omega_to_m*}
    1\le\frac{m_\omega}{m_*}\le 1+o(1).
\end{equation}
\end{lemma}

\begin{proof}
First, since $z_\omega\in \dot H^1$ satisfies
\begin{equation*}
\intrn |\nabla z_\omega|^2\diff x=m_\omega
\end{equation*}
and
\begin{equation*}
\frac{1}{2^*}=\intrn F_\omega(z_\omega)\diff x
=\frac{1}{2^*}\intrn r(z_\omega)^{2^*}\diff x-\frac{\omega}{2}\intrn r(z_\omega)^2\diff x
<\frac{1}{2^*}\intrn r(z_\omega)^{2^*}\diff x\le \frac{1}{2^*}\intrn z_\omega^{2^*}\diff x,
\end{equation*}
we obtain
\begin{equation}\label{eq:squeeze}
m_*\le J_*(z_\omega)
=\frac{\int |\nabla z_\omega|^2}{\left(\int z_\omega^{2^*}\right)^{\frac{2}{2^*}}}
=\frac{m_\omega}{\left(\int z_\omega^{2^*}\right)^{\frac{2}{2^*}}}
<m_\omega.
\end{equation}
Now, if $N\ge5$, we have $W\in L^2$ and we can use the family 
$\{W_\lambda : \lambda>0\}\subset H^1$ as test functions
for $J_\omega$. By Lemma~\ref{lem:lambda-apriori}, we see that, if $\lambda\gg 1$ and $\omega\lambda^2\ll1$,
$$
\intrn F_\omega (W_\lambda)\diff x=
\frac{1}{2^*}\intrn r(W_\lambda)^{2^*}\diff x-\frac{\omega}{2}\intrn r(W_\lambda)^2\diff x>0,
$$
\emph{i.e.}~$W_\lambda\in K_\omega$.
Choosing $\lambda=\omega^{-1/{N}}$, it follows from Lemma~\ref{lem:lambda-apriori} 
that
\begin{align*}
m_\omega \le 
J_\omega(W_\lambda)
&= 
\frac{\int|\nabla W|^2}{\left(\int W^{2^*}\right)^{\frac{2}{2^*}}}\cdot
\left(
\frac{\int W^{2^*}}
{\int r(W_\lambda)^{2^*}-\frac{2^*}{2}\omega\int r(W_\lambda)^2}\right)
^{\frac{2}{2^*}} \\
&=m_*\left(\int r(W_\lambda)^{2^*}-\frac{2^*}{2}\omega\int r(W_\lambda)^2\right)^{-\frac{2}{2^*}} \\
&\le m_*\left(\left(\frac{1}{1+2\delta \lambda^{-(N-2)}}\right)^{\frac{2^*}{2}}-2\cdot2^{*}\omega \lambda^2\int W^2\right)^{-\frac{2}{2^*}}\\
&= m_*\left(1+ O(\lambda^{-(N-2)}) + O(\omega^{1-\frac{2}{N}})\right)^{-\frac{2}{2^*}}\\
&= m_*\left(1+ O(\omega^{1-\frac{2}{N}})\right)^{-\frac{2}{2^*}}\\
&=m_*(1+O(\omega^{1-\frac{2}{N}})), \quad\text{as} \ \omega\to 0.
\end{align*}
We deduce that
\begin{equation*}
1\le \frac{m_\omega}{m_*} \le 1 + O(\omega^{1-\frac{2}{N}}),
\quad\text{as} \ \omega\to 0,
\end{equation*}
which concludes the proof in case $N\ge 5$.

For $N\in\{3,4\}$, we take $R\gg\lambda$ and
use the same cut-off function $\eta_R$ as in the proof
of Lemma~\ref{lem:omega-asymptotics}. Then $W_\lambda\not\in L^2$ but
$\eta_R W_\lambda\in H^1$ can be used
as a test function for $J_\omega$. We have
\begin{equation*}
m_\omega\le J_\omega(\eta_R W_\lambda)
=\frac{\int |\nabla \eta_R W_\lambda|^2}
{\left(\int |\eta_R W_\lambda|^{2^*}\right)^{\frac{2}{2^*}}}
\cdot 
\left(\frac{\int (\eta_R W_\lambda)^{2^*}}
{\int r(\eta_R W_\lambda)^{2^*}
-\frac{2^*}{2}\omega\int r(\eta_R W_\lambda)^2}\right)^{\frac{2}{2^*}}.
\end{equation*}
First, by dominated convergence, $\int |\nabla \eta_R W_\lambda|^2\to \int |\nabla W_\lambda|^2=m^*$ as $R\to\infty$.

Next, similarly to \eqref{eq:as_int_Wp}, we obtain 
\begin{equation*}\label{eq:cutoff-estimate1'}
\intrn r(\eta_R W_\lambda)^{2^*}\diff x \ge \left(\frac{1}{1+2\delta \lambda^{-(N-2)}}\right)^{\frac{2^*}{2}}\intrn (\eta_{R}W_\lambda)^{2^*}\diff x.
\end{equation*}
Furthermore, by \eqref{eq:Wscaling},
\begin{equation*}\label{eq:cutoff-estimate1''}
\intrn (\eta_R W_\lambda)^{2^*}\diff x= \intrn W_\lambda^{2^*}\diff x
+ O\Big(\frac{1}{(R/\lambda)^{N}}\Big)
= 1 + O\Big(\frac{1}{(R/\lambda)^{N}}\Big), \quad R/\lambda\gg1.
\end{equation*}
Using Lemma~\ref{lem:lowdimensions_estimates} and proceeding as in \eqref{eq:as_int_W2}, we have
\begin{equation*}\label{eq:cutoff-estimate2'}
g_N(R,\lambda):=\intrn r(\eta_R W_\lambda)^2\diff x \le 4 \intrn (\eta_R W_\lambda)^2\diff x
=
\begin{cases}
O(\lambda R) & \text{if} \ N=3, \\
O(\lambda^2\log(R/\lambda)) & \text{if} \ N=4,
\end{cases}
\quad R/\lambda\gg1.
\end{equation*}
Now
\begin{align*}
\frac{\int (\eta_R W_\lambda)^{p+1}}
{\int r(\eta_R W_\lambda)^{p+1}
-\frac{2^*}{2}\omega\int r(\eta_R W_\lambda)^2}
&\le 
\frac{1+O\left(\frac{1}{(R/\lambda)^{N}}\right)}
{\left(\frac{1}{1+2\delta \lambda^{-(N-2)}}\right)^{\frac{2^*}{2}}\left(1+ O\Big(\frac{1}{(R/\lambda)^{N}}\Big)\right)
-\frac{2^*}{2}\omega g_N(R,\lambda)} \\
&=\frac{1+O\Big(\frac{1}{(R/\lambda)^{N}}\Big)}
{1+O\left(\lambda^{-(N-2)}\right) + O\Big(\frac{1}{(R/\lambda)^{N}}\Big)-\frac{2^*}{2}\omega g_N(R,\lambda)}.
\end{align*}
We now conclude the proof in the following way. Let $R=\lambda^\alpha$ with $\alpha>1$ to be chosen later.

If $N=3$, we let $\lambda=\omega^{-1/4}$, so that
$$
\frac{1}{(R/\lambda)^{N}}=\frac{1}{\lambda^{3(\alpha-1)}}=\omega^{\frac34(\alpha-1)}, \quad
\frac{1}{\lambda^{N-2}}=\omega^{\frac14}
$$
and
$$
\omega g_3(\lambda^\alpha,\lambda)=O(\omega\lambda^{\alpha+1})=O(\omega^{1-\frac14(\alpha+1)}),
$$
provided that  $1<\alpha<3$. This implies, by taking $\frac43 \le \alpha\le 2$, 
$$
\frac{\int (\eta_R W_\lambda)^{p+1}}
{\int r(\eta_R W_\lambda)^{p+1}
-\frac{2^*}{2}\omega\int r(\eta_R W_\lambda)^2}
\le 1+O(\omega^{\frac14})+O(\omega^{\frac34(\alpha-1)})+O(\omega^{1-\frac14(\alpha+1)})=1+O(\omega^{\frac14}).
$$ 
If $N=4$, we let $\lambda=\omega^{-1/6}$, so that
$$
\frac{1}{(R/\lambda)^{N}}=\frac{1}{\lambda^{4(\alpha-1)}}=\omega^{\frac23(\alpha-1)}, \quad
\frac{1}{\lambda^{N-2}}=\omega^{\frac13}
$$
and
$$
\omega g_4(\lambda^\alpha,\lambda)=O(\omega\lambda^2\log\lambda)
=O\Big(\omega^{\frac23}\log\frac{1}{\omega}\Big)=o(\omega^{\frac13}).
$$
This implies, by taking $\alpha\ge \frac32$,
$$
\frac{\int (\eta_R W_\lambda)^{p+1}}
{\int r(\eta_R W_\lambda)^{p+1}
-\frac{2^*}{2}\omega\int r(\eta_R W_\lambda)^2}
\le 1+O(\omega^{\frac 13})+ O(\omega^{\frac 23(\alpha-1)})=1+O(\omega^{\frac 13}).
$$
The proof is complete.
\end{proof}

\begin{remark}\label{rem:matching_bounds_1} 
Observe from the proof of Lemma~\ref{lem:omega-asymptotics_critical} that the rate of decay of the remainder in
the right-hand side of \eqref{eq:m_omega_to_m*} depends on the dimension $N$. As will be discussed in detail
below (see Remark~\ref{rem:matching_bounds_2}), the decay rate obtained here for $N=4$ is not optimal.
\end{remark}

We can now determine the behavior of $z_\omega$ as $\omega\to0$.

\begin{lemma}\label{lem:omega_z_to_0_crit}
Consider a sequence $(\omega_n)\subset (0,\infty)$, such that $\omega_n\to0$. 
For all $n\in\N$, let $z_n:=z_{\omega_n}$ be a minimizer for \eqref{m_omega} with $\omega=\omega_n$. 
Then, as $n\to\infty$,
\begin{equation}\label{eq:L2-asymptotics_z_crit}
\|\nabla z_n\|_{L^2}^2 \to m_*=\|\nabla W_1\|_{L^2}^2, \quad 
\omega_n\|r(z_n)\|_{L^2}^2\to 0, \quad
\|r(z_n)\|_{L^{2^*}}^{2^*}\to 1
\end{equation}
and
\begin{equation}\label{eq:L2-asymptotics_z_crit_L_p+1}
\|z_n\|_{L^{2^*}}^{2^*}\to 1.
\end{equation}
\end{lemma}

\begin{proof}
First, the limits in \eqref{eq:L2-asymptotics_z_crit} are proved in the same way as
Lemma~\ref{lem:omega_z_to_0}, using Lemma~\ref{lem:omega-asymptotics_critical}
instead of Lemma~\ref{lem:omega-asymptotics}.
Next, \eqref{eq:L2-asymptotics_z_crit_L_p+1} follows directly 
from \eqref{eq:m_omega_to_m*} and \eqref{eq:squeeze}.
\end{proof}

We can now give the

\begin{proof}[Proof of Proposition~\ref{prop:conv_of_vn'}]
By Lemma~\ref{lem:omega_z_to_0_crit}, the radial sequence
$\{z_n\}\subset\dot{H}^1(\rn)\cap L^{2^*}(\rn)$ satisfies
\begin{equation*}
\|\nabla z_n\|_{L^2}^2 \to m_*, \quad
\|z_n\|_{L^{2^*}}^{2^*}\to 1, \quad n\to\infty.
\end{equation*}
Then, by the concentration-compactness principle (see the original paper \cite{lions}, or \cite{willem}
for a more concise exposition), 
there exists a scaling sequence $(\lambda_n)$ such that
\begin{equation*}
w_n:=\lambda_n^\frac{N-2}{2}z_n(\lambda_n \cdot)
\to W \quad\text{in}\quad \dot{H}^1(\rn)\cap L^{2^*}(\rn), \quad n\to\infty.
\end{equation*}
\end{proof}


\subsubsection{Further convergence properties of the solutions}

We first prove the upper bounds in \eqref{eq:lower-upper_bounds}. 

\begin{proposition}\label{prop:upper_bounds}
The scaling function $\omega\mapsto\lambda_\omega:(0,\infty)\to(0,\infty)$ can be chosen so that,
as $\omega\to 0$,
\begin{equation}\label{eq:upper_bounds}
\begin{cases}
\lambda_\omega \lesssim \omega^{-\frac38} & \text{if } N=3, \\
\lambda_\omega \lesssim \omega^{-\frac13}  & \text{if } N=4, \\
\lambda_\omega \lesssim \omega^{-\frac 1N} & \text{if }  N\ge 5.
\end{cases}
\end{equation}
\end{proposition}

\begin{proof}
Consider $\delta_\omega:=m_\omega-m_*$. By the proof of Lemma~\ref{lem:omega-asymptotics_critical}, 
we have
\begin{equation}\label{eq:delta_omega}
0\le \delta_\omega \lesssim
\begin{cases}
\omega^{\frac14} & \text{if } N=3, \\
\omega^{\frac13}  & \text{if } N=4, \\
\omega^{1-\frac 2N} & \text{if } N\ge 5.
\end{cases}
\end{equation}
The upper bounds \eqref{eq:upper_bounds} follow directly from \eqref{eq:delta_omega} and the following estimates:
\begin{equation}\label{eq:upper_bound_est_1}
\omega\lambda_\omega^2\|w_\omega\|_{L^2}^2\lesssim \delta_\omega, 
\end{equation}
\begin{equation}\label{eq:upper_bound_est_2}
\|w_\omega\|_{L^2}\gtrsim 1.
\end{equation}
To prove \eqref{eq:upper_bound_est_1}, 
we first observe that
\begin{equation}\label{eq:L^2_w_and_z}
\lambda_\omega^2\|w_\omega\|_{L^2}^2=\|z_\omega\|_{L^2}^2.
\end{equation}
By Lemma~\ref{lem:uniform_bound_z}~(i), there exists $\eta_0>0$ and $M_0>0$ such that
\begin{equation*}
\|z_\omega\|_{L^\infty}\le M_0, \quad  0<\omega<\eta_0.
\end{equation*}
Hence,
$$
r(z_\omega)^2\le z_\omega^2 \le M_0^2
$$
and it follows by Lemma~\ref{propofr.lem} that
$$
\intrn r(z_\omega)^2\diff x\ge \intrn \frac{z_\omega^2}{1+2\delta r(z_\omega)^2}\diff x 
\ge \frac{1}{1+2\delta M_0^2} \intrn z_\omega^2 \diff x.
$$
Thus,
$$
\|z_\omega\|_{L^2}^2 \le (1+2\delta M_0^2)\|r(z_\omega)\|_{L^2}^2
$$
and \eqref{eq:L^2_w_and_z} yields
$$
\omega\lambda_\omega^2\|w_\omega\|_{L^2}^2\lesssim \omega \|r(z_\omega)\|_{L^2}^2.
$$
The proof of \eqref{eq:upper_bound_est_1} will be complete if we show that
\begin{equation}\label{eq:upper_bound_claim}
\omega \|r(z_\omega)\|_{L^2}^2 \lesssim \delta_\omega.
\end{equation}
To this end, we use $z_\omega$ as a test function for $J_*$:
$$
m_*\le J_*(z_\omega)=\frac{\int|\nabla z_\omega|^2}{\|z_\omega\|_{L^{2^*}}^2}
=\frac{m_\omega}{(\int z_\omega^{2^*})^{2/2^*}}
\le \frac{m_\omega}{(\int r(z_\omega)^{2^*})^{2/2^*}}.
$$
Since
\begin{equation*}
1=\intrn r(z_\omega)^{2^*}\diff x-\frac{2^*}{2}\omega\intrn r(z_\omega)^2\diff x,
\end{equation*}
we deduce that
$$
m_*\le \frac{m_\omega}{(1+\frac{2^*}{2}\omega\int r(z_\omega)^2)^{2/2^*}}.
$$
It follows by a first order Taylor expansion that
$$
\frac{2^*}{2}\omega\intrn r(z_\omega)^2\diff x \le m_\omega^\frac{2^*}{2}-m_*^\frac{2^*}{2}
=\frac{2^*}{2}m_\omega^{\frac{2^*}{2}-1}(m_\omega-m_*)+o(m_\omega-m_*).
$$
This proves \eqref{eq:upper_bound_claim} and completes the proof of \eqref{eq:upper_bound_est_1}.

To prove \eqref{eq:upper_bound_est_2}, we observe that $w_\omega\to W$ in $\dot H^1(\rn)$ as $\omega\to0$
implies $w_\omega\to W$ in $L^{2^*}(\rn)$ and in $L^2_\mathrm{loc}(\rn)$. Since
$$
\|\chi_{B_1(0)}w_\omega\|_{L^2}
\ge \|\chi_{B_1(0)}W\|_{L^2}-\|\chi_{B_1(0)}(W-w_\omega)\|_{L^2},
$$
it follows that
$$
\|w_\omega\|_{L^2} \ge \|\chi_{B_1(0)}W\|_{L^2}+o(1),
\quad \omega\to0.
$$
This proves \eqref{eq:upper_bound_est_2} and completes the proof of the proposition.
\end{proof}

We next turn to the lower bounds in \eqref{eq:lower-upper_bounds}, which are more involved.

\begin{proposition}\label{prop:lower_bounds}
The scaling function $\omega\mapsto\lambda_\omega:(0,\infty)\to(0,\infty)$ can be chosen so that,
as $\omega\to 0$,
\begin{equation}\label{eq:lower_bounds}
\begin{cases}
\lambda_\omega \gtrsim \omega^{-\frac14 }  & \text{if } N=3, \\
\lambda_\omega \gtrsim \omega^{-\frac16}   & \text{if } N=4, \\
\lambda_\omega \gtrsim \omega^{-\frac 1N}  & \text{if }  N\ge 5.
\end{cases}
\end{equation}
\end{proposition}

\begin{remark}\label{rem:matching_bounds_2} 
For $N\ge 5$, the upper bounds \eqref{eq:upper_bounds} and the lower bounds \eqref{eq:lower_bounds}
come with the same power of $\omega^{-1}$. For $N=3,4$, they do not match. This is due to the
rough estimate \eqref{eq:upper_bound_est_2} used to derive the upper bounds from \eqref{eq:upper_bound_est_1}
and also, in case $N=4$, to the non-optimal decay of $\delta_\omega$ that was already pointed out in Remark~\ref{rem:matching_bounds_1}.
For $N=3,4$, estimate \eqref{eq:upper_bound_est_2} will be improved in Lemma~\ref{lem:L^2_lower_bound_w}
and the optimal lower bounds on $\lambda_\omega$ will be given 
in Proposition~\ref{prop:tight_upper_bounds_N3}
and Proposition~\ref{prop:tight_upper_bounds_N4}, respectively.
\end{remark}

The proof of Proposition~\ref{prop:lower_bounds} will use the following lemmas.

\begin{lemma}\label{lem:lambda_to_infty}
$$\lim_{\omega\to0}\lambda_\omega=\infty.$$
\end{lemma}

\begin{proof}
In this proof we use the Nehari and Pohozaev integral identities derived from the equation
\begin{equation}\label{eq:rescaled_min}
-\Delta w_\omega=\lambda_\omega^{1+\frac{N}{2}}m_\omega f_\omega(\lambda_\omega^{-\frac{N-2}{2}}w_\omega),
\end{equation}
which respectively read
\begin{multline*}
\intrn r'\big(\lambda_\omega^{-\frac{N-2}{2}}w_\omega\big)^2\Big(1+2\delta r\big(\lambda_\omega^{-\frac{N-2}{2}}w_\omega\big)^2\Big)|\nabla w_\omega|^2\diff x \\
=\lambda_\omega^N {m_\omega}\intrn r\big(\lambda_\omega^{-\frac{N-2}{2}}w_\omega\big)^{2^*}\diff x
-\frac{2^*}{2}\omega\, {m_\omega}\lambda_\omega^N\intrn r\big(\lambda_\omega^{-\frac{N-2}{2}}w_\omega\big)^2\diff x
\end{multline*}
and
\begin{multline*}
\intrn r'\big(\lambda_\omega^{-\frac{N-2}{2}}w_\omega\big)^2\Big(1+4\delta r\big(\lambda_\omega^{-\frac{N-2}{2}}w_\omega\big)^2\Big)|\nabla w_\omega|^2\diff x \\
=\lambda_\omega^N {m_\omega}\intrn r\big(\lambda_\omega^{-\frac{N-2}{2}}w_\omega\big)^{2^*}\diff x
-\omega\,{m_\omega}\lambda_\omega^N\intrn r\big(\lambda_\omega^{-\frac{N-2}{2}}w_\omega\big)^2\diff x.
\end{multline*}
We deduce from these identities that
\begin{equation}\label{eq:Nehari-Pohozaev_for_zn}
\delta \intrn r'\big(\lambda_\omega^{-\frac{N-2}{2}}w_\omega\big)^2r\big(\lambda_\omega^{-\frac{N-2}{2}}w_\omega\big)^2|\nabla w_\omega|^2\diff x
=\frac{1}{N-2}\omega\,{m_\omega}\lambda_\omega^N\intrn r\big(\lambda_\omega^{-\frac{N-2}{2}}w_\omega\big)^2\diff x.
\end{equation}

Suppose by contradiction there exist $\ell\in[0,\infty)$ and a sequence $(\omega_n)\subset(0,\infty)$ such that 
$\omega_n\to 0$ and $\lambda_{\omega_n}\to\ell$, as $n\to\infty$. For the remainder of the proof,
we shall again abbreviate the notation as $z_n:=z_{\omega_n}$, $w_n:=w_{\omega_n}$, $\lambda_n:=\lambda_{\omega_n}$ and $m_n:=m_{\omega_n}$.

First suppose that $\ell=0$. From Lemma~\ref{lem:uniform_bound_z}~(i), we know there exists
$n_0\in\N$ such that $\|z_n\|_{L^\infty}\le M_0$ for all $n\ge n_0$.
It follows that
$$
\big|\lambda_n^\frac{N-2}{2}z_n(\lambda_nx)\big|\le M_0\lambda_n^\frac{N-2}{2} \to 0, \quad \forall x\in\rn,
$$
whence $w_n\to 0$ pointwise as $n\to\infty$, which contradicts $w_n\to W$.

Now suppose that $\ell\in(0,\infty)$. 
It follows from Fatou's Lemma that
\begin{equation}\label{eq:Fatou}
\intrn r'\big(\ell^{-\frac{N-2}{2}}W\big)^2r\big(\ell^{-\frac{N-2}{2}}W\big)^2|\nabla W|^2\diff x
\le \liminf_{n\to\infty} 
\intrn r'\big(\lambda_n^{-\frac{N-2}{2}}w_n\big)^2r\big(\lambda_n^{-\frac{N-2}{2}}w_n\big)^2|\nabla w_n|^2\diff x.
\end{equation}
On the other hand, by Lemma~\ref{lem:omega-asymptotics_critical} and \eqref{eq:upper_bound_claim},
\begin{align}\label{eq:kill_Fatou}
\omega_n {m_n}\lambda_n^N\intrn r\big(\lambda_n^{-\frac{N-2}{2}}w_n(x)\big)^2\diff x
&={m_n} \omega_n\lambda_n^N\intrn r\big(z_n(\lambda_n x)\big)^2\diff x \nonumber\\
&={m_n} \omega_n\intrn r\big(z_n(y)\big)^2\diff y \to 0, \quad n\to\infty.
\end{align}
It follows by \eqref{eq:Nehari-Pohozaev_for_zn}, \eqref{eq:Fatou} and \eqref{eq:kill_Fatou} that
$$
\intrn r'\big(\ell^{-\frac{N-2}{2}}W\big)^2r\big(\ell^{-\frac{N-2}{2}}W\big)^2|\nabla W|^2\diff x=0.
$$
This contradiction shows that, indeed, $\lambda_\omega\to\infty$ as $\omega\to0$.
\end{proof}

\begin{lemma}\label{lem:u_omega_L^2_w_omega_L^2}
As $\omega\to0$, there holds
\begin{equation*}
\lambda_\omega^2\|w_\omega\|_{L^2}^2
\lesssim
\|u_\omega\|_{L^2}^2 
\lesssim
\lambda_\omega^2\|w_\omega\|_{L^2}^2.
\end{equation*}
\end{lemma}

\begin{proof}
We have
\begin{align*}
\frac{1}{\lambda_\omega^2}\|u_\omega\|_{L^2}^2
=\frac{1}{\lambda_\omega^2}\|r(v_\omega)\|_{L^2}^2
=\frac{1}{\lambda_\omega^2}m_\omega^{\frac{N}{2}}\intrn r(z_\omega(x))^2\diff x 
=\frac{1}{\lambda_\omega^2}m_\omega^{\frac{N}{2}}\intrn z_\omega(x)^2 g_\omega(x)\diff x,
\end{align*}
where
$$
g_\omega(x)=\frac{r(z_\omega(x))^2}{z_\omega(x)^2}.
$$
Using Lemma \ref{propofr.lem}~(iv), we deduce that, for all $x\in\rn$,
\begin{align*}
    r'(z_\omega(x))^2\le g_\omega(x)\le  4 r'(z_\omega(x))^2
\end{align*}
with 
\begin{equation*}
    r'(z_\omega(x))^2=\frac{1}{1+2\delta r(z_\omega(x))^2}.
\end{equation*}
Thanks to Lemma \ref{eq:uniform_bound_z}, for all $\omega \in (0,\eta_0)$, 
$\|r(z_\omega)\|_{L^\infty}\le \|z_\omega\|_{L^\infty}\le M_0$, so that 
\begin{equation}\label{eq:ineq_for_g}
    \frac{1}{1+2\delta M_0^2}\le r'(z_\omega(x))^2\le g_\omega(x)\le 4, \quad x\in\rn.
\end{equation}
As a consequence, since $m_\omega\ge m_*$ by Lemma~\ref{lem:omega-asymptotics_critical},
\begin{align*}
    \frac{1}{\lambda_\omega^2}\|u_\omega\|_{L^2}^2\ge \frac{m_*^{\frac{N}{2}}}{1+2\delta M_0^2}\frac{1}{\lambda_\omega^2}\intrn z_\omega(x)^2\diff x
    =\frac{m_*^{\frac{N}{2}}}{1+2\delta M_0^2}\intrn w_\omega(x)^2\diff x.
\end{align*}
On the other hand, using again Lemma \ref{lem:omega-asymptotics_critical} we obtain,
for $\omega$ small enough, 
\begin{align*}
    \frac{1}{\lambda_\omega^2}\|u_\omega\|_{L^2}^2\le 4(m_*+1)^{\frac{N}{2}}\frac{1}{\lambda_\omega^2}\intrn z_\omega(x)^2\diff x=4(m_*+1)^{\frac{N}{2}}\intrn w_\omega(x)^2\diff x.
\end{align*}
\end{proof}

\begin{lemma}\label{lem:pohozaev_w_omega}
As $\omega\to 0$, there holds
\begin{equation}\label{eq:pohozaev_w_omega}
\lambda_\omega^{-(N-2)}\intrn w_\omega^2|\nabla w_\omega|^2\diff x
\lesssim \omega\lambda_\omega^2\|w_\omega\|_{L^2}^2.
\end{equation}
\end{lemma}

\begin{proof}
Firstly, by \eqref{eq:pohozaev} and \eqref{eq:nehari}, 
we deduce using $p+1=2^*=\frac{2N}{N-2}$ that
\begin{equation}\label{eq:key_estimate}
\delta \intrn u_\omega^2|\nabla u_\omega|^2\diff x
=\frac{1}{N-2}\omega\|u_\omega\|_{L^2}^2.
\end{equation}
We will now exploit this identity using the relations
\[
u_\omega(x)=r\big(v_\omega(x)\big)=r\big(z_\omega(m_\omega^{-1/2}x)\big)
=r\big(\lambda_\omega^{-\frac{N-2}{2}}w_\omega(\lambda_\omega^{-1} m_\omega^{-1/2}x)\big).
\]
We have
\[
\nabla u_\omega(x)
=m_\omega^{-1/2}\lambda_\omega^{-N/2}
r'\big(\lambda_\omega^{-\frac{N-2}{2}}w_\omega(\lambda_\omega^{-1} m_\omega^{-1/2}x)\big)
\nabla w_\omega(\lambda_\omega^{-1} m_\omega^{-1/2}x).
\]
On the one hand, by \eqref{eq:ineq_for_g}, there holds, for all $\omega\in (0,\eta_0)$:
\begin{align*}
\intrn u_\omega^2|\nabla u_\omega|^2\diff x
&=m_\omega^{-1}\lambda_\omega^{-N}
\intrn r\big(z_\omega(m_\omega^{-1/2}x)\big)^2
r'\big(z_\omega(m_\omega^{-1/2}x)\big)^2
|\nabla w_\omega(\lambda_\omega^{-1} m_\omega^{-1/2}x)|^2\diff x \\
&=m_\omega^{\frac N2-1}\lambda_\omega^{-N}
\intrn r\big(z_\omega(y)\big)^2
r'\big(z_\omega(y)\big)^2
|\nabla w_\omega(\lambda_\omega^{-1}y)|^2\diff y \\
&=m_\omega^{\frac N2-1}\lambda_\omega^{-N}
\intrn z_\omega(y)^2g_\omega(y)
r'\big(z_\omega(y)\big)^2
|\nabla w_\omega(\lambda_\omega^{-1}y)|^2\diff y \\
&\ge \frac{1}{(1+2\delta M_0^2)^2} m_\omega^{\frac N2-1}\lambda_\omega^{-N}
\intrn z_\omega(y)^2
|\nabla w_\omega(\lambda_\omega^{-1}y)|^2\diff y \\
&= \frac{1}{(1+2\delta M_0^2)^2} m_\omega^{\frac N2-1}\lambda_\omega^{-N}
\intrn \lambda_\omega^{-(N-2)}w_\omega(\lambda_\omega^{-1}y)^2
|\nabla w_\omega(\lambda_\omega^{-1}y)|^2\diff y \\
&= \frac{1}{(1+2\delta M_0^2)^2} m_\omega^{\frac N2-1}
\intrn \lambda_\omega^{-(N-2)}w_\omega(x)^2
|\nabla w_\omega(x)|^2\diff x.
\end{align*}
Since $m_\omega\ge m_*$ by Lemma~\ref{lem:omega-asymptotics_critical}, we deduce that
\[
\intrn u_\omega^2|\nabla u_\omega|^2\diff x
\gtrsim 
\lambda_\omega^{-(N-2)}
\intrn w_\omega^2|\nabla w_\omega|^2\diff x.
\]
On the other hand, by Lemma~\ref{lem:u_omega_L^2_w_omega_L^2}, we have
that
\[
\frac{1}{N-2}\omega\|u_\omega\|_{L^2}^2
\lesssim
\omega\lambda_\omega^2\|w_\omega\|_{L^2}^2, \quad \omega\to0.
\]
Hence, the result follows by \eqref{eq:key_estimate}.
\end{proof}

\begin{lemma}\label{lem:GN_inequality}
Consider $2\le s<\frac{4N}{N-2}$, $s<q<\frac{4N}{N-2}$ and let $\theta=\frac{(N-2)(q-s)}{2s+N(4-s)}$.
Then there exists a constant 
$C=C(N,s,q)>0$ such that
\begin{equation}\label{eq:GN_inequality}
\forall u\in C_0^\infty(\rn), \quad
\intrn |u|^q\diff x 
\le  
C\Big(\intrn |u|^2|\nabla u|^2\diff x\Big)^{\theta \frac{N-2}{N}}
\Big(\intrn |u|^s\diff x\Big)^{1-\theta}.
\end{equation}
\end{lemma}

\begin{proof}
Apply the classical Gagliardo-Nirenberg inequality (see \cite{nirenberg}) to $|u|^2$.
\end{proof}

We are now in a position to prove Proposition~\ref{prop:lower_bounds}.

\begin{proof}[Proof of Proposition~\ref{prop:lower_bounds}]
Extending \eqref{eq:GN_inequality} to $X$ by density and applying it to
$w_\omega$ with exponents $q>2^*$ and $s=2^*$, we obtain
\begin{equation}\label{eq:applying_GN_1}
\intrn |w_\omega|^{q}\diff x 
\le C
\Big(\intrn |w_\omega|^{2}|\nabla w_\omega|^2\diff x\Big)^{\theta \frac{N-2}{N}}
\Big(\intrn |w_\omega|^{2^*}\diff x\Big)^{1-\theta}.
\end{equation}
Furthermore, the H\"older inequality yields a constant $C>0$ such that
\[
\|\chi_{B_1(0)}w_\omega\|_{L^q}\ge C\|\chi_{B_1(0)}w_\omega\|_{L^{2^*}}
\ge C\big(\|\chi_{B_1(0)}W\|_{L^{2^*}}-\|\chi_{B_1(0)}(W-w_\omega)\|_{L^{2^*}}\big).
\]
Since $w_\omega\to W$ in $L^{2^*}(\rn)$, it follows that
$$
\|w_\omega\|_{L^q} \ge \|\chi_{B_1(0)}w_\omega\|_{L^q}
\ge C\|\chi_{B_1(0)}W\|_{L^{2^*}}+o(1),
\quad \omega\to0.
$$
Therefore, \eqref{eq:applying_GN_1} shows that 
$\intrn |w_\omega|^{2}|\nabla w_\omega|^2\diff x$
is bounded away from zero.
Hence, by \eqref{eq:upper_bound_est_1} and \eqref{eq:pohozaev_w_omega},
\begin{equation}\label{eq:lower_bound_basic_est}
   \lambda_\omega^{-(N-2)}\lesssim \delta_\omega . 
\end{equation}
The conclusion now follows from \eqref{eq:delta_omega}.
\end{proof}

As was pointed out in Remark~\ref{rem:matching_bounds_2}, for all $N\ge 5$
the upper and lower bounds on $\lambda_\omega$ have the same blow-up rate. 
We shall now tighten the upper bounds on $\lambda_\omega$ for $N=3,4$.

\begin{lemma}\label{lem:lower_exp_decay}
For all $R>0$, 
there exist $\eta_1=\eta_1(R)\in(0,\eta_0)$ and  $C=C(R)>0$
such that
\begin{equation}\label{eq:lower_exp_decay} 
\forall \omega\in(0,\eta_1), \ \forall x\in\rn\sm B_R(0), \quad
|w_\omega(x)|\ge C \frac{e^{-\sqrt{m_\omega\omega}\lambda_\omega|x|}}{|x|^{N-2}}.
\end{equation}
\end{lemma}

\begin{proof}
    The rescaled minimizer $w_\omega$ defined in \eqref{eq:def_rescaled_min} solves the equation \eqref{eq:rescaled_min}. As a consequence, 
    \begin{align*}
        -\Delta w_\omega +m_{\omega}\omega \lambda^2_\omega w_\omega &= m_\omega \lambda_\omega^{1+\frac{N}{2}}f_\omega(\lambda_\omega^{-\frac{N-2}{2}}w_\omega)+m_{\omega}\omega \lambda^2_\omega w_\omega\\
        &\ge m_\omega \omega \lambda_\omega^2\left(-\lambda_\omega^{\frac{N-2}{2}}r'(\lambda_\omega^{-\frac{N-2}{2}}w_\omega)r(\lambda_\omega^{-\frac{N-2}{2}}w_\omega)+w_\omega\right)\\
        &=  m_\omega \omega \lambda_\omega^2 k(\lambda_\omega^{-\frac{N-2}{2}}w_\omega){w_\omega}
    \end{align*}
    with $k:(0,\infty)\to \R$ defined by $k(s)=1-\frac{r'(s)r(s)}{s}$. Since $0\le r'(s)\le 1$ and $0\le r(s)/s\le 1$ for all $s\in (0,\infty)$, we deduce 
     that $k(s)\ge 0$ for all $s\in (0,\infty)$. Hence,
    \begin{equation*}
        -\Delta w_\omega(x) +m_{\omega}\omega \lambda^2_\omega w_\omega(x)\ge 0
    \end{equation*}
    for all $x\in \R^N$. For all $x\in \R^N\setminus \{0\}$, let 
    \begin{equation*}
        h_\omega(x)=|x|^{-(N-2)}e^{-\sqrt{m_\omega\omega}\lambda_\omega |x|}.
    \end{equation*}
    A direct computation shows that, for all $R>0$,
    \begin{align*}
        -\Delta h_\omega(x) +m_{\omega}\omega \lambda^2_\omega h_\omega(x)=&\, |x|^{-(N-2)}\left(-\omega m_\omega \lambda^2_\omega+(N-1)|x|^{-1}\sqrt{m_\omega\omega}\lambda_\omega\right)e^{-\sqrt{m_\omega\omega}\lambda_\omega |x|}\\
        &-2(N-2)\sqrt{m_\omega\omega}\lambda_\omega|x|^{-(N-1)}e^{-\sqrt{m_\omega\omega}\lambda_\omega |x|}+m_\omega \omega\lambda_\omega^2 |x|^{-(N-2)}e^{-\sqrt{m_\omega\omega}\lambda_\omega |x|}\\
        =&\,(3-N)\sqrt{m_\omega\omega}\lambda_\omega|x|^{-(N-1)}e^{-\sqrt{m_\omega\omega}\lambda_\omega |x|}\le 0
    \end{align*}
    on $\R^N\setminus B_R(0)$, provided that $N\ge 3$.

    Next, thanks to Proposition~\ref{prop:conv_of_vn'} and Lemma~\ref{lem:radial_lemma}, we have, for all $R>0$, 
    \begin{equation*}
        \|(w_\omega-W)\chi_{B_{2R}(0)\setminus B_{R/2}(0)}\|_{L^{\infty}}\to 0 
    \end{equation*}
    as $\omega$ goes to $0$. As a consequence, 
    \begin{equation*}
        w_\omega(R)\ge \frac{1}{2} W(R) 
    \end{equation*}
    for all $\omega$ sufficiently small. Finally, we choose $C=C(R)>0$ such that 
    \begin{equation*}
        \frac{1}{2}W(R)\ge C R^{-(N-2)}.
    \end{equation*}
    This implies $w_\omega(R)\ge C h_\omega(R)$ and the conclusion follows by  
    applying the maximum principle to the function $w_\omega - C h_\omega$.
\end{proof}

As a consequence of~\eqref{eq:lower_exp_decay}, using the same arguments as in \cite[Lemma 4.9, Lemma 4.11]{MorMur-14}, we obtain the following.

\begin{lemma}\label{lem:L^2_lower_bound_w}
As $\omega\to0$, there holds
\begin{equation}\label{eq:L^2_lower_bound_w}
\|w_\omega\|_{L^2}^2\gtrsim \begin{cases}
    \omega^{-1/2}\lambda_\omega^{-1} & \text{if } N=3,\\
    \log\frac{1}{\sqrt{\omega}\lambda_\omega} & \text{if } N=4.
\end{cases}
\end{equation}
\end{lemma}


This is enough to conclude in the case $N=3$. 

\begin{proposition}\label{prop:tight_upper_bounds_N3}
Let $N=3$. As $\omega\to 0$, there holds
\begin{equation}\label{eq:tight_upper_bounds_N3}
\lambda_\omega \lesssim \omega^{-\frac14}.
\end{equation}
\end{proposition}

\begin{proof}
As in the proof of Proposition~\ref{prop:upper_bounds}, we use estimate
\eqref{eq:upper_bound_est_1} but now with the refined estimate \eqref{eq:L^2_lower_bound_w}
instead of \eqref{eq:upper_bound_est_2}. This leads to 
\begin{equation*}
    \omega \lambda^2_\omega \omega^{-1/2}\lambda_\omega^{-1}\lesssim\omega \lambda^2_\omega \|w_\omega\|^2_{L^2} \lesssim \delta_\omega
\end{equation*}
which implies
\[
\lambda_\omega\lesssim \omega^{-1/2}\delta_\omega.
\]
The result then follows from \eqref{eq:delta_omega}.
\end{proof}

The case $N=4$ is slightly more involved. Indeed, to obtain a matching lower and upper bound for $\lambda_\omega$ we need to improve the upper bound for $\delta_\omega$.

\begin{proposition}\label{prop:tight_upper_bounds_N4}
Let $N=4$. As $\omega\to 0$, there holds
\begin{equation}\label{eq:tight_upper_bounds_N4}
\delta_{\omega}\lesssim \left(\omega \log \frac{1}{\omega}\right)^{\frac12} \text{ and } \left(\omega \log \frac{1}{\omega}\right)^{-\frac14}\lesssim \lambda_\omega \lesssim \left(\omega \log \frac{1}{\omega}\right)^{-\frac14}.
\end{equation}
\end{proposition}

\begin{remark}
Heuristically, we know from the proof of Proposition~\ref{prop:lower_bounds} that, for $N=4$, 
\begin{equation*}
    \lambda_\omega^2 \gtrsim \delta_\omega^{-1}.
\end{equation*}    
On the other hand, estimates
\eqref{eq:upper_bound_est_1} and \eqref{eq:L^2_lower_bound_w} imply that 
\begin{align*}
    \delta_\omega \gtrsim \omega \lambda^2_\omega \|w_\omega\|^2_{L^2}\gtrsim \omega \lambda^2_\omega \log\frac{1}{\sqrt{\omega}\lambda_\omega}.
\end{align*}
Now, since $\omega^{\frac13}\lesssim \sqrt{\omega}\lambda_\omega \lesssim \omega^{\frac16}$, we deduce 
\begin{equation*}
    \log\frac{1}{\sqrt{\omega}\lambda_\omega}\gtrsim \log\frac{1}{\omega}
\end{equation*}
so that 
\begin{equation*}
    \delta_\omega \gtrsim \omega \delta_\omega^{-1}\log\frac{1}{\omega} \implies \delta_\omega^2 \gtrsim \omega \log\frac{1}{\omega}.
\end{equation*}
Hence, the bound \eqref{eq:delta_omega} can be improved in order to obtain a matching upper and lower bound for $\lambda_\omega$.
\end{remark}

\begin{proof}
    From the proof of Lemma~\ref{lem:omega-asymptotics_critical}, we know that, for $N=4$, 
    $$
        \frac{\int (\eta_R W_\lambda)^{2^*}}
        {\int r(\eta_R W_\lambda)^{2^*}
        -\frac{2^*}{2}\omega\int r(\eta_R W_\lambda)^2}
        \le 1+O((R/\lambda)^{-4})+ O(\lambda^{-2})+O(\omega\lambda^2\log(R/\lambda))
    $$
    provided that $\lambda \gg 1$, $\frac{R}{\lambda}\gg 1$ and $\omega\lambda^2\log(R/\lambda)\ll 1$. 
    Hence, we let $\lambda=\left(\omega\log\frac{1}{\omega}\right)^{-1/4}$ and $R=\lambda^\alpha$ with $\alpha>1$ to be chosen later. We have
$$
\frac{1}{(R/\lambda)^{N}}=\frac{1}{\lambda^{4(\alpha-1)}}=\left(\omega\log\frac{1}{\omega}\right)^{(\alpha-1)}, \quad
\frac{1}{\lambda^{2}}=\frac{1}{\lambda^{4(\alpha-1)}}=\left(\omega\log\frac{1}{\omega}\right)^{\frac12}
$$
and
$$
O(\omega\lambda^2\log(R/\lambda))=O(\omega\lambda^2\log\lambda)
= O\left(\left(\omega\log\frac{1}{\omega}\right)^{-1/2}\omega \log\frac{1}{\omega}\right)=O\left(\left(\omega\log\frac{1}{\omega}\right)^{1/2}\right).
$$
This implies, by taking $\alpha\ge \frac32$,
$$
\frac{\int (\eta_R W_\lambda)^{2^*}}
{\int r(\eta_R W_\lambda)^{2^*}
-\frac{2^*}{2}\omega\int r(\eta_R W_\lambda)^2}
\le 1+O\left(\left(\omega\log\frac{1}{\omega}\right)^{1/2}\right)
$$
so that
\begin{equation*}
    \delta_\omega \lesssim \left(\omega\log\frac{1}{\omega}\right)^{1/2}.
\end{equation*}
Arguing as in the proof of Proposition~\ref{prop:lower_bounds}, we get the lower bound
\begin{equation*}
    \lambda_\omega \gtrsim \delta_\omega^{-\frac12} \gtrsim \left(\omega\log\frac{1}{\omega}\right)^{-1/4}.
\end{equation*}  
On the other hand, estimates \eqref{eq:upper_bound_claim},
\eqref{eq:upper_bound_est_1} and \eqref{eq:L^2_lower_bound_w} imply that 
\begin{align*}
    \left(\omega\log\frac{1}{\omega}\right)^{1/2}\gtrsim \delta_\omega \gtrsim \omega \lambda^2_\omega \|w_\omega\|^2_{L^2}\gtrsim \omega \lambda^2_\omega \log\frac{1}{\sqrt{\omega}\lambda_\omega} \gtrsim \omega \lambda^2_\omega \log\frac{1}{\omega}
\end{align*}
which leads to 
\begin{equation*}
    \lambda_\omega \lesssim \left(\omega\log\frac{1}{\omega}\right)^{-1/4},
\end{equation*}
as claimed.
\end{proof}

\begin{lemma}\label{lem:supercritical_bounds}
For all $q\in(2^*,\frac{4N}{N-2})$, 
\begin{equation}
\|w_\omega\|_{L^q}\lesssim 1, \quad \omega\to0.
\end{equation}
\end{lemma}

\begin{proof}
We shall again use \eqref{eq:applying_GN_1}. We already know that $\|w_\omega\|_{L^{2^*}}$
is bounded. Furthermore, by \eqref{eq:delta_omega}, \eqref{eq:upper_bound_est_1},
\eqref{eq:pohozaev_w_omega}, \eqref{eq:tight_upper_bounds_N3} and \eqref{eq:tight_upper_bounds_N4}, we have
\begin{equation*}
\intrn |w_\omega|^{2}|\nabla w_\omega|^2\diff x
\lesssim \lambda_\omega^{N-2}\delta_\omega
\lesssim
\begin{cases}
\omega^{-\frac14}\omega^{\frac14}=1, & \text{if } N=3, \\
\left(\omega \log \frac{1}{\omega}\right)^{-\frac12}\left(\omega \log \frac{1}{\omega}\right)^{\frac12}=1, & \text{if } N=4, \\
\omega^{-\frac{N-2}{N}}\omega^{\frac{N-2}{N}}=1, & \text{if } N\ge 5,
\end{cases}
\end{equation*}
which completes the proof.
\end{proof}

As a consequence of Lemma~\ref{lem:supercritical_bounds}, the following $L^\infty$-bound 
can be proved similarly to Lemma~\ref{lem:uniform_bound_z}~(i), using
\eqref{eq:rescaled_min} instead of \eqref{eq:equ_for_zn}.

\begin{lemma}\label{lem:uniform_bound_w}
There exists $\eta_2>0$ and $M_2>0$ such that 
\begin{equation}\label{eq:uniform_bound_w}
\sup_{0<\omega<\eta_2}\|w_\omega\|_{L^\infty}\le M_2.
\end{equation}
\end{lemma}

Thanks to this estimate, we can now establish $C^2$-convergence of $w_\omega$.

\begin{proposition}\label{prop:C^2-convergence}
Consider $(\omega_n)\subset(0,\infty)$ such that $\omega_n\to0$ as $n\to\infty$. Then
\begin{equation*}
w_{\omega_n} \to W \quad \text{in} \quad C^2(\rn), \ \text{as } n\to\infty.
\end{equation*}
\end{proposition}

\begin{proof}
Fix an arbitrary $s>\max\{N,\frac{2N}{N-2}\}$. We will first prove that
\begin{equation}\label{eq:W2s_conv}
\|w_{\omega_n}-W\|_{W^{2,s}(\rn)} \to 0, \quad n\to \infty.
\end{equation}
This convergence follows by interpolation between
\begin{equation}\label{eq:Ls_conv}
\|w_{\omega_n}-W\|_{L^{s}(\rn)} \to 0
\end{equation}
and
\begin{equation}\label{eq:Laplace_Ls_conv}
\|\Delta w_{\omega_n}-\Delta W\|_{L^{s}(\rn)} \to 0.
\end{equation}

To prove \eqref{eq:Ls_conv}, we split $\|w_{\omega_n}-W\|_{L^{s}(\rn)}^s$ into
\[
\|w_{\omega_n}-W\|_{L^{s}(\rn)}^s
= \|w_{\omega_n}-W\|_{L^{s}(B_1(0))}^s
+\|w_{\omega_n}-W\|_{L^{s}(\rn\sm B_1(0))}^s.
\]
On the one hand, thanks to Lemma~\ref{lem:uniform_bound_w}, 
$\|w_{\omega_n}-W\|_{L^{s}(B_1(0))}^s\to 0$ by dominated convergence. On the other, 
$\|w_{\omega_n}-W\|_{L^{s}(\rn\sm B_1(0))}^s\to 0$ by Lemma~\ref{lem:radial_lemma}~(ii).

To prove \eqref{eq:Laplace_Ls_conv}, let $R>0$ and write
\[
\|\Delta w_{\omega_n}- \Delta  W\|_{L^{s}(\rn)}^s
= \|\Delta  w_{\omega_n}- \Delta W\|_{L^{s}(B_R(0))}^s
+\|\Delta w_{\omega_n}-\Delta W\|_{L^{s}(\rn\sm B_R(0))}^s.
\]
Using the elliptic equations \eqref{eq:rescaled_min} and \eqref{eq:lane_emden_fowler}
satisfied by $w_{\omega_n}$ and $W$, respectively, we have
\[
\Delta  w_{\omega_n}- \Delta W = 
-\Big(\lambda_{\omega_n}^{1+\frac{N}{2}}m_{\omega_n} 
f_{\omega_n}(\lambda_{\omega_n}^{-\frac{N-2}{2}}w_{\omega_n})
-m_*W^\frac{N+2}{N-2}\Big).
\]
Since $\omega\lambda_\omega^2\to0$ as $\omega\to0$ by \eqref{eq:upper_bounds}, 
it is easy to show that
\begin{equation}\label{eq:convergence_RHS}
\lambda_{\omega_n}^{1+\frac{N}{2}}m_{\omega_n} f_{\omega_n}(\lambda_{\omega_n}^{-\frac{N-2}{2}}w_{\omega_n})
\to m_*W^\frac{N+2}{N-2} \quad \text{a.e. on } \rn, \quad n\to\infty.
\end{equation}
It follows by dominated convergence and Lemma~\ref{lem:uniform_bound_w} that
$\|\Delta  w_{\omega_n}- \Delta W\|_{L^{s}(B_R(0))}^s\to 0$, for any fixed $R>0$.

As for $\|\Delta w_{\omega_n}-\Delta W\|_{L^{s}(\rn\sm B_R(0))}^s$, we now show that it can be 
made as small as desired by choosing $R>0$ large enough. On the one hand, 
since $W(x)\lesssim |x|^{-\frac{N-2}{2}}$, we deduce that
\begin{equation}\label{eq:Ls>R_1}
\|W\|_{L^{s}(\rn\sm B_R(0))}^s
=\int_{|x|\ge R}W^{\frac{N+2}{N-2}s}\diff x \lesssim R^{N-\frac{N+2}{2}s}. 
\end{equation}
On the other hand, since 
\[
|f_\omega(s)|\le r(s)^\frac{N+2}{N-2}+\omega r(s) 
\le  s^\frac{N+2}{N-2}+\omega s,
\]
we have that
\begin{align*}
&\|\lambda_{\omega_n}^{1+\frac{N}{2}}m_{\omega_n} 
f_{\omega_n}(\lambda_{\omega_n}^{-\frac{N-2}{2}}w_{\omega_n})\|_{L^{s}(\rn\sm B_R(0))}^s \\
&= \int_{|x|\ge R} 
\lambda_{\omega_n}^{(1+\frac{N}{2})s}m_{\omega_n}^s
\big|f_{\omega_n}(\lambda_{\omega_n}^{-\frac{N-2}{2}}w_{\omega_n})\big|^s\diff x \notag \\
&\lesssim
\int_{|x|\ge R} \lambda_{\omega_n}^{\frac{N+2}{2}s}
\big(\lambda_{\omega_n}^{-\frac{N-2}{2}}w_{\omega_n}\big)^{\frac{N+2}{N-2}s}\diff x
+ \int_{|x|\ge R} \lambda_{\omega_n}^{\frac{N+2}{2}s}
\big(\omega_n\lambda_{\omega_n}^{-\frac{N-2}{2}}w_{\omega_n}\big)^{s}\diff x \notag \\
&\lesssim 
\int_{|x|\ge R} w_{\omega_n}^{\frac{N+2}{N-2}s} \diff x
+ \int_{|x|\ge R} (\omega_n \lambda_{\omega_n}^2)^s w_{\omega_n}^s \diff x.
\end{align*}
Furthermore, since $\{w_{\omega_n}\}$ is bounded in $L^{2^*}(\rn)$, it follows by 
Lemma~\ref{lem:radial_lemma}~(i) that there exists a constant $C_N$ such that
\[
\forall n\in\N, \quad |w_{\omega_n}(x)|\le C_N |x|^{-\frac{N-2}{2}}.
\]
Thus, using again $\omega\lambda_\omega^2\to0$ as $\omega\to0$, we find that
\begin{align}\label{eq:Ls>R_2}
\|\lambda_{\omega_n}^{1+\frac{N}{2}}m_{\omega_n} 
f_{\omega_n}(\lambda_{\omega_n}^{-\frac{N-2}{2}}w_{\omega_n})\|_{L^{s}(\rn\sm B_R(0))}^s
& \lesssim 
\int_{|x|\ge R} |x|^{-\frac{N+2}{2}s}\diff x + \int_{|x|\ge R} |x|^{-\frac{N-2}{2}s}\diff x \notag \\
&\lesssim
R^{N-\frac{N+2}{2}s}+R^{N-\frac{N-2}{2}s}.
\end{align}
Since $s>\frac{2N}{N-2}$, it follows from \eqref{eq:Ls>R_1} and \eqref{eq:Ls>R_2} that,
given any $\eps>0$, we can choose $R>0$ so large that
\[
\|\Delta w_{\omega_n}-\Delta W\|_{L^{s}(\rn\sm B_R(0))}^s<\eps.
\]
This completes the proof of \eqref{eq:Laplace_Ls_conv} and thus of \eqref{eq:W2s_conv}.

Since $s>N$, we deduce that $w_{\omega_n} \to W$ in $C^1(\rn)$. We now bootstrap this to 
$w_{\omega_n} \to W$ in $C^2(\rn)$ by using the ODE's satisfied by 
the radial functions $w_{\omega_n}$ and $W$:
\begin{equation}\label{eq:ODE_w_omega}
-w_{\omega_n}''-\frac{N-1}{r}w_{\omega_n}' =
\lambda_{\omega_n}^{1+\frac{N}{2}}m_{\omega_n} f_{\omega_n}(\lambda_{\omega_n}^{-\frac{N-2}{2}}w_{\omega_n}),
\end{equation}
\begin{equation}\label{eq:ODE_W} 
-W''-\frac{N-1}{r}W' = m_* W^\frac{N+2}{N-2},
\end{equation}
where $'$ denotes differentiation with respect to $r\in(0,\infty)$. We use the same notation 
for $w_{\omega_n}$, $W$ and their radial counterparts, \emph{i.e.}~$w_{\omega_n}(r)$, $W(r)$, with $r=|x|$.

We first note that, using $\|w_{\omega_n} - W\|_{L^\infty}\to 0$, the pointwise convergence in
\eqref{eq:convergence_RHS} can be improved to 
\begin{equation}\label{eq:uniform_convergence_RHS}
\Big\|\lambda_{\omega_n}^{1+\frac{N}{2}}m_{\omega_n} f_{\omega_n}(\lambda_{\omega_n}^{-\frac{N-2}{2}}w_{\omega_n}) - m_*W^\frac{N+2}{N-2}\Big\|_{L^\infty} \to 0,
\quad n\to\infty.
\end{equation}
Subtracting \eqref{eq:ODE_w_omega} from \eqref{eq:ODE_W}, multiplying by $r^{N-1}$ and integrating
yields
\[
r^{N-1}\left(w_{\omega_n}'(r)-W'(r)\right)=-\int_0^r s^{N-1}
\left(\lambda_{\omega_n}^{1+\frac{N}{2}}m_{\omega_n} f_{\omega_n}(\lambda_{\omega_n}^{-\frac{N-2}{2}}w_{\omega_n}(s)) - m_*W^\frac{N+2}{N-2}(s)\right)\diff s.
\]
With the change of variables $s=rt$, this identity becomes
\[
\frac{w_{\omega_n}'(r)-W'(r)}{r}=-\int_0^1 t^{N-1}
\left(\lambda_{\omega_n}^{1+\frac{N}{2}}m_{\omega_n} f_{\omega_n}(\lambda_{\omega_n}^{-\frac{N-2}{2}}w_{\omega_n}(rt)) - m_*W^\frac{N+2}{N-2}(rt)\right)\diff t
\]
and it follows that
\[
\left|\frac{w_{\omega_n}'(r)-W'(r)}{r}\right|
\le
\Big\|\lambda_{\omega_n}^{1+\frac{N}{2}}m_{\omega_n} f_{\omega_n}(\lambda_{\omega_n}^{-\frac{N-2}{2}}w_{\omega_n}) - m_*W^\frac{N+2}{N-2}\Big\|_{L^\infty}
\int_0^1 t^{N-1}\diff t.
\]
Hence, by \eqref{eq:uniform_convergence_RHS},
\begin{equation}\label{eq:uniform_convergence_w'_over_r}
\left\|\frac{w_{\omega_n}'(r)-W'(r)}{r}\right\|_{L^\infty}\to 0, \quad n\to\infty.
\end{equation}
It then follows from \eqref{eq:ODE_w_omega}, \eqref{eq:ODE_W} and \eqref{eq:uniform_convergence_w'_over_r} that
\[
\left\|w_{\omega_n}''(r)-W''(r)\right\|_{L^\infty}\to 0, \quad n\to\infty.
\]
Since we already know that $w_{\omega_n} \to W$ in $C^1(\rn)$, this completes the proof.
\end{proof}




\subsubsection{Asymptotic behavior of $M(\omega)=\|u_\omega\|^2_{L^2}$ and $M'(\omega)$ as $\omega\to 0$.}

As a consequence of Lemma~\ref{lem:u_omega_L^2_w_omega_L^2} and the exact asymptotic behavior of $\lambda_\omega$, 
we deduce the following proposition which gives the behavior of $M(\omega)$ as $\omega$ goes to $0$.

\begin{proposition}
    \label{prop:asymptotic_mass_critical} Let $N\ge 3$ and $p=\frac{N+2}{N-2}$. As $\omega \to 0$, we have    
    \begin{equation*}
    \lim_{\omega\to 0}\|u_\omega\|^2_{L^2}=+\infty.
    \end{equation*}
    More precisely, 
    \begin{equation}\label{eq:exact_asymptotic_mass_critical}
        \begin{cases}
        \omega^{-\frac{3}{4}}\lesssim M(\omega)\lesssim \omega^{-\frac{3}{4}} &\text{if } N=3,\\
        \left(\frac{1}{\omega}\log\frac{1}{\omega}\right)^{\frac12}\lesssim M(\omega)\lesssim \left(\frac{1}{\omega}\log\frac{1}{\omega}\right)^{\frac12} &\text{if } N=4,\\
       \omega^{-\frac{2}{N}} \lesssim M(\omega)\lesssim \omega^{-\frac{2}{N}} &\text{if } N\ge 5.
        \end{cases}
    \end{equation}
\end{proposition}

\begin{proof}
    From Lemma~\ref{lem:u_omega_L^2_w_omega_L^2}, we know that, as $\omega\to 0$, 
    \begin{equation*}
        \lambda_\omega^2\|w_\omega\|_{L^2}^2
\gtrsim
M(\omega) 
\gtrsim
\lambda_\omega^2\|w_\omega\|_{L^2}^2.
    \end{equation*}
    Moreover, using~\eqref{eq:upper_bound_est_1} and~\eqref{eq:upper_bound_est_2} for $N\ge 5$ or \eqref{eq:L^2_lower_bound_w} for $N\in \{3,4\}$, together with the exact behavior of $\lambda_\omega$, we deduce that
    \begin{equation*}
        \begin{cases}
        \omega^{-\frac{1}{4}}\lesssim \|w_\omega\|_{L^2}^2\lesssim \omega^{-\frac{1}{4}} &\text{if } N=3,\\
        \log\frac{1}{\omega}\lesssim \|w_\omega\|_{L^2}^2\lesssim \log\frac{1}{\omega} &\text{if } N=4,\\
       1 \lesssim \|w_\omega\|_{L^2}^2\lesssim 1 &\text{if } N\ge 5.\\
        \end{cases}
    \end{equation*}
    Combining these estimates with \eqref{eq:lower-upper_bounds} completes the proof.
\end{proof}


To deduce the asymptotic behavior of $M'(\omega)$, we proceed as in Proposition~\ref{prop:upper_bound_M'}.

\begin{proposition}
    \label{prop:upper_bound_M'_critical} Let $N\ge 3$ and $p=\frac{N+2}{N-2}$. Then,  for $\omega>0$ small enough, we have
    \begin{align}
        \label{eq:upper_bound_M'_critical}
        \frac{M'(\omega)}{2}&\,\left[2\beta(\omega)+(N-2)\right]< \frac{1}{4}\frac{M(\omega)}{\omega}\left[N(N-2)(\beta(\omega)-1) -4\right]
    \end{align}
   where 
    \begin{equation*}
        T(\omega)= \int_{\R^N} |\nabla u_\omega(x)|^2\diff x,\quad \beta(\omega)=T(\omega)^{-1}\int_{\R^N}u_\omega(x)^{2^*}\diff x.
    \end{equation*}
    Moreover, 
            \begin{equation*}
                \lim_{\omega\to 0^+} M'(\omega)=-\infty.
            \end{equation*}
\end{proposition}

\begin{proof}
    As in the proof of Proposition~\ref{prop:upper_bound_M'}, for any $\omega>0$, let $\mathcal L(\omega)=L_+$ be defined by~\eqref{eq:defLplus} and
    $L=(L_{ij})$ be the symmetric matrix given by the restriction of $\mathcal L(\omega)$ to the finite dimensional space spanned by $\left\{\partial_\omega u_\omega, u_\omega, x\cdot\nabla u_\omega+\frac{N}{2}u_\omega\right\}$. 

    When  $p=\frac{N+2}{N-2}$, the same arguments detailed above give
    \begin{align*}
        L_{11}:=&\,\langle \partial_\omega u_\omega, \mathcal L(\omega)\partial_\omega u_\omega\rangle=-\frac{M'(\omega)}{2},\quad L_{12}:=\langle \partial_\omega u_\omega, \mathcal L(\omega) u_\omega\rangle=-M(\omega),\\
        L_{13}:=&\,\langle \partial_\omega u_\omega ,\mathcal L(\omega)\left(x\cdot\nabla u_\omega+\frac{N}{2}u_\omega\right)\rangle=-\langle u_\omega ,x\cdot\nabla u_\omega+\frac{N}{2}u_\omega\rangle=0,\\ 
        L_{22}:=&\,\langle  u_\omega, \mathcal L(\omega) u_\omega\rangle= 8\delta Q(\omega)-\frac{4}{N-2}\beta(\omega)T(\omega),\\
        L_{23}:=&\,\langle  u_\omega,\mathcal L (\omega)\left(x\cdot\nabla u_\omega+\frac{N}{2}u_\omega\right)\rangle=4\delta N Q(\omega)-\frac{4}{N-2}\beta(\omega)T(\omega)-2\omega M(\omega),\\
        L_{33}:=&\,\langle  x\cdot\nabla u_\omega+\frac{N}{2}u_\omega,\mathcal L (\omega)\left(x\cdot\nabla u_\omega+\frac{N}{2}u_\omega\right)\rangle=2\delta\frac{N(N+2)}{2}Q(\omega)-\frac{4}{N-2}\beta(\omega)T(\omega),
    \end{align*}
    with $Q(\omega)=\int_{\R^N}u_\omega(x)^2 |\nabla u_\omega(x)|^2\diff x$.
 
    In the critical case, the integral identities of Proposition~\ref{prop:integral_identities} give
    \begin{equation}
        \label{eq:QomegaBetaomega_critical}
        \begin{cases}
            &2\delta Q(\omega)=\frac{2}{N-2}\omega M(\omega),\\
            &\beta(\omega)=1+\frac{N+2}{N-2}\frac{\omega M(\omega)}{T(\omega)}.
        \end{cases}
    \end{equation}
    As a consequence, 
    \begin{align*}
        L_{22}=&\,\frac{4}{N-2}T(\omega)\left[2 \frac{\omega M(\omega)}{T(\omega)} -\beta(\omega)\right], \\
        L_{23}=&\, \frac{2}{N-2}T(\omega)\left[(N+2)\frac{\omega M(\omega)}{T(\omega)} -2\beta(\omega)\right],\\
        L_{33}=&\, \frac{1}{N-2}T(\omega)\left[N(N+2) \frac{\omega M(\omega)}{T(\omega)} -4\beta(\omega)\right],
    \end{align*}
    and
    \begin{align*}
        L_{22}L_{33}&-L^2_{23}=\frac{4}{N-2}\omega M(\omega)T(\omega)\left[(N+2)\frac{\omega M(\omega)}{T(\omega)}-N\beta(\omega)\right]\\
        &=\frac{4}{N-2}\omega M(\omega)T(\omega)\left[-2\beta(\omega)-(N-2)\right]<0
    \end{align*}
    for all $\omega>0$.

    Since $\mathcal L(\omega)$ has a unique negative eigenvalue, we deduce that the determinant of $L$ is negative:
    \begin{align*}
        0>\det(L)=\frac{M'(\omega)}{2}(L^2_{23}-L_{22}L_{33})-\frac{1}{N-2}M(\omega)^2 T(\omega)\left[N(N-2)(\beta(\omega)-1) -4\right].
    \end{align*}
    This gives the estimate~\eqref{eq:upper_bound_M'_critical}.

    In the limit $\omega \to 0$, $\omega M(\omega)$ goes $0$ which implies $\lim_{\omega \to 0}\beta(\omega)=1$. As a consequence, 
    \begin{align*}
        \left[N(N-2)(\beta(\omega)-1) -4\right] \underset{\omega \to 0}{\sim} -4
    \end{align*}
    and, for $\omega>0$ small enough,
    \begin{equation*}
        M'(\omega)<-c\frac{M(\omega)}{\omega}
    \end{equation*}
    for some positive constant $c$. This implies, 
    \begin{equation*}
        \lim_{\omega\to 0^+} M'(\omega)=-\infty.
    \end{equation*}
\end{proof}

Collecting the main results of Section~\ref{sec:critical}, we can now complete the proof of 
Theorem~\ref{thm:asympt_omega_to_0}.

\begin{proof}[Proof of Theorem~\ref{thm:asympt_omega_to_0}~(ii)]
The proof follows from Propositions \ref{prop:conv_of_vn'}, \ref{prop:C^2-convergence}, 
\ref{prop:asymptotic_mass_critical} and \ref{prop:upper_bound_M'_critical}.
\end{proof}




\section{Appendix}

\begin{proof}[Proof of Lemma~\ref{lem:radial_lemma}]
For the whole proof, we denote by $\wt u(r), \wt v(r)$ etc., the functions 
on $\real_+$ defined by $\wt u(r)=u(x), \wt v(r)=v(x)$, etc., for $r=|x|$.
Any radial $u\in \dot H^1(\rn)$ has a representative such that $\wt u$ 
is continuous on $\real_+$, and $\wt u' \in L^2(\real_+;r^{N-1}\diff r)$
is related to $\nabla u$ by
$$
\|\nabla u\|_{L^2}^2=|\mathbb{S}^{N-1}|\int_{\real_+}|\wt u'(r)|^2r^{N-1}\diff r.
$$

Part (i) follows from the estimate
$$
\|u\|_{L^s}^s 
\ge |\mathbb{S}^{N-1}| \int_0^r|\wt u(t)|^st^{N-1}\diff t
\ge |\mathbb{S}^{N-1}| |\wt u(r)|^s\frac{r^N}{N}, \quad \forall r>0.
$$

To prove (ii), we fix an arbitrary $R>0$. 
We first address the convergence in $L^q(\rn\sm B_R(0))$. Since $\|u_n\|_{\dot H^1}$
is bounded, there exists $u\in \dot H^1(\rn)$ such that, up to a subsequence, 
$u_n$ converges to $u$, weakly in $\dot H^1(\rn)$ and a.e.~on $\rn$, as $n\to\infty$. 
Applying \eqref{eq:radial_estimate} with $s=\frac{2N}{N-2}$ and the Sobolev embedding
theorem, there is a constant $C_N>0$ (independent of $n$) such that
$$
|u_n(x)-u(x)|\le C_N\|\nabla u_n-\nabla u\|_{L^2}|x|^{-\frac{N-2}{2}}.
$$
Hence, for $n$ large,
$$
|u_n(x)-u(x)|^q\lesssim \|\nabla u\|_{L^2}^q |x|^{-\frac{N-2}{2}q}
\in L^1(\rn\sm B_R(0)), \quad \forall q>2^*.
$$
Dominated convergence then implies $u_n\to u$ in $L^q(\rn\sm B_R(0))$.

To prove convergence in $L^\infty(\rn\sm B_R(0))$, we let $v_n=u_n-u$
and we show that
$$
\lim_{n\to\infty}\sup_{r\ge R}|\wt v_n(r)|=0.
$$
We first apply \eqref{eq:radial_estimate} again with $s=\frac{2N}{N-2}$, 
using the fact that $\|\nabla v_n\|_{L^2}$ is bounded. Given $\eps>0$, there
exist $R_\eps>R$ and a constant $C>0$ such that,
$$
\forall n\in\N, \quad
\sup_{r>R_\eps}|\wt v_n(r)|
=\sup_{|x|>R_\eps}|v_n(x)|
\le C R_\eps^{-\frac{N-2}{2}}
<\eps.
$$
Hence, we only need to show that $\wt v_n(r)\to 0$ as $n\to\infty$, 
uniformly for $r\in[R,R_\eps]$. By an Arzela-Ascoli type argument which will be
made explicit below, this is a consequence of the equicontinuity 
of the sequence $\{\wt v_n\}$ on $[R,R_\eps]$, which we prove now.

Let $s,t \ge R, \ s<t$. Since
$\{v_n\}$ is bounded in $\dot H^1$, there exists $M>0$
such that,
\begin{align}
\forall n\in\N, \quad
|\wt v_n(t) - \wt v_n(s)|
&\le \int_s^t |\wt v_n'(\sigma)|\diff\sigma \\
&\le (t-s)^{1/2}\Big(\int_s^t |\wt v_n'(\sigma)|^2\diff\sigma\Big)^{1/2} \\
&\le M (t-s)^{1/2}.
\end{align}
Hence, $\{\wt v_n\}$ is equicontinuous on $[R,\infty)$. 

We now claim that $\wt v_n\to 0$ as $n\to\infty$ for all $r\in [R,R_\eps]$.
We will prove below that the convergence is uniform. 
To complete the proof, the claim can be proved by similar arguments,
using convergence almost everywhere and equicontinuity of $\{\wt v_n\}$ on $[R,R_\eps]$.

Suppose by contradiction that $\{\wt v_n\}$ does not converge uniformly
to $0$ on $[R,R_\eps]$: there exists $\eps>0$ such that, for all $n\in\N$
there exists $r_n\in [R,R_\eps]$ such that 
$$
|\wt v_n(r_n)|\ge \eps. 
$$
There exists a subsequence $(r_{n_j})\subset [R,R_\eps]$ 
and a point $r^*\in [R,R_\eps]$ such that $r_{n_j}\to r^*$ as $j\to\infty$.
Suppose without loss of generality that $r^*\in (R,R_\eps)$.
By equicontinuity, there exists $\delta>0$ such that 
$$
\forall n\in\N, \
\forall r \in (r^*-\delta,r^*+\delta)\subset [R,R_\eps], \quad
|\wt v_n(r)-\wt v_n(r^*)|<\frac{\eps}{2}.
$$
Now choose $N_\eps\in \N$ such that $r_{n_j}\in (r^*-\delta,r^*+\delta)$
for all $j\ge N_\eps$. It follows that,
$$
\forall j\ge N_\eps, \quad
|\wt v_{n_j}(r^*)|
\ge ||\wt v_{n_j}(r_{n_j})|-|\wt v_{n_j}(r_{n_j})-\wt v_{n_j}(r^*)||
\ge \frac{\eps}{2}.
$$
This contradicts the pointwise convergence $\wt v_{n_j}(r^*)\to0$
and finishes the proof.
\end{proof}

\begin{proof}[Proof of Lemma~\ref{lem:uniform_bound_z}]
We start by proving part (ii), which will be used in the proof of part (i).
To prove estimate \eqref{eq:uniform_Ls_bounds}, we follow the scheme of proof laid down in \cite[Lemma~5.5]{adachi-watanabe_2014}.
Let 
$$
q_0=2^*=\frac{2N}{N-2} \quad \text{and} \quad q=\frac{p-1}{2}.
$$
Since $p<\frac{3N+2}{N-2}$, there holds
$$
q_0-q>1.
$$
Multiplying both sides of \eqref{eq:moser} by $z_\omega^{q_0-q}$ and integrating by parts,
one has
\begin{equation}\label{eq:moser_int_eq}
\intrn \nabla z_\omega\cdot\nabla(z_\omega^{q_0-q})\diff x 
= m_\omega \intrn f_\omega(z_\omega) z_\omega^{q_0-q} \diff x.
\end{equation}
Since $\|\nabla z_\omega\|_{L^2}^2=m_\omega$, it follows by 
Lemma~\ref{lem:omega-asymptotics}/\ref{lem:omega-asymptotics_critical} and the Sobolev embedding theorem 
that there exists $\eta>0$ and $M>0$ such that 
\begin{equation*}
\forall \omega\in(0,\eta), \quad \|z_\omega\|_{L^{q_0}}\le M.
\end{equation*}
Hence, using \eqref{eq:upper-bound_on_f} and 
Lemma~\ref{lem:omega-asymptotics}/\ref{lem:omega-asymptotics_critical}, there is a constant $C$
such that, for all $\omega\in(0,\eta)$,
\begin{equation}\label{eq:moser_estimate_0}
m_\omega \intrn f_\omega(z_\omega) z_\omega^{q_0-q} \diff x
\le m_\omega C_0\intrn z_\omega^q z_\omega^{q_0-q} \diff x \le C \intrn z_\omega^{q_0} \diff x 
\le CM^{q_0},
\end{equation}
On the other hand,
\begin{equation*}
\intrn \nabla z_\omega\cdot\nabla(z_\omega^{q_0-q})\diff x=(q_0-q)\intrn z_\omega^{q_0-q-1}|\nabla z_\omega|^2\diff x.
\end{equation*}
Hence, by \eqref{eq:moser_int_eq} and \eqref{eq:moser_estimate_0},
\begin{equation}\label{eq:moser_estimate_1}
(q_0-q)\intrn z_\omega^{q_0-q-1}|\nabla z_\omega|^2\diff x \le CM^{q_0}.
\end{equation}
Let
$$
q_1=\frac{q_0}{2}(q_0-q+1)>q_0.
$$
A direct calculation shows that
\begin{equation*}
(q_0-q)\intrn z_\omega^{q_0-q-1}|\nabla z_\omega|^2\diff x=
\frac{4(q_0-q)}{(q_0-q+1)^2}\intrn \Big|\nabla\Big(z_\omega^\frac{q_0-q+1}{2}\Big)\Big|^2\diff x.
\end{equation*}
By the Sobolev embedding, there is a constant $C_S$ such that
$$
\Big\|z_\omega^\frac{q_0-q+1}{2}\Big\|_{L^{q_0}}
\le C_S\Big\|\nabla \Big(z_\omega^\frac{q_0-q+1}{2}\Big)\Big\|_{L^2}.
$$
Therefore,
$$
(q_0-q)\intrn z_\omega^{q_0-q-1}|\nabla z_\omega|^2\diff x
\ge\frac{4(q_0-q)}{(q_0-q+1)^2}\frac{1}{C_S^2}
\Big(\intrn z_\omega^{\frac{q_0}{2}(q_0-q+1)}\diff x\Big)^\frac{2}{q_0}.
$$
By \eqref{eq:moser_estimate_1}, there exists a constant $C_1>0$ such that
\begin{equation*}
\forall \omega\in(0,\eta), \quad \intrn z_\omega^{q_1}\diff x\le C_1.
\end{equation*}
The proof now follows by iteration. Let 
$$
q_i=\frac{q_0}{2}(q_{i-1}-q+1)>q_{i-1}, \quad i\in\N^*.
$$
For each $i\in\N^*$, multiplying both sides of \eqref{eq:moser} by $z_\omega^{q_{i-1}-q}$ yields 
\begin{equation*}\label{eq:moser_estimate_i}
(q_{i-1}-q)\intrn z_\omega^{q_{i-1}-q-1}|\nabla z_\omega|^2\diff x \le CM^{q_{i-1}}
\end{equation*}
instead of $\eqref{eq:moser_estimate_1}$. Next, the above argument shows that
$$
\Big(\intrn z_\omega^{q_i}\diff x\Big)^\frac{2}{q_0}
\lesssim \intrn z_\omega^{q_{i-1}-q-1}|\nabla z_\omega|^2\diff x,
$$
and we conclude that there exists a constant $C_i>0$ such that
\begin{equation*}
\forall \omega\in(0,\eta), \quad \intrn z_\omega^{q_i}\diff x\le C_i.
\end{equation*}
Finally, it is an easy exercise to show that $q_i\to\infty$ as $i\to\infty$.
The conclusion thus follows by H\"older interpolation.

We next prove part (i). 
First recall that, for any $\omega>0$, 
$z_\omega$ is positive, radial decreasing, and satisfies the elliptic equation
\begin{equation}\label{eq:equ_for_zn}
-\Delta z_\omega=c_\omega(x) z_\omega,
\end{equation}
where 
$$
c_\omega(x):=\frac{m_\omega f_\omega(z_\omega(x))}{z_\omega(x)},
\quad x\in\rn.
$$
By Lemma~\ref{lem:omega-asymptotics_critical}, there exist $\eta_0>0$ and $C>0$ such that
\begin{align*}
\forall \omega\in(0,\eta_0), \quad
|c_\omega(x)|\le C\big(r(z_\omega(x))^{p-1}+\omega\big)
\le C\big(z_\omega(x)^{\frac{4}{N-2}}+\omega\big).
\end{align*}
Furthermore, by part (ii) proved above,
for any $s\ge 2^*$ there exists a constant $C_s>0$ such that 
\begin{equation}\label{eq:moser_again}
\forall \omega\in(0,\eta_0), \quad
\|z_\omega\|_{L^s(\rn)}\le C_s.
\end{equation}
Hence, by the Radial Lemma, taking $C_s$ larger if needed,
$$
\forall \omega\in(0,\eta_0), \quad
z_\omega(x)\le C_s|x|^{-N/s}.
$$
Thus,
$$
z_\omega(x)^{\frac{4}{N-2}}=O\big(|x|^{-\frac{4N}{s(N-2)}}\big),
$$
uniformly for $\omega\in(0,\eta_0)$.
Choosing $s>\frac{2N}{N-2}$, it is possible to find 
$q\in(\frac{N}{2},s\frac{N-2}{4})$, so that
$N-\frac{q}{s}\frac{4N}{N-2}>0$
and hence $\|c_\omega\|_{L^q(B_1(0))}$ is bounded, for $\omega\in(0,\eta_0)$.
Then, by Theorem~5.1 and Remark~5.1 in \cite{stampacchia},
there is a constant $K>0$, independent of $\omega$, such that
\begin{equation}
\forall \omega\in(0,\eta_0), \quad
\|z_\omega\|_{L^\infty(B_{1/2}(0))}\le K \|z_\omega\|_{L^2(B_1(0))}^2.
\end{equation}
Since $z_\omega$ converges in $L^2(B_1(0))$ as $\omega\to 0$, it 
follows that $\|z_\omega\|_{L^\infty(B_{1/2}(0))}$ is bounded,
for $\omega\in(0,\eta_0)$. Finally, it follows by part (ii)
of Lemma~\ref{lem:radial_lemma} that 
$\|z_\omega\|_{L^\infty(\rn\sm B_{1/2}(0))}$ is also bounded, which
completes the proof.
\end{proof}


\bibliographystyle{siam}
\bibliography{quasiNLS}

\end{document}